\documentclass[12pt,a4paper,reqno]{amsart}
\usepackage[english]{babel}
\usepackage[applemac]{inputenc}
\usepackage[T1]{fontenc}
\usepackage{enumitem} 
\usepackage{fourier}
\usepackage{amsmath}
\usepackage{amssymb}
\usepackage{amsthm}
\usepackage{mathtools}
\usepackage{amsfonts}
\usepackage{comment}
\usepackage{tikz-cd}
\usepackage{graphicx}

\usepackage{amsrefs}
\usepackage[colorlinks = true, citecolor = blue]{hyperref}
\title{Porosity in Conformal Dynamical Systems}
\address{University of Connecticut, Department of Mathematics}
\address{University of North Texas, Department of Mathematics}
\author{Vasileios Chousionis}
\author{Mariusz Urba\'nski}
\thanks{V.C. was supported in part by the Simons Foundation Collaboration grant no.\  521845. M.U. was supported in part by the Simons Foundation Collaboration grant no.\  581668. 
}
\email{vasileios.chousionis@uconn.edu}
\email{urbanski@unt.edu}

\newcommand{\ve}{\varepsilon}

\newcommand{\f}{\phi}
\newcommand{\sg}{\sigma}
\newcommand{\R}{\mathbb{R}}
\newcommand{\Rn}{\mathbb{R}^n}

\newcommand{\N}{\mathbb{N}}

\newcommand{\C}{\mathbb{C}}
\newcommand{\Z}{\mathbb{Z}}

\renewcommand{\a}{\alpha}
\newcommand{\om}{{\omega}}

\newcommand{\cS}{\mathcal{S}}

\newcommand{\spt}{\operatorname{spt}}

\newcommand{\cH}{\mathcal{H}}

\newcommand{\G}{\mathbb{G}}
\newcommand{\cC}{\mathcal{C}}
\newcommand{\cf}{\mathcal{CF}}

\newcommand{\diam}{\operatorname{diam}}

\newcommand{\dist}{\operatorname{dist}}

\newcommand{\bd}{\operatorname{\overline{BD}}}
\newcommand{\Rea}{\operatorname{Re}}
\renewcommand{\f}{\phi}

\newcommand{\por}{\operatorname{por}}
\newcommand{\Fin}{\operatorname{Fin}}

\DeclareMathOperator{\Int}{Int}
\DeclareMathOperator{\Img}{Im}
\DeclareMathOperator{\Crit}{Crit}
\DeclarePairedDelimiter\floor{\lfloor}{\rfloor}
\def\om{\omega}
\def\du{\bigoplus}
\def\fg{{\mathfrak g}}
\def\fv{{\mathfrak v}}
\newcommand{\Heis}{{{\mathbb{H}}}}

\newcommand{\stm}{\setminus}
\newcommand{\ra}{\rightarrow}

\newcommand{\oc}{\hat{\C}}
\newcommand{\Sing}{{\rm Sing}}
\newcommand{\es}{\emptyset}
\newcommand{\bu}{\bigcup}

\newcommand{\sms}{\setminus}
\newcommand{\PS}{{\rm PS}}
\newcommand{\ov}{\overline}
\newcommand{\sbt}{\subset}
\newcommand{\fr}{\noindent}

\numberwithin{equation}{section}

\newtheorem{thm}{Theorem}[section]
\newtheorem{theorem}{Theorem}[section]

\newtheorem{lm}[thm]{Lemma}
\newtheorem{coro}[thm]{Corollary}
\theoremstyle{definition}

\newtheorem{propo}[thm]{Proposition}
\theoremstyle{definition}
\newtheorem{defn}[thm]{Definition}
\theoremstyle{definition}
\newtheorem{definition}[thm]{Definition}
\theoremstyle{definition}
\newtheorem{rem}[thm]{Remark}
\newtheorem{remark}[thm]{Remark}
\newtheorem{claim}{Claim}

\begin{document}
\begin{abstract} In this paper we study various aspects of porosities for conformal fractals. We first explore porosity in the general context of infinite graph directed Markov systems (GDMS), and we show that, under some natural assumptions, their limit sets are porous in large (in the sense of category and dimension) subsets, and they are mean porous almost everywhere. On the other hand, we prove that if the limit set of a GDMS is not porous then it is not porous almost everywhere. We also revisit porosity for finite graph directed Markov systems, and we provide checkable criteria which guarantee that limit sets have holes of relative size at every scale in a prescribed direction. 

We then narrow our focus to systems associated to complex continued fractions with arbitrary alphabet and we provide a novel characterization of porosity for their limit sets. Moreover, we introduce the notions of upper density and upper box dimension for subsets of Gaussian integers and we explore their connections to porosity. As applications we show that limit sets of complex continued fractions system whose alphabet is co-finite, or even a co-finite subset of the Gaussian primes, are not porous almost everywhere, while they are mean porous almost everywhere. 

We finally turn our attention to complex dynamics and we delve into porosity for Julia sets of meromorphic functions. We show that if the Julia set of a tame meromorphic function is not the whole complex plane then it is porous at a dense set of its points and it is almost everywhere mean porous with respect to natural ergodic measures. On the other hand, if the Julia set is not porous then it is not porous almost everywhere. In particular, if the function is elliptic we show that its Julia set is not porous at a dense set of its points.

\end{abstract}

\maketitle
\tableofcontents

\section{Introduction}
Let $(X,d)$ be a metric space. A set $E \subset X$ is called  {\em porous}  if there exists a positive constant $c>0$ such that every open ball centered at $E$ with radius $r \in (0, \diam(E))$ contains an open ball of radius $cr$, which does not intersect $E$. If this condition is satisfied for balls centered at a fixed point $x$, then the set $E$ is called {\em porous at} $x$.
If $(X,d)$ is a Euclidean space (or even  a separable space equipped with a doubling Borel measure), then the Lebesgue density theorem easily implies that porous sets have zero Lebesgue measure. Therefore, one could obtain quantitative information about the size and structure of singular sets by investigating  how ``porous" they are. For example one could explore if holes of relative size appear at all or at just a fixed percentage of scales, or how the holes are locally spread around. 

Indeed, various aspects of porosities have been introduced over the years in order to quantify the size and structure of exceptional sets in different contexts. A notion of porosity already appeared the work of Denjoy \cite{De} on trigonometric series in the 1920s. Since then, porosities have been studied widely, for example, in connection to geometric measure theory \cite{salli,mat,smibel, smibeletal,shmemp,kasu1,kasu2,ejj,za}, geometric function theory \cite{mavu, do, vai, kore, niem},  differentiability of Lipschitz maps \cite{preissza,preissza2,lpt,speight}, harmonic analysis \cite{dirpor, bruna, marpor}, fractal geometry and complex dynamics \cite{dirpor, urbpor, jjm, prro}.

In this paper we explore porosity for conformal dynamical systems. First, we perform a comprehensive study of various porosities in the context of \textit{conformal graph directed Markov systems (GDMS)}, and we pay special attention to systems generated by complex continued fractions. Our results apply to a very broad family of fractals, as the general framework of conformal GDMS encompasses a wide selection of geometric objects, including limit sets of Kleinian and complex hyperbolic Schottky groups, Apollonian circle packings, self-conformal and self-similar sets. We then  turn our attention to complex dynamics and we study various porosities for Julia sets of meromorphic functions. 

Graph directed systems with a finite alphabet consisting of similarities were introduced by Mauldin and Williams in \cite{MW}, see also \cite{edm}. Mauldin and the second named author developed an extensive theory of conformal GDMS with a countable alphabet in \cite{MUGDMS} stemming from \cite{MU1}. In the recent monograph \cite{CTU}, the authors  together with Tyson extended the theory of conformal GDMS in the setting of nilpotent stratified Lie groups (Carnot groups) equipped with a sub-Riemannian metric. We also refer to \cite{CLU, kes1, kes2, roy, polur} for recent advances on various aspects of GDMS.

We defer the formal definition of a GDMS to Section \ref{sec:prelim}, and we now only give a short heuristic description. A conformal GDMS $\cS$ is modeled on a directed multigraph $(E,V)$, where $E$ is a countable set of edges and $V$ is a finite set of vertices. Each vertex $v \in V$ corresponds to a compact set $X_v$ and each edge $e \in E$, which connects the vertices $t(e), i(e)$,  corresponds to a contracting conformal map $\f_{e}: X_{t(v)} \ra X_{i(v)}$. An incidence matrix $A:E \times E \ra \{0,1\}$ determines if a pair of these maps is allowed to be composed. The \textit{limit set} of $\cS$ is denoted by $J_\cS$, and is defined as the image of a natural projection from the symbol space of admissible words to $X:=\cup_{v\in V} X_v$.

If $\cS$ is a finite and irreducible conformal GDMS (see Section \ref{sec:prelim} for the exact definitions) whose limit set $J_{\cS}$ has zero Lebesgue measure then it is porous, see e.g. \cite[Theorem 4.6.4]{MUGDMS},  \cite[Theorem 2.5]{urbpor} or \cite[Theorem 2.6]{jjm}. Nevertheless, if the system $\cS$ is infinite the situation is very different.  As we shall see in the following, there are many examples of conformal GDMS whose limit sets have Lebesgue measure zero but they are not porous; for example the limit set associated to complex continued fractions. 

Although limit sets of finitely irreducible conformal GDMS  are very often not porous, we will prove that they are \textit{always} porous in large (in the sense of category and dimension) subsets. We record that \eqref{xx1} in the following theorem is a mild non-degeneracy condition which ensures that the limit set has Lebesgue measure zero.
\begin{thm}
\label{thmporfixintro}
Let $\cS=\{\f_e\}_{e \in E}$ be a finitely irreducible conformal GDMS which satisfies \eqref{xx1}. Then the following hold:
\begin{enumerate}[label=(\roman*)]
\item \label{porthm2} The limit set $J_{\cS}$ is porous at every fixed point of $\cS$, in particular $J_{\cS}$ is porous at a dense set of $J_{\cS}$.
\item \label{porthm3a} For every $\ve>0$ there exists some finite $F(\ve):=F \subset E$  with  
$$
\dim_{\cH} (J_{\cS_F})>\dim_{\cH} (J_{\cS})-\ve,
$$
such that $J_{\cS}$ is porous at every $x \in J_{\cS_F}$ with porosity constant only depending on $F$ and $\cS$.
\item \label{porthm3b} There exists some set $\tilde{J} \subset J_{\cS}$ such that $\dim_{\cH} (\tilde{J})=\dim_{\cH} (J_{\cS})$ and $J_{\cS}$ is porous at every $x \in \tilde{J}$.
\end{enumerate}
\end{thm}

Moreover, we show that under some natural assumptions limit sets of conformal GDMSs are porous on a fixed percentage of scales. This behavior is quantified by the notion of \textit{mean porosity} whose formal definition is deferred to Section \ref{sec:mean}. Koskela and Rohde introduced mean porosity in \cite{kore} in connection with dimension estimates and quasiconformal mappings, and it was further investigated by several authors, ex.  \cite{smibel, smibeletal, shmemp, niem}. We will show that limit sets of conformal GDMS are mean porous almost everywhere with respect to natural measures.
 \begin{thm}
\label{meanporthminto}
Let $\cS=\{\f_e\}_{e \in E}$ be a finitely irreducible conformal GDMS such that 
\eqref{xx1} holds. Let $\tilde{\mu}$ be any $\sigma$-invariant ergodic measure on $E_A^\N$ with 
finite Lyapunov exponent  $\chi_{\tilde{\mu}}(\sigma)$. Then   $J_{\cS}$ is mean porous at $\tilde{\mu} \circ \pi^{-1}$-a.e. $x\in J_{\cS}$ with porosity constants depending only on $\cS$ and $\tilde{\mu}$.
\end{thm}
Using Theorem \ref{meanporthminto} we show that if, moreover, $\cS$ is strongly regular then $J_{\cS}$ is mean porous at $m_h$-a.e. $x\in J_{\cS}$, where $h:=\dim_{\cH}(J_{\cS})$ and $m_h$ is the $h$-\textit{conformal measure} of $\cS$, see Corollary \ref{meanporoconfhaus}. It is well known that the $h$-conformal measure captures the right amount of information for the limit set of a conformal GDMS. If the system $\cS$ is finite, $m_h$ is, up to multiplicative constants, equal to the Hausdorff and packing measures restricted on $J_{\cS}$, 
see e.g. \cite[Theorem 7.18]{CTU}. But if the system has infinitely many generators, then Hausdorff measure can vanish while packing measure can be infinite, and then (and only then) the above property may fail. See the end of Section \ref{sec:prelim} for a short introduction to conformal measures.

As mentioned earlier, limit sets of finite and irreducible GDMSs with zero Lebesgue measure are porous; that is for any point of the limit set there are nearby holes of radius proportional to their distance from that point. 
However, if one is interested on the spatial distribution of the holes, \textit{directed porosity} is the most appropriate kind of porosity to study. Given $v \in S^{n-1}$ we say that a set $E \subset \R^n$ is \textit{$v$-directed porous} at $x$ if the corresponding holes are centered in the line $\{x+t v, t \in \R\}$, see Section \ref{sec:dirpor} for the exact definition. 

As far as the authors know, directed (or directional) porosity first appears in the work of Preiss and Zaj\' i\v{c}ek \cite{preissza, preissza2} regarding differentiability of Lipschitz maps in Banach spaces. See also the book by Lindenstrauss, Preiss, and Tiser \cite{lpt} and Speight's Ph.D thesis \cite{speight} for further advances and related references. Directed porosity was also introduced independently in \cite{dirpor} (as the first named author was not aware of \cite{preissza} at the time)  in the context of fractal geometry and harmonic analysis. Actually, more general notions of directed porosity were considered in  \cite{dirpor}, where lines could be replaced by any $m$-dimensional plane in $\Rn$. Among other things, it was shown in \cite{dirpor} that if the limit set of a finite conformal IFS has Hausdorff dimension less or equal than $1$, then it is directed porous for all directions, see \cite[Corollary 1.3]{dirpor}. In this paper we provide an easily checkable sufficient condition for a finite and irreducible GDMS to be directed porous at all directions.
 \begin{thm}\label{dirporlimsetintro} Let $\cS=\{\f_{e}\}_{e \in E}$ be a finite and irreducible conformal GDMS such that
\begin{equation}
\label{lsawayfrombdryintro}
\dist(\partial X, J_{\cS} )>0.
\end{equation}
Then $J_{\cS}$ is $v$-directed porous for every $v \in S^{n-1}$.
 \end{thm}
Using Theorem \ref{dirporlimsetintro} we show that any finite and irreducible conformal GDMS which satisfies the well known Strong Separation Condition, is directed porous at all directions, see Theorem \ref{dirporssc}. We also consider systems where \eqref{lsawayfrombdryintro}  does not necessarily hold. In Theorem \ref{rotfreess} we show that if an IFS consists of rotation free similarities and there exist directions $v \in S^{1}$ such that the lines $\{x+t v, t \in \R\}$ miss the set of first iterations in the interior of the set $X$, then the limit set is $v$-directed porous.  

Theorems \ref{thmporfixintro},  \ref{meanporthminto} and \ref{dirporlimsetintro} all deal with various positive aspects of porosities for conformal GDMSs. On the opposite spectrum we also investigate non-porosity of infinite GDMSs. Among other things we prove that if the limit set of a finitely irreducible conformal GDMS is not porous at a single point of its closure then it is not porous in a set of full conformal measure.
\begin{thm}
\label{notporoaeconfint} 
Let $\cS=\{\f_e\}_{e \in E}$ be a  finitely irreducible conformal GDMS such that $J_{\cS}$ is not porous, or $J_{\cS}$ is not porous at some $\zeta \in \overline{J_{\cS}}$. Then   $J_{\cS}$ is not porous at $m_{h}$-a.e. $x\in J_{\cS}$. 
\end{thm}
Moreover for every $h \in (0,n)$ we construct an infinite IFS $\cS_h$ consisting of similarities in $\Rn$ such that $\dim_{\cH} (J_{\cS_h})=h$ and $J_{\cS_h}$ is not porous, and we use Theorem \ref{notporoaeconfint} to conclude that $J_{\cS_h}$ is not porous $m_h$-a.e. See Theorem \ref{hnotpor} for the construction.

In the last three sections we study porosity for some well-known families of dynamical systems. In Section \ref{sec:ccf} we perform a comprehensive study of porosity in the setting of \textit{complex continued fractions}.  The origins of complex continued fractions can be traced back to the works of the brothers Adolf Hurwitz \cite{ahur} and Julius Hurwitz \cite{jhur}. Since their pioneering contributions, complex continued fractions  have been studied widely from different viewpoints, see e.g. \cite[Chapter 5]{henbook} and the references therein. 

Recall that any irrational number in $[0,1]$ can be represented as a continued fraction
$$\cfrac{1}{e_1+\cfrac{1}{e_2+\cfrac{1}{e_3+\dots}}},$$
where $e_i \in \N$ for all $i \in \N$. 
\begin{figure} 
\centering
\includegraphics[scale=0.4]{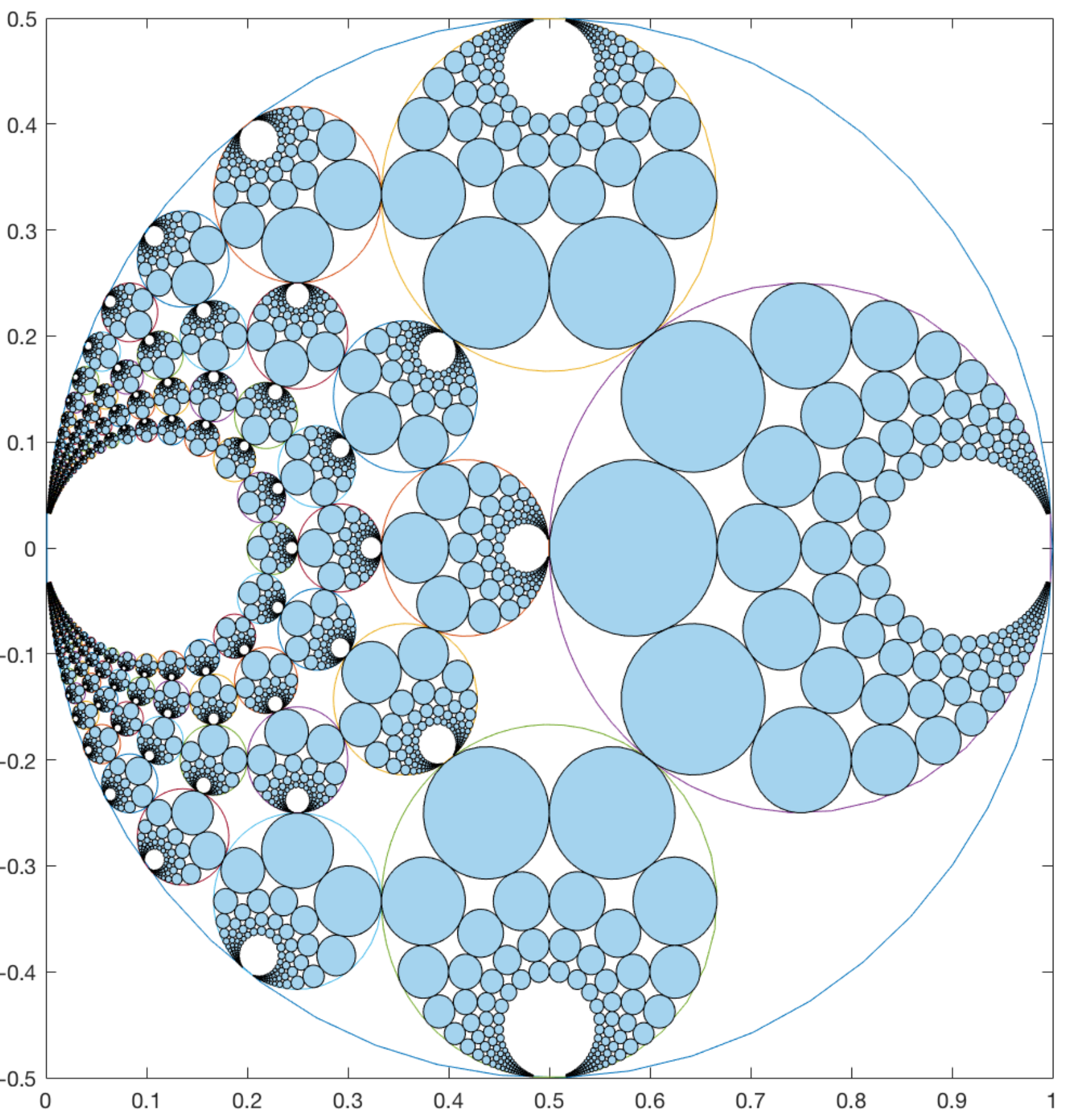}
\includegraphics[scale=0.4]{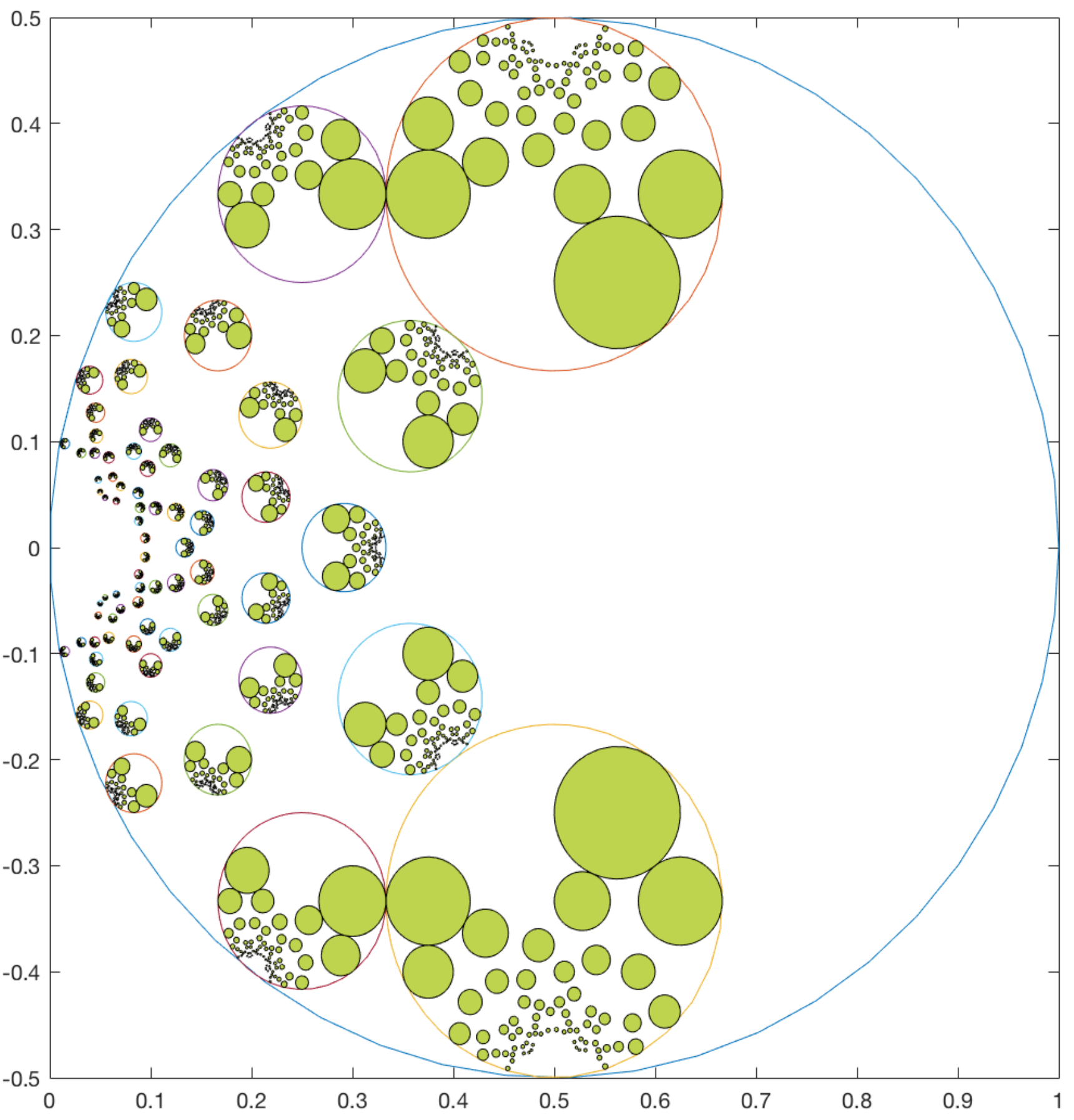}
\caption{Approximation of the limit sets  $\mathcal{CF}_E$ and of $\mathcal{CF}_{GP_{-\pi,\pi}}$ after two iterations.}
\label{gausprimes}
\end{figure}
It is a remarkable fact that continued fractions can be described by the infinite conformal IFS
$$\cf_\N:= \left\{\f_{n}:[0,1] \ra [0,1]: \f_n(x)=\frac{1}{n+x} \mbox{ for }n \in I \subset \N \right\}.$$
 As in the case of real continued fractions, complex continued fractions can be represented by the infinite conformal IFS $$\cf_E=\{\f_e: \bar{B}(1/2,1/2) \ra \bar{B}(1/2,1/2)\}_{e \in E}$$ where $$E=\{m+ni:(m,n) \in \N \times \Z \} \mbox{ and }\f_e(z)=\frac{1}{e+z}.$$
This system was first considered in \cite{MU1} and its geometric and dimensional properties were further studied by various authors, see e.g. \cite{CLU, urbpor, nuss, pri}. In particular it was shown in \cite{urbpor} that $\cf_E$ is not porous.

In this paper we consider porosity properties of $\cf_I$ when $I \subset E$ and we generalize the results from \cite{urbpor} in various ways. In the core of our approach lies Theorem \ref{poroccf}, which provides a novel characterization of porosity for limit sets of complex continued fractions with arbitrary alphabet. One interesting aspect of our characterization is that given any alphabet $I \subset E$ one can check if $J_I$ is porous by solely examining how $I$ is distributed. The proof of Theorem \ref{poroccf}, which is the most technical proof in our paper, is  quite delicate and rather long as we have to consider several cases. See also Remark \ref{remk:marporchar} on how Theorem \ref{poroccf} extends and streamlines earlier results from \cite{urbpor}.

Of special interest to us, are complex continued fraction systems whose alphabet $I$ is finite, co-finite, or even a radial subset of Gaussian primes of the form,
 $$GP_{a,b}=\{w \in E: w \mbox{ is a Gaussian prime and }\arg w \in [a,b)\},$$
where $-\pi/2\leq a<b \leq \pi/2$. We record that Gaussian primes are a very intriguing subset of $\Z[i]$ with many related important open questions, see e.g. \cite{moat}. 

\begin{thm}
\label{poroccfintro} Let $I\subset E$ and denote the limit set of $\mathcal{CF}_I$ by $J_{I}:=J_{\mathcal{CF}_I}$. Moreover let $m_{h_I}$ be the $h_I$-conformal measure of $\mathcal{CF}_I$, where $h_I=\dim_{\cH}(J_{I})$.
\begin{enumerate}[label=(\roman*)]
\item \label{poroccfintro1} If $I$ is finite then $J_{I}$ is directed porous for all $v \in S^{1}$.
\item \label{poroccfintro2} If $I$ is co-finite then:
\begin{enumerate}
\item \label{np2} $J_{I}$ is $m_{h_I}$-a.e. not porous, 
\item $J_{I}$ is $m_{h_I}$-a.e. mean porous with porosity constant only depending on $I$.
\end{enumerate}
\item \label{poroccfintro3} If $I$ is a co-finite subset of $GP_{a,b}$ for some $-\pi/2\leq a<b \leq \pi/2$ then:
\begin{enumerate}
\item \label{np3} $J_{I}$ is $m_{h_I}$-a.e. not porous, 
\item $J_{I}$ is $m_{h_I}$-a.e. mean porous with porosity constant only depending on $I$.
\end{enumerate}
\end{enumerate}
\end{thm}

Theorem \ref{poroccfintro}  highlights the differences between  finite and infinite continued fractions systems regarding their porosity properties. If the alphabet of a continued fractions system is finite, then it is porous in all directions. On the other hand, if the alphabet is infinite, even if it is relatively sparse as a radial subset of the Gaussian primes, then the limit set is not porous almost everywhere. However, continued fractions whose alphabets are co-finite subsets of $E$ or co-finite subsets of radial sectors of Gaussian primes are a.e. porous on a fixed percentage of scales.

The proof of Theorem \ref{poroccfintro} employs several new ideas. In particular, we introduce \textit{upper densities} and \textit{upper box dimensions} (see Defintions \ref{dens} and \ref{bddefn} respectively) for subsets of Gaussian integers and we explore their connections to porosity. In Proposition \ref{densle1} we show that if the limit set of a complex continued fractions system is porous, then its alphabet has upper density strictly less than $1$. Since co-finite subsets of $E$ have upper density equal to $1$, we thus deduce that their limit sets are \textit{not} porous. We then employ Theorem \ref{notporoaeconfint} in order to obtain Theorem \ref{poroccfintro} \ref{np2}. Moreover in Theorem \ref{bdporous} we prove that if the limit set of a complex continued fractions system is porous, then its alphabet has upper box dimension strictly less than $2$. Using Hecke's Prime Number theorem we then show that co-finite subsets of radial sectors of Gaussian primes have upper box dimension equal to $2$, hence they define limit sets which are \textit{not} porous. Thus, Theorem \ref{poroccfintro} \ref{np3} again follows from Theorem \ref{notporoaeconfint}.

We finally turn our attention to complex dynamics and we conclude our paper with two fairly short sections which explore porosity for Julia sets of meromorphic, especially transcendental, functions and then, more specifically, of elliptic functions. We first investigate various porosity properties of tame meromorphic functions. We record that tameness is a mild hypothesis which is satisfied by many natural classes of maps, see Remark \ref{tamerem} for more details. 

\begin{thm}
\label{merocombined} Let $f:\C\to\oc$ be a tame meromorphic function such that $J(f)\ne\C$, where $J(f)$ denotes the Julia set of $f$. Moreover, let $\mu$ be a Borel probability $f$-invariant ergodic measure on $J(f)$ with full topological support.
\begin{enumerate}[label=(\roman*)]
\item The Julia set $J(f)$ is porous at a dense set of its points.
\item If $\mu$ has finite Lyapunov exponent  
$\chi_{\mu}(f):=\int_{J(f)}\log|f'|\,d\mu$, then  $J(f)$ is mean porous at $\mu$-a.e. $x\in J(f)$ with porosity constants only depending on $f$.
\item If $J(f)$ is not porous in $\C$ then $J(f)$ is not porous at $\mu$-a.e. $x \in J(f)$.
\end{enumerate}
\end{thm}

The proof of Theorem \ref{merocombined} is based on the porosity results on conformal iterated function systems obtained in Section \ref{sec:porogdms}. Remarkably, such applications are possible due to a powerful tool of complex dynamics known as \textit{nice set}. Roughly speaking, each sufficiently ``good'' meromorphic function admits a set, commonly named as a nice set, such that the holomorphic inverse branches of the first return map it induces, form a conformal iterated function system satisfying the open set condition. We carefully define nice sets in Section~\ref{pmf} and we show how they lead to conformal iterated function systems. More details about these sets, their properties, and their constructions can be found for example in \cite{Riv07, PR1,Dob11, SkU,KU1}. 

Recall that a meromorphic function $f: \C \ra \hat{\C}$ is called {\it elliptic} if it is doubly periodic. Iteration of elliptic functions have been studied widely in complex dynamics. A comprehensive and systematic account of iteration of elliptic functions can be found in the forthcoming book \cite{KU1}; we also refer the reader to the following articles: \cite{KU2,KU3,HK,MayerU3}. In addition to the results on general meromorphic function discussed above, the following theorem establishes several, more refined, porosity properties of Julia sets of elliptic functions. We would also like to emphasize that in this case we impose no tameness assumptions. 

\begin{thm}
\label{nonporoellipintro}
Let $f:\C \ra \hat{\C}$ be a non-constant elliptic function. Then:
\begin{enumerate}[label=(\roman*)]
\item The Julia set $J(f)$ is not porous at a dense set of its points, in particular it is not porous at any point of the set 
$$
P_f:=\bigcup_{n=1}^{\infty} f^{-n}(\infty).
$$

\item For every $b \in P_f$ and for all $\kappa \in (0,1)$ there exists $R(b,\kappa)>0$ such that
$$
\por(J(f),b,r)\le \kappa
$$
for all $r\in(0,R(b,\kappa))$.
\item[(iii)] If in addition $J(f)\ne\C$, then  $J(f)$ is porous at a dense set of its points; the repelling periodic points of $f$.
\end{enumerate}
\end{thm}

The paper is organized as follows. In Section \ref{sec:prelim} we  introduce all the relevant concepts related to graph directed Markov systems and their thermodynamic formalism.  Section \ref{sec:porogdms} consists of four subsections, each of them devoted to a different aspect of porosity for general  graph directed Markov systems. In Subsection \ref{porosity} we explore (classical) porosity and we prove Theorem \ref{thmporfixintro}. Subsection \ref{sec:mean} deals with mean porosity and contains, among other results, the proof of Theorem \ref{meanporthminto}. In Subsection \ref{sec:dirpor} we study directed porosity for finite systems and we prove Theorem \ref{dirporlimsetintro}. In Subsection \ref{sec:nonpor} we investigate how non-porous are the limit sets of general graph directed Markov systems, and we prove Theorem \ref{notporoaeconfint}. In Section \ref{sec:ccf} we narrow our focus to complex continued fractions and we prove Theorem \ref{poroccfintro}. Sections \ref{pmf} and \ref{sec:elliptic} delve into porosity for Julia sets of meromorphic functions. In Section \ref{pmf} we deal with tame meromorphic functions and we prove Theorem \ref{merocombined}. Section \ref{sec:elliptic} concerns elliptic functions and contains the proof of Theorem \ref{nonporoellipintro}.

\section{Preliminaries}
\label{sec:prelim}
A {\it graph directed Markov system} (GDMS) \index{GDMS}
$$
\cS= \big\{ V,E,A, t,i, \{X_v\}_{v \in V}, \{\f_e\}_{e\in E} \big\}
$$
consists of
\begin{itemize}
\item a directed multigraph $(E,V)$ with a countable set of edges $E$, frequently referred to as the {\em alphabet} of $\cS$, and a finite set of vertices $V$,

\item an incidence matrix $A:E \times E \ra \{0,1\}$,

\item two functions $i,t:E\ra V$ such that $t(a)=i(b)$ whenever $A_{ab}=1$,

\item a family of non-empty compact metric spaces $\{X_v\}_{v\in V}$,


\item  a family of injective contractions $$ \{\phi_e:X_{t(e)}\to X_{i(e)}\}_{e\in E}$$ such that every $\phi_e,\, e\in E,$ has Lipschitz constant no larger than $s$ for some $s \in (0,1)$.
\end{itemize}

 We will always assume that the alphabet $E$ contains at least two elements and for every $v \in V$ there exist $e,e' \in E$ such that $t(e)=v$ and $i(e')=v$. We will frequently use the simpler notation $\cS=\{\f_e\}_{e \in E}$ for a GDMS. If a GDMS has finite alphabet it will be called {\it finite}\index{GDMS!finite}. 

We now introduce some standard notation from symbolic dynamics. For every $\om \in E^*:=\bigcup_{n=0}^\infty E^n$, we denote by $|\om|$  the unique integer
$n \geq 0$ such that $\om \in E^n$, and we call $|\om|$ the {\em length} of
$\om$. We also set $E^0=\{\emptyset\}$. If $\om \in
E^\mathbb{N}$ and $n \geq 1$, we define
$$
\om |_n:=\om_1\ldots \om_n\in E^n.
$$
For $\tau \in E^*$ and $\om \in E^* \cup E^\mathbb{N}$, we let
$$\tau\om:=(\tau_1,\dots,\tau_{|\tau|},\om_1,\dots).$$
Given $\om,\tau\in E^{\mathbb N}$, we denote the longest
initial block common to both $\om$ and $\tau$ by $\omega\wedge\tau  \in
E^{\mathbb N}\cup E^*$. We also denote by
$$
\sg: E^\mathbb{N} \ra E^\mathbb{N}
$$
the \textit{shift map},  which  is given by the
formula
$$
\sg\left( (\om_n)^\infty_{n=1}  \right) =  \left( (\om_{n+1})^\infty_{n=1}  \right).
$$

Given a matrix $A:E \times E \to \{0,1\}$ we denote
$$
E^\mathbb{N}_A
:=\{\om \in E^\mathbb{N}:  \,\, A_{\om_i\om_{i+1}}=1  \mbox{ for
  all }\,  i \in \N
\}.
$$
Elements \index{A-admissible matrices@$A$-admissible matrices} of $E^\mathbb{N}_A$ are called {\it $A$-admissible} (infinite) words. We also set
$$
E^n_A
:=\{w \in E^\mathbb{N}:  \,\, A_{\om_i\om_{i+1}}=1  \,\, \mbox{for
  all}\,\,  1\leq i \leq
n-1\}, \quad n \in \N,
$$
and
$$
E^*_A:=\bigcup_{n=0}^\infty E^n_A.
$$
The elements of $E_A^\ast$ are called $A$-admissible (finite) words.  A matrix $A:E \times E \ra \{0,1\}$ is called {\it finitely irreducible} if there exists a finite set $\Lambda \subset E_A^*$ such that for all $i,j\in E$ there exists $\om\in \Lambda$ for which $i\om j\in E_A^*$. If the associated matrix of a GDMS is finitely irreducible, we will call the GDMS finitely irreducible as well. For
every  $\om \in E^*_A$, we let
$$
[\om]:=\{\tau \in E^\mathbb{N}_A:\,\, \tau_{|_{|\om|}}=\om \}.
$$



Let $\cS= \big\{ V,E,A, t,i, \{X_v\}_{v \in V}, \{\f_e\}_{e\in E} \big\}$ be a GDMS. For $\om \in E^*_A$ we define the map coded by $\om$:
\begin{equation}\label{phi-om}
\phi_\om=\phi_{\om_1}\circ\cdots\circ\phi_{\om_n}:X_{t(\om_n)}\to
X_{i(\om_1)} \qquad \mbox{if $\om\in E^n_A$.}
\end{equation}
Slightly abusing notation we will let $t(\om) = t(\om_n)$ and $i(\om)=i(\om_1)$ for
$\om$ as in \eqref{phi-om}.

For $\om \in E^{\mathbb N}_A$, the sets
$\{\f_{\om|_n}(X_{t(\om_n)})\}_{n=1}^\infty$ form a decreasing (in the sense of inclusion)
sequence of non-empty compact sets and therefore they have nonempty intersection.
Moreover
$$
\diam(\f_{\om|_n}(X_{t(\om_n)})) \le s^n\diam(X_{t(\om_n)})\le s^n\max\{\diam(X_v):v\in V\}
$$
for every $n\in\N$, hence
$$
\pi(\om):=\bigcap_{n\in  \N}\f_{\om|_n}(X_{t(\om_n)})
$$
is a singleton. Thus we can now define the coding map \index{coding map}
\begin{equation}
\label{picoding}
\pi:E^{\mathbb N}_A\to \du_{v\in V}X_v:=X,
\end{equation}
the latter being a disjoint union of the sets $X_v$, $v\in V$.
The set
$$
J=J_\cS:=\pi(E^{\mathbb N}_A)
$$
will be called the {\it limit set} \index{limit set} (or {\it attractor}) \index{attractor} of the GDMS $\cS$.

For  $\alpha > 0$, we define the metrics $d_\alpha$ on
$E_A^{\mathbb N}$ by setting
\begin{equation}\label{d-alpha}
d_\alpha(\om,\tau) ={\rm e}^{-\alpha|\om\wedge\tau|}.
\end{equation}
It is easy to see that all the metrics $d_{\alpha} $induce the same topology. It is also well known, see \cite[Proposition 4.2]{CTU}, that the coding map $\pi:E^{\mathbb N}_A\to \du_{v\in V}X_v$ is H\"older continuous, when $E^{\mathbb N}_A$ is
equipped with any of the metrics $d_\alpha$ as in \eqref{d-alpha} and
$\du_{v\in V}X_v$ is equipped with the direct sum metric.

In this paper we will focus on conformal GDMS. Let $U \subset \R^n$ be open and connected. Recall that a $C^1$ diffeomorphism $\f: U \ra \R^n$ is \textit{conformal} if its derivative at every point of $U$ is a similarity map. By $D \f(z) : \R^n \ra \R^n$ we denote the derivative of $\f$ evaluated at the point $z$ and by we denote by  $\|D \f (z)\|$ its norm, which in the conformal case coincides with the  similarity ratio. 
Note that for
\begin{itemize}
\item  $n = 1$ the map $\f$ is conformal if and only if it is a $C^1$-diffeomorphism, 
\item $n=2$ the map $\f$ is conformal if and only if it is either holomorphic or antiholomorphic, 
\item $n \geq 3$ the map $\f$ is conformal if and only if it is a  M\"obius transformation.
\end{itemize}
\begin{defn}
\label{gdmsdef}\label{Carnot-conformal-GDMS}
A graph directed Markov system is called {\it conformal} if the following conditions are satisfied.
\begin{itemize}
\item[(i)] For every vertex $v\in V$, $X_v$ is a compact connected
subset of a fixed Euclidean space $\R^n$, $n\ge 1$, and $X_v=\overline{\Int(X_v)}$.
\item[(ii)] ({\it Open Set Condition} or {\it OSC}). \index{open set condition} For all $a,b\in
E$, $a\ne b$,
$$
\phi_a(\Int(X_{t(a)})) \cap \phi_b(\Int(X_{t(b)}))= \emptyset.
$$
\item[(iii)] For every vertex $v\in V$ there exists an open connected
set $W_v \supset X_v$ such that for every $e\in E$ with $t(e)=v$, the map
$\f_e$ extends to a $C^1$ conformal diffeomorphism of $W_v$ into $W_{i(e)}$.
\end{itemize}
\end{defn}
\begin{remark}
\label{ifsosc}
When $V$ is a singleton and for every $e_1,e_2 \in E$, $A_{e_1e_2}=1$ if and only if $t(e_1)=i(e_2)$, the GDMS is called an {\it iterated function system} (IFS). 
\end{remark}

\begin{remark}
\label{ssc} In several instances we will use a stronger separation condition than OSC. We will say that a conformal GDMS $\cS=\{\f_e\}_{e \in E}$ satisfies the {\em Strong Separation Condition}, or {\it SSC}, if 
$$\f_a ( X_{t(a)}) \cap \f_b ( X_{t(b)})=\emptyset$$
for all $a,b \in E$ such that $a \neq b$.
\end{remark}

We will denote the Euclidean metric by $d$ or $|\cdot|$. For each $v \in V$, we select a compact set $S_v$ such that
$X_v \subset \Int(S_v) \subset S_v \subset W_v$. Moreover the sets $S_v, v \in V,$ are chosen to be pairwise disjoint. We denote
\begin{equation}
\label{itaS}
\eta_{\cS}=\min \{\dist(X_v, \partial S_v):v \in V\}.
\end{equation}

We will always assume that a conformal GDMS satisfies the \textit{Bounded Distortion Property}. This means that there exist constants $\alpha>0$ and $K \geq 1$ so that $$
\biggl|\frac{\|D\f_\om(p)\|}{\|D\f_\om(q)\|}-1\biggr|\le Kd^\alpha(p,q)
$$
and
\begin{equation}
\label{bdp}
K^{-1}\le\frac{\|D\f_\om(p)\|}{\|D\f_\om(q)\|}\le K
\end{equation}
for every $\om\in E_A^*$ and every pair of points $p,q\in S_{t(\om)}$, where $d$ is the Euclidean metric on $\R^n$. If $n \geq 2$ conformal maps automatically satisfy the Bounded Distortion Property (with $\alpha=1$), see e.g. \cite{MUGDMS}, nevertheless it is not automatically satisfied for conformal GDMS in $\R$.

For $\om \in E^*_A$ we set
$$\|D \f_\om\|_\infty := \|D \f_\om\|_{S_{t(\om)}}.$$
Note that \eqref{bdp} and the Leibniz rule easily imply
that if $\om \in E_A^\ast$ and $\om=\tau \upsilon$ for some $\tau, \upsilon \in E_A^\ast$, then
\begin{equation}\label{quasi-multiplicativity1}
K^{-1} \|D\f_{\tau}\|_\infty \, \|D\f_{\upsilon}\|_\infty \le
\|D\f_\om\|_\infty \le \|D\f_{\tau}\|_\infty \,
\|D\f_{\upsilon}\|_\infty.
\end{equation}
Moreover, 
for every $\om \in E^*_A$, and every $p,q \in B(X_{t(\om)}, \eta_{\cS}/2)$,
\begin{equation}
\label{boundbyderlip} 
d(\f_{\om}(p), \f_{\om}(q)) \leq  \|D\f_\om\|_\infty d(p,q).
\end{equation}
In particular for every $\om \in E^*_A$
\begin{equation}
\label{boundbyder} 
\diam(\f_\om(X_{t(\om)}))\leq   \|D\f_\om\|_\infty \diam(X_{t(\om)}).
\end{equation}
 The following lemma, which we'll use repeatedly, provide information on the distortion of the iterations $\f_{\om}$. It's proof can be found  in \cite[Section 4.1]{MUGDMS}.
\begin{lm}[Egg Yolk Principle]\label{eyp} Let $S=\{\phi_e\}_{e\in E}$ be a  conformal GDMS. Then for all finite words $\om\in E_A^*$, all $p \in X_{t(\om)}$ and all
$0<r<\eta_{\cS}$,
\begin{equation}\label{4.1.8}
B(\f_\om(p),K^{-1}\|D\f_\om\|_\infty r) \subset
\f_\om(B(p,r)) \subset
B(\phi_\om(p),\|D\phi_\om\|_\infty r).
\end{equation}
\end{lm}

We also need to recall some well known facts from thermodynamic formalism. Let $\mathcal{S}=\{\f_e\}_{e\in E}$  be a finitely irreducible  conformal GDMS. For $t\ge 0$ and $n \in \N$ let
\begin{equation}\label{zn}
Z_n(\cS,t):=Z_{n}(t) := \sum_{\om\in E^n_A} \|D\phi_\om\|^t_\infty.
\end{equation}
By \eqref{quasi-multiplicativity1} we easily  see that
\begin{equation}
\label{zmn}
Z_{m+n}(t)\le Z_m(t)Z_n(t),
\end{equation}
and consequently the
sequence $(\log Z_n(t))_{n=1}^\infty$ is subadditive. Thus, the limit
$$
\lim_{n \to  \infty}  \frac{ \log Z_n(t)}{n}
$$
exists and equals $\inf_{n \in \N} (\log Z_n(t) / n)$. The value of
the limit is denoted by $P(t)$ (or if we want to be more precise by  $P_\mathcal{S}(t)$) and it is called the {\em topological pressure} of the system $\mathcal{S}$ evaluated at the parameter $t$. We also define two special parameters related to topological pressure; we let
$$\theta(\cS):=\theta=\inf\{t \geq 0: P(t)<\infty\}\quad \mbox { and }\quad h(\cS):=h=\inf\{t \geq 0: P(t)\leq 0\}.$$
The parameter $h(\cS)$ is known as \textit{Bowen's parameter}. 

It is well known that $t \mapsto P(t)$  is decreasing on $[0,+\infty)$ with $\lim _{t \ra +\infty} P(t)= -\infty$, and it is convex and continuous on $\overline{\{t \geq 0: P(t)<\infty\}}.$ Moreover
\begin{equation}
\label{thetaz}
\theta({\cS}):=\theta=\inf\{t \geq 0: P(t)<\infty\}=\inf\{t \geq 0: Z_1(t)<\infty\},
\end{equation}
and for $t \geq 0$
\begin{equation}
\label{presz1}
P(t)<+\infty \mbox{ if and only if }Z_1(t)<+\infty.
\end{equation}
For proofs of these facts see e.g. \cite[Proposition 7.5]{CTU} and \cite[Lemma 3.10]{CLU}.

\begin{defn}
\label{regulardef}
A finitely irreducible conformal GDMS $\cS$ is:
\begin{enumerate}[label=(\roman*)]
\item {\it regular}  if $P(h)=0$,
\item {\it strongly regular} if there exists $t \geq 0$ such that $0< P(t) <+\infty$.
\item {\it co-finitely regular} if $P(\theta)=+\infty$.
\end{enumerate}
\end{defn}
It is well known, see e.g. \cite[Proposition 7.8]{CTU} that
\begin{equation}
\label{cofimpliesreg}
\mbox{co-finitely regular} \Longrightarrow \mbox{strongly regular} \Longrightarrow \mbox{regular}.
\end{equation}

Topological pressure plays a key role in the dimension theory of conformal dynamical systems:
\begin{thm}
\label{721}
If $\mathcal{S}$ is a finitely irreducible conformal GDMS, then
$$
h(S)= \dim_{\mathcal{H}}(J_\mathcal{S})
= \sup \{\dim_\cH(J_F):  \, F \subset E \, \mbox{finite} \, \}.
$$
\end{thm}
For the proof see \cite[Theorem 7.19]{CTU} or \cite[Theorem 4.2.13]{MUGDMS}. 

We close this section with a discussion regarding conformal measures. If $\cS=\{\f_e\}_{e \in E}$ is a finitely irreducible conformal GDMS  we define
$$\Fin(\cS):=\left\{t:\sum_{e\in E}||D\phi_e||_\infty^t<+\infty\right\}.$$
Moreover for $t \in \Fin(\cS)$ we define the \textit{Perron-Frobenius operator} with respect to $\cS$ and $t$ as
\begin{equation}\label{1j89}
\mathcal{L}_t g(\om)= \sum_{i:\, A_{i \om_1}=1}
g(i \om)\| D\phi_i(\pi (\om))\|^t \quad \mbox{for $g \in C_b(E^\mathbb{N}_A)$ and $\om
  \in E^\mathbb{N}_A$,}
\end{equation}
where $C_b(E^\mathbb{N}_A)$ is the Banach space of all real-valued bounded continuous functions on $E_A^\N$ endowed with the supremum norm $||\cdot||_\infty$. Note that $\mathcal{L}_t: C_b(E^\mathbb{N}_A) \ra C_b(E^\mathbb{N}_A)$. As usual we denote by $\mathcal{L}_t^*: C_b^*(E^\mathbb{N}_A)\to C_b^*(E^\mathbb{N}_A)$ the dual operator of $\mathcal{L}_t$. We will use the following theorem repeatedly. Its proof can be found in \cite[Theorem 7.4]{CTU}.
\begin{theorem} \label{thm-conformal-invariant}
Let $\mathcal{S}=\{\f_e\}_{e\in E}$  be a finitely irreducible conformal GDMS and let $t\in\Fin(\cS)$. 
\begin{enumerate}[label=(\roman*)]
\item \label{eiggibstci} There exists a unique eigenmeasure $\tilde m_t$ of the conjugate Perron-Frobenius operator $\mathcal{L}_t^*$ and the corresponding eigenvalue is  $e^{P(t)}$.
\item The eigenmeasure $\tilde m_t$ is a Gibbs state for the potential $$\om \mapsto t \log\|D\phi_{\om_1}(\pi(\sg(\om))\|:=t \zeta (\om).$$

\item The potential $t\zeta:E_A^\N\to\R$ has a unique shift-invariant Gibbs state $\tilde\mu_t$ which is ergodic and globally equivalent to $\tilde m_t$.
\end{enumerate}
\end{theorem}

For all $t \in \Fin(\cS)$ we will denote
\begin{equation}\label{mutdef}
m_t := \tilde m_t\circ \pi^{-1} \  \  {\rm and  }  \  \  \  \mu_t := \tilde \mu_t\circ \pi^{-1}.
\end{equation}
We record that the measures $\tilde \mu_t, \tilde m_t$ are probability measures, since they are Gibbs states, hence the measures $m_t, \mu_t$ are probability measures as well. Note also that $h=h(\cS) \in \Fin (\cS)$, and as it turns out  the measure $m_h$, which we will call the \textit{$h$-conformal measure} of $\cS$, is particularly important for the geometry of $J_{\cS}$. This is evidenced by the following theorem, which is straightforward to prove. 

\begin{theorem}\label{t120190729}
Let $\mathcal{S}=\{\f_e\}_{e\in E}$ be a finitely irreducible conformal GDMS.
If either $\mathcal{H}^h |_{J_{\cS}}$, the $h$-dimensional Hausdorff measure restricted to $J_{\cS}$, is positive or $\mathcal{P}^h |_{J_{\cS}}$, the $h$-dime- nsional packing measure restricted to $J_{\cS}$, is finite, 
then the system $\mathcal{S}$ is regular, and in the former case $m_h=\mathcal{H}^h(J_{\cS})^{-1}\mathcal{H}^h |_{J_{\cS}}$ while in the latter case $m_h=\mathcal{P}^h(J_{\cS})^{-1}\mathcal{P}^h |_{J_{\cS}}$
\end{theorem}

See \cite[Chapters 6,7,8]{CTU} or \cite[Chapter 4]{MUGDMS} for more information on Gibbs states and conformal measures. 

\section[Porosities and GDMS]{Porosities and conformal graph directed Markov systems}
\label{sec:porogdms}
In this section we study various aspects of porosities for general graph directed Markov systems. We note that the results obtained in Subsections \ref{porosity}, \ref{sec:mean} and \ref{sec:nonpor} hold for Carnot conformal GDMS as well, see \cite{CTU} for more information on Carnot GDMS.
\subsection{Porosity}
\label{porosity}
We start with the formal definition of porosity. If $(X,d)$ is a metric space and $E \subset X$, $x \in X$ and $r \in (0,\diam(E))$ we let
\begin{equation}
\label{porexr}
\por (E,x,r)=\sup \{c \geq 0: B(y,c r) \subset B(x,r)\stm E \mbox{ for some }y\in X\}.
\end{equation}
\begin{defn}
Let $(X,d)$ be a metric space and let $E \subset X$ be a bounded set. 
Given $c\in (0,1)$ and $x \in X$, we say that $E \subset X$ is:
\begin{enumerate}[label=(\roman*)]
\item   {\em$c$-porous at} $x$ if there exists some $r_0>0$ such that $\por (E,x,r) \geq c $ for every $r \in (0, r_0)$, 
\item {\em porous at} $x$ if there exists some $c \in (0,1)$ such that $E$ is {\em$c$-porous at} $x$,
\item {\em $c$-porous} if there exists some $r_0>0$ such that $\por (E,x,r) \geq c$ for every $x \in E$ and $r \in (0, r_0
)$,
\item {\em porous} if it is $c$-porous for some $c\in (0,1)$.
\end{enumerate}
\end{defn}
Before proving our first auxiliary lemma we need to introduce some extra notation. Recall first that $X=\cup_{v \in V} X_v$. For $n \in \N$ and $v \in V$ let
\begin{equation}
\label{xn}
X^n:=\bigcup_{\om \in E_A^n} \f_\om(X_t(\om)),
\end{equation}
and
\begin{equation}
\label{xnv}
X^n_v:=X^n \cap X_v =\bigcup_{\om \in E_A^n:i(\om)=v} \f_\om(X_t(\om)).
\end{equation}

\begin{lm}
\label{lemmax1}
Let $\cS=\{\f_e\}_{e \in E}$ be a finitely irreducible conformal GDMS such that 
\begin{equation}
\label{xx1}
X \stm \overline{ \cup_{e \in E} \f_e (X_{t(e)})} \neq \emptyset.
\end{equation}
Then 
\begin{enumerate}[label=(\roman*)]
\item \label{lemmax11}  For every $v \in V$ there exists some $m_v \in \N$ such that
\begin{equation}
\label{xv}
\Int X_{v} \stm \overline{X_{v}^{m_v}} \neq \emptyset.
\end{equation}
\item \label{lemmax12} There exists a family of open balls $\{B_v\}_{v \in V}, B_v \subset \Int X_v,$ such that for every $\om \in E_A^\ast$,
\begin{equation}
\label{bset1}\f_\om (B_{t(\om)}) \cap J_{\cS}=\emptyset.\end{equation}
\end{enumerate}
\end{lm}
\begin{proof} Throughout the proof we are going to use repeatedly the fact that the maps
$$\f_e :W_{t(e)} \ra \f_{e}( W_{t(e)}) \subset W_{i(e)}$$
are homeomorphisms for every $e \in E$. The proof of \ref{lemmax11} is similar to the proof of \cite[Theorem 8.26]{CTU}, nevertheless we include the details. First note that \eqref{xx1} implies that 
$$\Int X \stm \overline{X^1} \neq \emptyset.$$ 
Since $\Int X=\cup_{v \in V} \Int X_v$, it follows that there exists some $v_0 \in V$  such that
\begin{equation}
\label{xvo}
\Int X_{v_0} \stm \overline{X_{v_0}^1} \neq \emptyset.
\end{equation}

Now let $v \in V$ and  $e, e_0 \in E$ such that $i(e)=v$ and $t(e_0)=v_0$. Since $\cS$ is finitely irreducible there exists some $\om \in E_A^\ast$, with $|\om| \leq m_0$ for some $m_0 \in \N$ depending only on $\cS$, such that $\om':=e\om e_0 \in E_A^\ast$. Therefore, 
\begin{equation}
\label{poswprime}
\f_{\om'}(\Int X_{v_0}) \stm \f_{\om'}(\overline{X^1_{v_0}}) \neq \emptyset.
\end{equation}
Now if $|\om'|=k$ notice that
\begin{equation}
\label{posdif}
\overline{X^{k+1}_v} \cap \f_{\om'}(\Int X_{v_0}) \subset \f_{\om'}( \overline{X^1_{v_0}}).
\end{equation}
To prove \eqref{posdif}, first observe that
\begin{equation*}
\begin{split}
X_{v}^{k+1} &\subset \f_{\om'} \left( \bigcup_{a \in E:i(a)=v_0} \f_a (X_{t(a)})\right) \cup \bigcup_{\tau \in I_v} \f_\tau (X_{t(\tau)})\\
&=\f_{\om'}(X^1_{v_0})\cup \bigcup_{\tau \in I_v} \f_\tau (X_{t(\tau)}),
\end{split}
\end{equation*}
where $I_v=\{ \tau \in E_A^{k+1}:i(\tau)=v \text{ and }\tau|_k \neq \om'\}$. Therefore
\begin{equation}
\label{xvint1}
\overline{X_{v}^{k+1}} \subset \f_{\om'}(\overline{X^1_{v_0}})\cup \overline{\bigcup_{\tau \in I_v} \f_\tau (X_{t(\tau)})}.
\end{equation}
Note that for every $\tau \in I_v$ by the open set condition
$$\f_\tau (X_{t(\tau)}) \cap \f_{\om'}(\Int X_{v_0}) \subset \f_{\tau|_k} (X_{t(\tau_k)}) \cap \f_{\om'}(\Int X_{v_0})= \emptyset,$$
hence 
$$\bigcup_{\tau \in I_v} \f_\tau (X_{t(\tau)}) \cap \f_{\om'}(\Int X_{v_0})= \emptyset.$$
Since $\f_{\om'}(\Int X_{v_0})$ is open, we deduce that
\begin{equation}
\label{xvint2}\overline{\bigcup_{\tau \in I_v} \f_\tau (X_{t(\tau)})} \cap \f_{\om'}(\Int X_{v_0})= \emptyset.
\end{equation}
Therefore  \eqref{posdif} follows by \eqref{xvint1} and \eqref{xvint2}. 
Recall that
$$\f_{\om'}(\Int(X_{v_0}))=\f_{\om'}(\Int(X_{t(e_0)}))=\f_{\om'}(\Int(X_{t(\om')}))=\f_{\om'}(\Int(X_{t(e_0)}))\subset \Int X_{i(e)}=\Int X_v.$$
Hence,
\begin{equation*}
\begin{split}
\Int X_v \setminus \overline{X_v^{k+1}} &\supset \f_{\om'} (\Int X_{v_0}) \setminus \overline{X_v^{k+1}} = \f_{\om'} (\Int X_{v_0}) \cap (\f_{\om'} (\Int X_{v_0}) \setminus \overline{X_v^{k+1}})\\
&\overset{\eqref{posdif}}{\supset} \f_{\om'} (\Int X_{v_0}) \setminus \f_{\om'}(\overline{X^1_{v_0}}) \overset{\eqref{poswprime}}{\neq} \emptyset,
\end{split}
\end{equation*}
and this implies \eqref{xv} because $k=|\om'| \leq m_0+2$. The proof of \ref{lemmax11} is complete

We will now prove \ref{lemmax12}. By \eqref{xv}, for all $v \in V$, there exist $z_v \in X_v$ and $r_v>0$ such that
\begin{equation}
\label{bsubsetofint}
B_v:=B(z_v,r_v) \subset \Int X_{v} \stm \overline{X_{v}^{m_v}}.
\end{equation}

Observe first that 
\begin{equation}
\label{bset2}
B_v \cap J_{\cS}=\emptyset
\end{equation}
for all $v \in V$. Because if not, there exists some $y \in B_v \cap J_{\cS}$  and since $y \in J_{\cS}$ there exists some $\tau \in E_A^\N$ such that  
$$y= \pi (\tau)= \bigcap_{n=1}^\infty \f_{\tau|_n} (X_{t(\tau_n)}).$$
In that case $y \in \f_{\tau|_{m_v}} (X_{t(\tau_{m_v})}) \subset X_v^{m_v}$, which contradicts \eqref{bsubsetofint}. Hence \eqref{bset2} follows. 
Suppose now that there exists some $\om \in E_A^\ast$ and some $y \in B_{t(\om)}$ such that $\f_\om (y) \in J_{\cS}$. 
So there exists some $\tau \in E_A^\N$, such that
\begin{equation}
\label{fomy}
\f_\om (y)=\pi (\tau)= \bigcap_{n=1}^\infty \f_{\tau|_n} (X_{t(\tau_n)}).
\end{equation}

We will show that $\om=\tau|_{\om}$. By contradiction assume that this is not the case. Let $n_0 < |\om|$  be the smallest natural number such that $\om_{n_0} \neq \tau_{n_0}$. Then by the open set condition,
$$\f_{\om_{n_0}} (\Int (X_{t(\om_{n_0})})) \cap \f_{\tau_{n_0}} (X_{t(\tau_{n_0})})=\emptyset,$$
and consequently
$$\f_{\om|_{n_0}} (\Int (X_{t(\om_{n_0})})) \cap \f_{\tau|_{n_0}} (X_{t(\tau_{n_0})})=\emptyset.$$
Note also that 
$$\f_{\om_{n_0+1,\dots,|w|}}(\Int (X_{t(\om_{|\om|})}))=\Int \f_{\om_{n_0+1,\dots,|w|}} (X_{t(\om_{|\om|})})) \subset \Int (X_{t(\om_{n_0})}),$$
and
$$
\f_{\tau_{n_0+1,\dots,|w|}} (X_{t(\tau_{|\om|})}) \subset X_{t(\tau_{n_0})}.
$$
Therefore
\begin{equation}
\label{ballj}\f_{\om} (\Int (X_{t(\om_{|\om|})})) \cap \f_{\tau|_{|\om|}} (X_{t(\tau_{|\om|})})=\emptyset.
\end{equation}
Recalling \eqref{fomy} and the fact that $y \in B_{t(\om)} \subset \Int X_{t(\om)}$ we deduce that $$\f_{\om}(y) \in \f_{\om}( \Int (X_{t(\om_{|\om|})})) \cap \f_{\tau|_{|\om|}} (X_{t(\tau_{|\om|})}),$$ which contradicts \eqref{ballj}. 

Therefore  $\om=\tau|_{|\om|}$, and
$$\f_{\om} (y)= \f_{\tau|_{|\om|}} (\pi (\sigma^{|\om|}(\tau)))=\f_{\om} (\pi (\sigma^{|\om|}(\tau))).$$
Hence
$$y=\pi (\sigma^{|\om|}(\tau)) \in J_{\cS},$$
and this contradicts \eqref{bset2}. So we have established \eqref{bset1} and the proof of \ref{lemmax12} is complete.
\end{proof}

Under the mild assumption \eqref{xx1}, limit sets of finitely irreducible graph directed Markov systems are porous in large (in the sense of category and dimension) subsets. This is the content of Theorem \ref{thmporfixintro}, which we restate and prove in the following.  
\begin{thm}
\label{thmporfix}
Let $\cS=\{\f_e\}_{e \in E}$ be a finitely irreducible conformal GDMS which satisfies \eqref{xx1}. Then the following hold:
\begin{enumerate}[label=(\roman*)]
\item \label{porthm2} The limit set $J_{\cS}$ is porous at every fixed point of $\cS$, in particular $J_{\cS}$ is porous at a dense set of $J_{\cS}$. 
\item \label{porthm3a} For every $\ve>0$ there exists some finite $F(\ve):=F \subset E$  with  $\dim_{\cH} (J_{\cS_F})>\dim_{\cH} (J_{\cS})-\ve$, such that $J_{\cS}$ is porous at every $x \in J_{\cS_F}$ with porosity constant only depending on $F$ and $\cS$.
\item \label{porthm3b} There exists some set $\tilde{J} \subset J_{\cS}$ such that $\dim_{\cH} (\tilde{J})=\dim_{\cH} (J_{\cS})$ and $J_{\cS}$ is porous at every $x \in \tilde{J}$.
\end{enumerate}
\end{thm}

\begin{proof} 
Let $\{B(z_v,r_v)\}_{v \in V}$ as in Lemma \ref{lemmax1}. Throughout the proof, without loss of generality, we will assume that $r_0:=\min_{v \in V} {r_v}<\eta_{\cS}$.

Let $x_\om$ be a fixed point of $\cS$ corresponding to some $\om \in E_A^\ast$, that is
$$\{x_{\om}\}= \bigcap_{n=1}^\infty \f_{\om^n} (X_{t(\om)}).$$ 
Note that if $\om \in E_A^\ast$ then there exists at least one $\om' \in E_A^\ast$ such that $\om'|_{|\om|}=\om$ and $\om' \om' \in E_A^\ast$. Indeed by the finite irreducibility of $\cS$ there exists some $\rho \in E_A^\ast$ such that $\om \rho \om \in E_A^\ast$. Hence, if  $\om'=\om \rho$ then $\om' \om' \in E_A^\ast$. Therefore the fixed points of $\cS$ form a dense subset of $J_{\cS}$.

By the Leibniz rule for every $n \in \N$,
\begin{equation}
\label{leibfixed}
\|D \f_{\om^n}(x_\om)\|=\|D \f_{\om}(x_\om)\|^n,
\end{equation}
 and 
\begin{equation}
\label{bdpfixed}
\|D \f_{\om^n}\|_{\infty} \overset{\eqref{leibfixed} \wedge \eqref{bdp}}{\leq} K \,\|D \f_{\om}(x_\om)\|^n.
\end{equation}

Let $r<\min \{\diam (X_v):v \in V\}$ and let $n \in \N$ be the largest natural number such that,
\begin{equation}
\label{firstninc}\f_{\om^n} (X_{t(\om)}) \subset B(x_\om,r).
\end{equation}
Therefore, there exists some $z \in \f_{\om^{n-1}} (X_{t(\om)}) \stm B(x_\om,r)$. Hence, $d(x_\om,z)\geq r$ and since $x_\om \in  \f_{\om^{n-1}}(X_{t(\om)})$ we also have that 
\begin{equation}
\label{prebdpn-1}
\diam(\f_{\om^{n-1}}(X_{t(\om)})) \geq r.
\end{equation}
Moreover, assuming without loss of generality that $\diam(X) \leq 1$,
\begin{equation}
\begin{split}
\label{bdpn-1}
\diam(\f_{\om^{n-1}}(X_{t(\om)})) &\overset{\eqref{boundbyder}}{\leq}   \|D \f_{\om^{n-1}}\|_{\infty} \overset{\eqref{bdpfixed}}{\leq}  K \|D \f_{\om} (x_{\om})\|^{n-1} \\
&= \frac{K}{ \|D \f_{\om} (x_{\om})\|} {\|D \f_{\om} (x_{\om})\|^{n}}.
\end{split}
\end{equation}
By Lemma \ref{eyp} \ref{lemmax12},
\begin{equation}
\label{firstpor1}
\begin{split}
\f_{\om^n} (B_{t(\om)}) &\overset{\eqref{4.1.8}}{\supset} B(\f_{\om^n}(z_{t(\om)}), K^{-1} \|D \f_{\om^n}\|_\infty r_{t(\om)}) \\
&\supset B(\f_{\om^n}(z_{t(\om)}), K^{-1} \|D \f_{\om^n} (x_\om)\| r_{t(\om)}) \\
& \overset{\eqref{leibfixed}}{=} B(\f_{\om^n}(z_{t(\om)}), K^{-1} \|D \f_{\om}(x_\om)\|^n r_{t(\om)}) \\
&:=B^r.
\end{split}
\end{equation}
Therefore 
\begin{equation}
\label{porosecthm}
B^r \subset \f_{\om^n} (B_{t(\om)}) \subset \f_{\om^n} (X_{t(\om)}) \overset{\eqref{firstninc}}{\subset} B(x_{\om},r),
\end{equation}
and,
\begin{equation}
\label{firstpor2}J_{\cS} \cap B^r\overset{\eqref{bset1} \wedge \eqref{firstpor1}}{=}\emptyset.
\end{equation}
Moreover, 
\begin{equation}
\label{firstpor3}\frac{\mbox{radius}\,( B^r)}{r} \overset{\eqref{prebdpn-1} \wedge \eqref{bdpn-1} \wedge \eqref{firstpor1}}{\geq} \frac{K^{-1} \|D \f_{\om}(x_\om)\|^n  r_{t(\om)}}{\frac{ K}{ \|D \f_{\om} (x_{\om})\|} {\|D \f_{\om} (x_{\om})\|^{n}}} \geq \frac{r_0  \|D \f_{\om} (x_{\om})\|}{ K^2}.
\end{equation}
Therefore, \eqref{porosecthm}, \eqref{firstpor2} and \eqref{firstpor3} imply that $J_{\cS}$ is $\frac{r_0  \|D \f_{\om} (x_{\om})\|}{ K^2}$-porous at $x_{\om}$. The proof of \ref{porthm2} is complete.

We now move to the proof of  \ref{porthm3a}. Let $\Lambda \subset E$  be a set witnessing finite irreducibility for $E$. Let $F$ be a finite set such that $\Lambda \subset F \subset E$, and let
$$m_{F}=\min \{ \|D \f_e\|_{\infty}: e \in F\}.$$
Let $\om \in F^\N_A$ and $x=\pi(\om)$. Let $r<\min \{\diam (X_v):v \in V\}$ and let $n \in \N$ be the largest natural number such that  
\begin{equation}
\label{firstpor4}
\f_{\om|_n} (X_{t(\om_n)}) \subset B(x,r).
\end{equation}
By Lemma \ref{eyp} we have that 
\begin{equation}
\label{bdpfomj}
\f_{\om|_n} (B_{t(\om_n)}) \supset B( \f_{\om|_n}(z_{t(\om_n)}), K^{-1} \|D \f_{\om|_n}\|_\infty r_{t_{\om_n}}):=B^r_F.
\end{equation}
By \eqref{bset1},
$$J_{\cS} \cap B_F^r=\emptyset.$$
Arguing as in \eqref{prebdpn-1} and \eqref{bdpn-1} we also get that
$$r \leq \diam(\f_{\om|_{n-1}}(X_{t(\om_{n-1})})) \leq \frac{  K}{m_F} \|D \f_{\om|_n} \|_\infty.$$
Therefore, 
\begin{equation}
\label{firstpor5}\frac{\mbox{radius}\,( B_F^r)}{r} \geq \frac{K^{-1} \|D \f_{\om|_n} \|_\infty  r_{t(\om)}}{\frac{ K}{m_F} \|D \f_{\om|_n} \|_\infty} \geq \frac{r_0  m_F}{ \   K^2}.
\end{equation}
Hence \eqref{firstpor4}, \eqref{bdpfomj} and \eqref{firstpor5} imply that $J_{\cS}$ is $\frac{r_0  m_F}{ \   K^2}$-porous at every $x \in \pi( F^{\N})=J_{\cS_F}$. Now \ref{porthm3a} follows by Theorem \ref{721}.

Finally \ref{porthm3b} follows easily from \ref{porthm3a}. Take
$$\tilde{J}= \cup \{J_F : F \mbox{ is finite and irreducible}\}.$$
We then have that $\tilde{J} \subset J_{\cS}$ and by \ref{porthm3a} $J_{\cS}$ is porous at every point of $\tilde{J}$. Moreover by Theorem \ref{721}, $\dim_{\cH} (J_{\cS})=\dim_{\cH} (\tilde{J})$. The proof is complete.
\end{proof}

\subsection{Mean porosity}
\label{sec:mean}


In this section we will investigate \textit{mean porosity} for graph directed Markov systems. Recall \eqref{porexr} for the formal definition of mean porosity which we state below. 
\begin{defn} 
\label{meanpordfn} Let $(X,d)$ be a metric space. 
The set $E \subset X$ is {\em$(\a,p)$-mean porous at $x$} if
$$
\begin{aligned}
S\underline A(E,x;\a)
:&=\sup_{s\in(0,1/2]}\sup_{\beta > \alpha}\left\{\liminf_{i \ra +\infty} \frac{\sharp \big\{j\in [1,i] \cap \N: \por(E,x,s2^{-j}) > \beta\big\}}{i}\right\} \\
&=\sup_{s\in(0,+\infty)}\sup_{\beta > \alpha}\left\{\liminf_{i \ra +\infty} \frac{\sharp \big\{j\in [1,i] \cap \N: \por(E,x,s2^{-j}) > \beta\big\}}{i}\right\}
> p.
\end{aligned}
$$
Denote also, for future reference,
$$
\underline A(E,x;\beta,s)
:=\liminf_{i \ra +\infty} \frac{\sharp \big\{j \in [1,i] \cap \N: \por(E,x,s2^{-j})> \beta\big\}}{i}.
$$
\end{defn}

Obviously, all the quantities and concepts introduced in the above definition are invariant under isometries of the ambient metric space $(X,d)$. It is also fairly obvious that these are invariant under all similarity self-maps of $X$.

\begin{lm}
\label{porconfinv} Let $X \subset \Rn$ for some $n \in \N$. Let $\xi \in X$ and suppose that $W$ is an open connected neighborhood of $\xi$ and $\f: W \ra \Rn$ is a $C^1$ conformal diffeomorphism from $W$ onto $\f(W)$. If $X$ is $(\a, p)$-mean porous at $\xi$ in $\Rn$ and $Y \subset \Rn$ is any set such that $\f(\xi) \in Y$ and $Y \cap \f(W) \subset \f(X \cap W)$ then $Y$ is $(\a,p)$-mean porous at $\f(\xi)$.
\end{lm}
\begin{proof} Since $W$ is an open set containing $\xi$ and since $\f(W)$ is an open set containing $\f(\xi)$, there exists $i_1 \geq 0$ such that 
$$B(\xi, 2^{-j}) \subset W \mbox{ and }B(\f(\xi),2^{-j}) \subset \f(W)$$
for all integers $j \geq i_1$. Now suppose that $s \in (0,1/2 ]$ and $ \beta > \alpha$ are such that $\underline A (X, \xi; \beta, s)>p$. Fix $\ve>0$ so small that $(1+\ve)^{-2} \beta > \a$. Since $\f:W \ra \f(W)$ is a $C^1$ conformal diffeomorphism from $W$ onto $\f(W)$ (in particular $\f'(\xi) \neq 0$)  then there exists $i_2 \geq i_1$ such that
\begin{equation}
\label{porconfinv1}
B(\f(a), (1+\ve)^{-1} |\f'(\xi)|r) \subset \f (B(a,r)) \subset B(\f(a), (1+\ve) |\f'(\xi)|r)
\end{equation}
and
\begin{equation}
\label{porconfinv2}
B(\f^{-1}(b), (1+\ve)^{-1} |\f'(\xi)|^{-1}r) \subset \f^{-1} (B(b,r)) \subset B(\f^{-1}(b), (1+\ve) |\f'(\xi)|^{-1}r)
\end{equation}
for every $a \in B(\xi, 2^{-i_2}), b \in B(\f(\xi),2^{-i_2})$ and $r \in (0, 2^{-i_2})$. Take then $j \geq i_2$ such that $\por(X, \xi, s 2^{-j})>\beta$. Then there exists $z \in \Rn$ such that $B(z, \beta s 2^{-j}) \subset B(\xi, s 2^{-j}) \stm X$. Since $Y \cap \f(W) \subset \f(X \cap W)$ we have that
\begin{equation}
\label{porconfinv3}
\begin{split}
\f(B(z, \beta s 2^{-j})) &\subset \f(B(\xi, s 2^{-j}) \stm X) \subset \f(B(\xi, s 2^{-j})) \stm Y  \\
&\overset{\eqref{porconfinv1}}{\subset} B(\f(\xi), (1+\ve) |\f'(\xi)|s 2^{-j}) \stm Y.
\end{split}
\end{equation}
Moreover,
\begin{equation}
\label{porconfinv4}\f(B(z, \beta s 2^{-j}))\overset{\eqref{porconfinv1}}{\supset} B(\f(z), (1+\ve)^{-1} |\f'(\xi)| \beta s 2^{-j}).\end{equation}
Therefore,
\begin{equation}
\label{porconfinv5}B(\f(z), (1+\ve)^{-1} |\f'(\xi)| \beta s 2^{-j}) \overset{\eqref{porconfinv3} \wedge \eqref{porconfinv4}} {\subset}B(\f(\xi), (1+\ve) |\f'(\xi)|s 2^{-j}) \stm Y,
\end{equation}
and consequently
$$\por(Y, \f(\xi),(1+\ve) |\f'(\xi)|s 2^{-j}) \geq (1+\ve)^{-2} \beta.$$
Therefore,
$$\underline{A}(Y, \f(\xi);(1+\ve)^{-2}\beta, (1+\ve) |\f'(\xi)|s) \geq \underline{A}(X,\xi;\beta,s)>p.$$
Therefore
$$S\underline{A}(Y, \f(\xi);\a) \geq \underline{A}(Y, \f(\xi);(1+\ve)^{-2}\beta, (1+\ve) |\f'(\xi)|s) >p,$$
and the proof is complete.
\end{proof}

\medskip We  are now ready to prove Theorem \ref{meanporthminto} which we restate in a more precise manner. We also record that if $\cS=\{\f_e\}_{e \in E}$ is a finitely irreducible conformal GDMS and $\mu$ is a Borel probability shift-invariant ergodic measure on $E_A^\N$,  the {\it characteristic Lyapunov exponent} with respect to
$\mu$ and $\sigma$ is defined as
$$
\chi_\mu(\sg)=-\int_{E_A^\N}\log\|D\phi_{\om_1}(\pi(\sg(\om))\| d\mu (\om).
$$
Note that $\chi_\mu(\sg)>0$. 

\begin{thm}
\label{meanporthm} 
Let $\cS=\{\f_e\}_{e \in E}$ be a finitely irreducible conformal GDMS which satisfies
\eqref{xx1}. Let $\tilde{\mu}$ be any Borel probability $\sigma$-invariant ergodic measure on $E_A^\N$ with 
finite Lyapunov exponent  $\chi_{\tilde{\mu}}(\sigma)$. Then there exists some $\a_{\cS} \in (0,1/2)$ such that 
$$
\underline A\big(J_{\cS},x;2\a_{\cS},1\big)
\ge \frac{\log 2}{\chi_{\tilde{\mu}} (\sigma)}.
$$
for $\tilde{\mu} \circ \pi^{-1}$-a.e. $x\in J_{\cS}$. In consequence,
$$
S\underline A\big(J_{\cS},x;\a_{\cS}\big)
\ge \frac{\log 2}{\chi_{\tilde{\mu}} (\sigma)}
$$
for $\tilde{\mu} \circ \pi^{-1}$-a.e. $x\in J_{\cS}$ and the set $J_{\cS}$ is $(\a_{\cS},p)$--mean porous for every 
$$
p<\frac{\log 2}{\chi_{\tilde{\mu}} (\sigma)}
$$
at $\tilde{\mu} \circ \pi^{-1}$-a.e. $x\in J_{\cS}$.
\end{thm}

\begin{proof}
Let $\{B(z_v,r_v)\}_{v \in V}$  as  in Lemma \ref{lemmax1} and recall \eqref{bsubsetofint}. Let $\om \in E_A^\N$ and set
$$w_j=\f_{\om|_j}(z_{t(\om_j)}).$$
By Lemma \ref{lemmax1} \ref{lemmax12} we know that 
\begin{equation}
\label{meanpor3}\f_{\om|_j} (B_{t(\om_j)}) \cap J_{\cS} = \emptyset\end{equation}
for all $j \in \N$. Without loss of generality assume again that $r_0=\min_{v \in V} {r_v}<\eta_{\cS}$. By Lemma \ref{eyp},
\begin{equation}
\label{meanpor4}
\begin{split}\f_{\om|_j} (B_{t(\om_j)}) &\supset B( \f_{\om|_j}(z_{t(\om_j)}), K^{-1} \|D \f_{\om|_j}\|_\infty r_{t(\om_j)}) \\
&\supset B( w_j, K^{-1} \|D \f_{\om|_j}\|_\infty r_0).
\end{split}
\end{equation}
Hence, for all $j \in \N$,
\begin{equation}
\label{meanpor5}\dist(w_j, J_{\cS}) \overset{\eqref{meanpor3} \wedge \eqref{meanpor4} }{\geq} K^{-1} \|D \f_{\om|_j}\|_\infty r_0 .
\end{equation}
Moreover, since without loss of generality we can assume that $\diam(\cup_{v \in V} X_v) \leq 1$, we have that
\begin{equation}
\label{meanpor6}
d(w_j, \pi (\om))=d(\f_{\om|_j}(z_{t(\om_j)}), \f_{\om|_j}(\pi(\sigma^j(\om)))) \overset{\eqref{boundbyderlip}}{\leq}  \|D \f_{\om|_j}\|_\infty .
\end{equation}
Let
$$\a:=\min \left\{ \frac{r_0}{2  \,K }, \frac{1}{2} \right\}.$$
Then 
$$B(w_j, 2  \a \|D \f_{\om|_j}\|_\infty) \cap J_{\cS}\overset{(\ref{meanpor3}) \wedge (\ref{meanpor4})}{=}\emptyset$$
and 
$$B(w_j, 2  \a \|D \f_{\om|_j}\|_\infty) \overset{\eqref{meanpor6}}{\subset} B(\pi(\om), 2  \|D \f_{\om|_j}\|_\infty).$$
Therefore, for all $j \in \N$,
\begin{equation}
\label{meanpor7}
B(w_j, 2  \a \|D \f_{\om|_j}\|_\infty) \subset B(\pi(\om), 2  \|D \f_{\om|_j}\|_\infty) \stm J_{\cS}.
\end{equation}
Let $(n_j)_{j \in \N}, n_j \in \R,$ such that
$$2 \|D \f_{\om|_j}\|_\infty=2^{-n_j},$$
that is
$$n_j=-\frac{\log (2  \|D \f_{\om|_j}\|_\infty)}{\log 2}.$$
Since $\lim_{j \ra \infty}\|D \f_{\om|_j}\|_\infty =0$, see e.g. \cite[Lemma 4.18]{CTU}, there exists some some $j_0 \in \N$ such that $n_j >0$ for all $j \geq j_0$, and  $n_j \ra \infty$. Moreover, $n_j$ is increasing by \eqref{quasi-multiplicativity1}. Notice that for $j\geq j_0$, \eqref{meanpor7} implies that
$$
\por (J_{\cS},\pi(\om), 2^{-n_j}) \geq \alpha.
$$
It then follows that for all $j \geq j_0$
$$
\por (J_{\cS},\pi(\om), 2^{-{\left \lfloor{n_j}\right \rfloor}}) \geq \alpha/2,
$$
where ${\left \lfloor{x}\right \rfloor}$ denotes the greatest integer less than or equal to $x$. Therefore
\begin{equation}
\begin{split}
\label{meanpor9}
\underline A\big(J_{\cS},\pi(\om);\a/3,1)&\ge 
\liminf_{i \ra +\infty} \frac{\sharp \{j \in [1,i] \cap \N: \por(J_{\cS},\pi(\om),2^{-j}) \geq \a/2\}}{i} \\
&\geq \liminf_{j \ra +\infty} \frac{j}{{\left \lfloor{n_j}\right \rfloor} }.
\end{split}
\end{equation}

Since $\tilde{\mu}$ is $\sigma$-invariant and ergodic, Birkhoff's ergodic theorem implies that for $\tilde{\mu}$-a.e. $\om \in E_A^\N$,
\begin{equation}
\label{mporobet}-\lim_{n \ra +\infty} \frac{1}{n} \sum_{k=0}^{n-1} \zeta  \circ \sigma^k (\om)=-\int_{E_A^\N} \zeta d \tilde{\mu}=\chi_{\tilde{\mu}} (\sigma),
\end{equation}
where $\zeta= \log \|D \f_{\om_1}(\pi(\sigma(\om)))\|$. Notice that for $\om \in E^\N_A$,
\begin{equation*}\begin{split}
\sum_{k=0}^{n-1} - \zeta \circ \sg^k(\om)&=-\sum_{k=0}^{n-1} \log \|D \f_{\om_{k+1}}(\pi(\sg^{k+1}(\om)))\|\\
&=-\log\left(\|D \f_{\om_{1}}(\pi(\sg^{1}(\om)))\|\,\|D \f_{\om_{2}}(\pi(\sg^{2}(\om)))\|\cdots \|D \f_{\om_{n}}(\pi(\sg^{n}(\om)))\| \right).
\end{split}\end{equation*}
By the Leibniz rule and the fact that  $$\pi(\sg^m(\om))= \f_{\om_{m+1}}\circ \dots \circ \f_{\om_n}(\pi(\sg^n(\om)))\mbox{ for }1 \leq m \leq n,$$ we deduce that
\begin{equation}
\label{logleib}
\sum_{k=0}^{n-1} - \zeta \circ \sg^k(\om)=-\log \|D \f_{\om|_n}(\pi(\sg^n(\om)))\|.
\end{equation}
By \eqref{bdp} we have that for every $j \in \N$,
\begin{equation}
\label{kmder}K^{-1}\|D \f_{\om|_j}\|_\infty\leq \|D \f_{\om|_j}(\pi(\sg^j(\om)))\| \leq \|D \f_{\om|_j}\|_\infty.
\end{equation}
Therefore for $\tilde{\mu}$-a.e. $\om \in E_A^\N$,
\begin{equation}
\label{meanporo10}
\chi_{\tilde{\mu}} (\sigma)\overset{\eqref{mporobet} \wedge \eqref{logleib}}{=}\lim_{j \ra +\infty} -\frac{ \log \|D \f_{\om|_j}(\pi(\sg^j(\om)))\|}{j}\overset{\eqref{kmder}}{=}\lim_{j \ra +\infty} -\frac{ \log \|D \f_{\om|_j}\|_\infty}{j}.
\end{equation}
Hence
\begin{equation*}
\begin{split}
\lim_{j \ra +\infty} \frac{n_j}{j} &=\lim_{j \ra +\infty} \frac{-\log (2  \|D \f_{\om|_j}\|_\infty)}{j \log 2} \\
&=\frac{1}{ \log 2} \lim_{j \ra +\infty} \frac{- \log \|D \f_{\om|_j}\|_\infty}{j}\overset{\eqref{meanporo10}}{=}\frac{\chi_{\tilde{\mu}} (\sigma)}{\log 2}.
\end{split}
\end{equation*}
Moreover since by our assumption $\chi_{\tilde{\mu}} (\sigma) \in (0,+\infty)$,
\begin{equation}
\label{meanporo11}\liminf_{j \ra +\infty} \frac{j}{{\left \lfloor{n_j}\right \rfloor} } \geq \lim_{j \ra +\infty} \frac{j}{n_j}=\frac{\log 2}{\chi_{\tilde{\mu}} (\sigma)}.
\end{equation}
Combining \eqref{meanpor9} and \eqref{meanporo11} we deduce that 
$$
\underline A\big(J_{\cS},\pi(\om);\a/3,1)\ge \frac{\log 2}{\chi_{\tilde{\mu}} (\sigma)}.
$$
The proof is complete by choosing $\a_{\cS}:=\a/6$.
\end{proof}
We will now present a corollary of Theorem \ref{meanporthm} which asserts that the limit set of a strongly regular finitely irreducible GDMS is almost everywhere mean porous with respect to its conformal measure.

\begin{coro}
\label{meanporoconfhaus}
Let $\cS=\{\f_e\}_{e \in E}$ be a strongly regular finitely irreducible conformal GDMS which satisfies \eqref{xx1}.
Then 
\begin{enumerate}[label=(\roman*)]
\item\label{mpch1}
$$
S\underline A\big(J_{\cS},x;\a_{\cS}\big)
\ge\underline A\big(J_{\cS},x;2\a_{\cS},1\big)
\ge \frac{\log 2}{\chi_{\tilde{\mu}_h} (\sigma)}
$$
for $m_h$-a.e. $x\in J_{\cS}$, and the set $J_{\cS}$ is $(\a_{\cS},p)$--mean porous for every 
$$
p<\frac{\log 2}{\chi_{\tilde{\mu}_h} (\sigma)}
$$
at $m_h$-a.e. $x\in J_{\cS}$.

\item\label{mpch2}
$$
S\underline A\big(J_{\cS},x;\a_{\cS}\big)
\ge\underline A\big(J_{\cS},x;2\a_{\cS},1\big)
\ge \frac{\log 2}{\chi_{\tilde{\mu}_h} (\sigma)}
$$
for $\cH^h$-a.e. $x\in J_{\cS}$, and the set $J_{\cS}$ is $(\a_{\cS},p)$--mean porous for every 
$$
p<\frac{\log 2}{\chi_{\tilde{\mu}_h} (\sigma)}
$$
at $\cH^h$-a.e. $x\in J_{\cS}$,
\end{enumerate}
where $\a_{\cS}$ is as in Theorem \ref{meanporthm}, $h=\dim_{\cH}(J_{\cS})$,  $m_{h}$ is the $h$-conformal measure of $\cS$, and $\tilde{\mu}_h$ is the unique shift invariant ergodic Gibbs state globally equivalent to $\tilde{m}_h$. \end{coro}

\begin{proof} 
We will first show that if $\cS$ is strongly regular then $\chi_{\tilde{\mu}_h}(\sigma)$, the Lyapunov exponent of the unique ergodic shift invariant Gibbs state $\tilde{\mu}_h$,  is finite. Since $\cS$ is strongly regular there exists some $t>0$ such that $P(t) \in (0,\infty)$. It follows then, for example by \cite[Proposition 7.5]{CTU}, that $P$ is continuous and decreasing on $[t,+\infty)$. Hence, there exists some $\eta>0$ such that $P(h-\eta)<\infty$. Thus, \eqref{presz1} implies that $Z_1(h-\eta)<\infty$. Equivalently, 
\begin{equation}
\label{zheta}
\sum_{e \in E} \|D \f_e\|_{\infty}^{h-\eta}<\infty.
\end{equation}
Observe that for all but finitely many $e \in E$,
\begin{equation}
\label{dfelog}
\|D \f_e\|_{\infty}^{-\eta} \geq -\log(\|D \f_e\|_{\infty}).
\end{equation}
Indeed, if \eqref{dfelog} was false we could find infinitely many $(e_n)_{n \in \N}, e_n \in E,$ such that 
$$\|D \f_{e_n}\|_{\infty}^{\eta} \log (\|D \f_{e_n}\|_{\infty}^{-1})>1.$$
Nevertheless, this is impossible because by \cite[Lemma 4.18]{CTU} we know that $ \|D \f_{e_n}\|_{\infty} \ra 0$ as $n \ra \infty$, while $\lim_{x \ra 0^+} x^{\eta} \log(1/x)=0$. Therefore,
\begin{equation}
\label{zheta2}
\sum_{e \in E} -\log(\|D \f_e\|_{\infty}) \|D \f_e\|_{\infty}^{h}\overset{\eqref{zheta}\wedge \eqref{dfelog}}{<}\infty.
\end{equation}

We can now estimate $\chi_{\tilde{\mu}_h}(\sigma)$:
\begin{equation*}
\begin{split}
\chi_{\tilde{\mu}_h}(\sigma)&=-\int_{E_A^\N}\log\|D\phi_{\om_1}(\pi(\sg(\om))\| d\tilde{\mu}_h(\om)\\
&=\sum_{e \in E} \int_{[e]}-\log\|D\phi_{\om_1}(\pi(\sg(\om))\| d\tilde{\mu}_h(\om)\\
&\overset{\eqref{bdp}}{\leq} \log K\, \sum_{e \in E} \tilde{\mu}_h([e])+ \sum_{e \in E} (-\log(\|D \f_e\|_{\infty}) \tilde{\mu}_h([e])\\
&=\log K \, \tilde{\mu}_h(E_A^\N)+\sum_{e \in E} (-\log(\|D \f_e\|_{\infty}) \tilde{\mu}_h([e]).
\end{split}
\end{equation*}
Recall that $\tilde{\mu}_h$ is probability measure on $E_A^\N$ (as a Gibbs state), hence in order to show that $\chi_{\tilde{\mu}_h}(\sigma)<\infty$ it suffices to show that 
\begin{equation}
\label{zheta3}
\sum_{e \in E} (-\log(\|D \f_e\|_{\infty}) \tilde{\mu}_h([e])<\infty.
\end{equation}
Since $\cS$ is strongly regular, and thus regular by \eqref{cofimpliesreg}, it is well known, see e.g. \cite[Equation 7.18]{CTU}, that there exists some constant $c_h \geq 1$ such that for all $\om \in E_A^\ast$,
\begin{equation}
\label{zheta4}c_h^{-1} \|D \f_\om\|_\infty \leq \tilde{\mu}_h([\om]) \leq c_h \|D \f_\om\|_\infty.
\end{equation}
Thus, \eqref{zheta3} follows by \eqref{zheta2} and \eqref{zheta4}.

As we showed that $\chi_{\tilde{\mu}_h}(\sigma)<\infty$,  \ref{mpch1} follows from Theorem \ref{meanporthm}  and the fact that $\tilde{\mu}_h$ is globally equivalent to the conformal measure $\tilde{m}_h$. Finally \ref{mpch2} follows from \ref{mpch1} and \cite[Theorem 10.1]{CTU}.
\end{proof}

\subsection{Directed porosity}
\label{sec:dirpor}
We now turn our attention to directed porosity, whose formal definition we provide below. Given $v \in S^{n-1}$ we will denote by $l_v$ the line in $\Rn$ containing the origin and  $v$. For $E \subset \R^n$, $x \in \Rn$ and $r \in (0,\diam(E))$ we let
\begin{equation}
\label{porvexr}
\por_v (E,x,r)=\sup \{c \geq 0: B(y,c r) \subset B(x,r)\stm E \mbox{ for some }y\in x+l_v\}.
\end{equation}
\begin{defn}
Let $E \subset \Rn$ be a bounded set. 
Given $v \in S^{n-1}, c\in (0,1)$ and $x \in \Rn$, we say that $E$ is:
\begin{enumerate}[label=(\roman*)]
\item   {\em $v$-directed $c$-porous at} $x$ if there exists some $r_0>0$ such that $\por_v (E,x,r) \geq c $ for every $r \in (0, r_0)$, 
\item {\em$v$-directed porous at} $x$ if there exists some $c \in (0,1)$ such that $E$ is $v$-directed $c$-porous at $x$,
\item $E$ is  {\em $v$-directed $c$-porous} if there exists some $r_0>0$ such that $\por_v (E,x,r) \geq c$ for every $x \in E$ and $r \in (0, r_0
)$,
\item {\em $v$-directed porous} if it is  $v$-directed $c$-porous for some $c\in (0,1)$.
\end{enumerate}
\end{defn}

Our first theorem in this section provides a sufficient condition for a  finite and irreducible conformal GDMS to be directed porous at every direction.
 \begin{thm}\label{dirporlimset} Let $\cS=\{\f_{e}\}_{e \in E}$ be a finite and irreducible conformal GDMS in $\R^n$ such that
\begin{equation}
\label{lsawayfrombdry}
\dist(\partial X, J_{\cS} ):=R>0.
\end{equation}
Then $J_{\cS}$ is $v$-directed porous for every $v \in S^{n-1}$.
 \end{thm}

 \begin{proof} Let $v \in S^{n-1}$, $\om \in E_A^\N$ and $r>0$. Let $k$ be the smallest integer such that 
 \begin{equation}
 \label{smallestk}
 \f_{\om|_k} (X_{t(\om_k)}) \subset B(\pi(\om),r).
 \end{equation} 
 Now let $I_{v} \subset \f_{\om|_k} (X_{t(\om_k)})$ be a line segment at direction $v$, joining $\pi(\om)$ to $\partial \f_{\om|_k} (X_{t(\om_k)})$. Then there exists some $y \in \partial X_{t(\om_k)}$ such that
 \begin{equation*}
 \text{length}(I_{v}) \geq \dist (\pi(\om), \partial \f_{\om|_k} (X_{t(\om_k)}) )= |\pi(\om)-\f_{\om|_k}(y)|.
 \end{equation*}
 Hence, by \cite[Lemma 4.14]{CTU} there exists some constant $c>0$ such that
 \begin{equation*}
 \begin{split}
\text{length}(I_{v}) &\geq |\f_{\om|_k}(\pi( \sigma^k(\om))-\f_{\om|_k}(y)| \\
& \geq c \| D \f_{\om|_k}\|_{\infty} |\pi( \sigma^k(\om))-y| \geq c \| D \f_{\om|_k}\|_{\infty} R>0.
\end{split}
\end{equation*}
Note that  $\gamma_v:=\f_{\om|_k}^{-1} (I_v) \subset X_{t(\om_k)}$ is a smooth curve joining $\pi( \sigma^k(\om))$ to $\partial X$. In particular $\gamma_v$ joins $\pi( \sigma^k(\om))$ to $\partial X_{t(\om_k)}.$ Therefore there exists $z \in \gamma_v$ such that 
\begin{equation}
\label{pointinthevoid}
B(z, R/4) \subset X_{t(\om_k)} \subset X \mbox{ and }B(z, R/4) \cap  J_{\cS}=\emptyset,
\end{equation}
and consequently
\begin{equation}
\label{jsvoid}
\f_{\om|_k}(B(z, R/4)) \cap J_{\cS}=\emptyset.
\end{equation}
Let $R':=\min\{R/4,\eta_{\cS}/2\}$, where $\eta_{\cS}$ was defined in  \eqref{itaS}. Thus
 \begin{equation}
 \label{ballvoid}
 \begin{split}
B(\f_{\om|_k}(z), K^{-1} \|D\f_{\om|_k} \|_\infty R') &\overset{(\ref{4.1.8})}{\subset} \f_{\om|_k}(B(z, R'))  \\
&\overset{(\ref{pointinthevoid})}{\subset} \f_{\om|_k}(X_{t(\om_k)}) \overset{(\ref{smallestk})}{\subset} B(\pi(\om),r).
\end{split}
\end{equation}
Therefore, 
 \begin{equation}
 \label{seppor}
J_{\cS} \cap B(\f_{\om|_k}(z), K^{-1} \|D\f_{\om|_k} \|_\infty R')\overset{(\ref{jsvoid})\wedge(\ref{ballvoid})}{=}\emptyset.
\end{equation}
We also record that 
\begin{equation}
\label{pointinthedir}
\f_{\om|_k}(z) \in I_v \subset   \f_{\om|_k} (X_{t(\om_k)}) \subset B(\pi(\om),r).
\end{equation}
Without loss of generality we can assume that $\diam(X)=1$. So if $m_0=\min \{ \| D \f_e\|_\infty: e \in E\}$ we have that
\begin{equation}
\begin{split}
\label{rbdd} 
r \leq \diam (\f_{\om|_{k-1}}(X_{t(\om_{k-1})}))  & \overset{(\ref{boundbyder})}{\leq}  \frac{1}{m_0} \|D \f_{\om|_{k-1}}\|_\infty \|D \f_{\om_k}\|_\infty  \overset{(\ref{quasi-multiplicativity1})}{\leq} \frac{ K }{m_0}  \|D \f_{\om|_{k}}\|_\infty.
 \end{split}
 \end{equation} 
 Hence
  \begin{equation}
 \label{seppor2}
J_{\cS} \cap B(\f_{\om|_k}(z), m_0 K ^{-2} R' \, r)\overset{(\ref{seppor})\wedge(\ref{rbdd})}{=}\emptyset.
\end{equation}
Combining \eqref{seppor2} and \eqref{pointinthedir} we deduce that $J_{\cS}$ is $v$-directed $m_0 K ^{-2} R'$-porous at $\pi(\om).$ The proof is complete.
 \end{proof}
 

  

We will now see how we can use Theorem \ref{dirporlimset}   to show that if a finite and irreducible conformal GDMS satisfies the strong separation condition (recall Remark \ref{ssc}) then it is directed porous at every direction. Before doing so, we recall the notion of \textit{equivalent} graph directed Markov systems which was introduced in \cite{CTU}. 
\begin{definition}
\label{equivgdms}
\index{GDMS!equivalent}
Two  conformal GDMS $\cS$ and $\cS'$ are called {\it equivalent} if:
\begin{itemize}
\item[(i)] they share the same associated directed multigraph $(E,V)$,
\item[(ii)] they have the same incidence matrix $A$ and the same functions $i,t:E \ra V$,
\item[(iii)]  they are defined by the same set of conformal maps $\{\f_e : W_{t(e)} \ra W_{i(e)}\}$, where $W_v$ are open connected sets, and for every $v \in V$, $X_v \cup X'_v \subset W_v$.
\end{itemize}
\end{definition}
\begin{lm} 
\label{sscbdry}Let $\cS=\{V,E,A,t,i, \{X_v\}_{v \in V}, \{\f_e\}_{e \in E} \}$ be a finite conformal GDMS  which satisfies the strong separation condition. Then there exists a finite conformal $\cS'=\{V,E,A,t,i, \{X'_v\}_{v \in V}, \{\f_e\}_{e \in E} \}$ which is equivalent to $\cS$ and 
$$\dist(\partial X', \cup_{e \in E} \f_e (X'_{t(e)})):=R>0,$$
where $X'=\cup_{v \in V} X'_v.$
\end{lm}

\begin{proof} Since $\cS$ satisfies the strong separation condition there exists $\ve <\eta_{\cS}$ such that
\begin{equation}
\label{oscnew}
\f_{a} (X_{t(a)} (\ve)) \cap \f_{b} (X_{t(b)}(\ve))=\emptyset
\end{equation}
for all distinct $a,b \in E$, where $A(\ve)=\{x: d(x, A)<\ve\}$ for $A \subset \R^n$. We can also assume that 
$$\max \{ \|D \f_e\|_\infty:e \in E\} =s'<1.$$ We now let
$$X_v'=\overline {X_v(\ve)}.$$
We will show that for all $e \in E$,
\begin{equation}
\label{shrink}
\f_e(X'_{t(e)}) \subset \overline{X_{i(e)} (s' \ve)}.
\end{equation}
For all $p \in X'_{t(e)}$ there exists some $q \in X_{t(e)}$ such that 
$d(p,q)\leq \ve$. Hence, by the mean value theorem, if $\f_e(p) \notin X_{i(e)}$,
$$d(X_{i(e)}, \f_e(p)) \leq d(\f_e(X_{t(e)}), \f_e(p)) \leq |\f_e(q)- \f_e(p)| \leq \|D \f_e \|_\infty |p-q| \leq s' \ve.$$
Therefore, \eqref{shrink} follows.

Hence, \eqref{oscnew} and \eqref{shrink} imply that $\cS'=\{V,E,A,t,i, \{X'_v\}_{v \in V}, \{\f_e\}_{e \in E} \}$ is a conformal GDMS and moreover  \eqref{shrink} implies that
$$\dist(\partial X', \cup_{e \in E} \f_e (X'_{t(e)})) \geq (1-s') \ve >0.$$
The proof is complete.
\end{proof}

 \begin{thm}
 \label{dirporssc}
 Let $\cS=\{\f_{e}\}_{e \in E}$ be a finite and irreducible conformal GDMS in $\R^n$ which satisfies the strong separation condition. Then $J_{\cS}$ is $v$-directed porous for every $v \in S^{n-1}$.
\end{thm}
\begin{proof}
Note that if two conformal GDMS are equivalent then they generate the same limit set. Therefore Theorem \ref{dirporssc} follows by Theorem \ref{dirporlimset} and Lemma \ref{sscbdry}.
\end{proof}
In the next theorem we consider systems where \eqref{lsawayfrombdry}  does not necessarily hold. We prove that if an IFS consists of rotation free similarities and there exist directions $v \in S^{n-1}$ such that the lines $l_v+x$ miss all the first iterations in the interior of the set $X$, then the limit set is $v$-directed porous.  

\begin{thm}
\label{rotfreess}
 Let $\cS=\{\f_e:X \ra X\}_{e \in E}$ be a finite CIFS  consisting of rotation free similarities. Suppose that there exists a set $V \subset S^{n-1}$ such that
\begin{equation}
\label{linecond}
(\Int(X) \cap (l_v+x)) \stm \bigcup_{e \in E} \f_e(X) \neq \emptyset
\end{equation}
for every  $ x \in \bigcup_{e \in E} \f_e(X)$ and $v \in S^{n-1} \stm V$. Then $J_{\cS}$ is $v$-directed porous for every $v \in S^{n-1} \stm V$ and every $x \in J_{\cS}$.
\end{thm}
\begin{proof} Let $v \in S^{n-1} \stm V$. We will first show that there exists some $\eta>0$ such that for all $x \in \cup_{e \in E } \f_e(X)$ there exists some $z_x \in l_v+x$ such that
\begin{equation}
\label{dirgap0}
B(z_x, \eta)  \subset \Int(X) \stm  J_{\cS}.
\end{equation}
If $x \in \cup_{e \in E } \f_e(X)$ let  
$$\eta_x=\sup \{ \theta \geq 0: \exists z \in l_v+x\mbox{ such that }B(z, \theta) \subset \Int(X) \stm J_{\cS}\}.$$
In order to establish \eqref{dirgap0} it suffices to show that 
$$\inf \{\eta_x:x \in \cup_{e \in E } \f_e(X)\}>0.$$
By way of contradiction assume that $\inf \{\eta_x:x \in \cup_{e \in E } \f_e(X)\}=0$. So there exists a positive sequence  $(\eta_n)_{n \in \N}$ such that $\eta_n \ra 0$ and a sequence $(x_n)_{n \in \N}$ in $\cup_{e \in E} \f_e (X)$ such that if $z \in (l_{v}+x_n) \cap \Int(X)$  satisfies $B(z,\theta) \subset \Int(X) \stm  J_{\cS}$, then $\theta \leq \eta_n$.
By compactness, passing to subsequences if necessary, there exists $x \in \cup_{e \in E} \f_e (X)$ such that $x_n \ra x$. By \eqref{linecond} there exists $y_x \in (\Int(X) \cap (l_{v}+x)) \stm \bigcup_{e \in E} \f_e(X)$ and $R>0$ such that 
\begin{equation}
\label{byxr}
B(y_x,R) \subset \Int(X)\stm \bigcup_{e \in E} \f_e(X).
\end{equation}
For every $n \in \N$ let $y_n$ be the unique point in $(l_{v}^{\perp}+y_x) \cap (l_{v}+x_n)$, see Figure \ref{fig:par}. Let $n_0 \in \N$ big enough such that $|x_n-x|<R/4$ for all $n \geq n_0$. Notice that
$$|y_n-y_x| \leq |x_n -x| <R/4$$
for all $n \geq n_0$. Hence  $B(y_n, R/2) \subset B(y_x,R)$ for all $n \geq n_0$. 
Therefore by \eqref{byxr},
$$B(y_n, R/2) \subset \Int(X)\stm \bigcup_{e \in E} \f_e(X), \mbox{ for all }n \geq n_0.$$
Hence, we have found $y_n \in  X \cap (l_{v}+x_n)$ such that
$$B(y_n, R/2) \subset \Int(X) \stm J_{\cS}.$$
But this is a contradiction because $\eta_n \ra 0$ and \eqref{dirgap0} has been proven.

Now let $x= \pi (\om)$ for some $\om \in E^\N$ and fix some $r>0$. Let $n$ be the smallest integer such that 
$$\f_{\om|_n} (X) \subset B(\pi(\om), r).$$
Denote $l_{\om}:=l_v+x$ and let $l_0:=\f_{\om|_n}^{-1} (l_{\om})$. Since $\cS$ consists of rotation free similarities, $l_0$ is a line parallel to $l_{\om}$. Observe that
\begin{equation}
\label{3pisglo}
\pi (\sg^n(\om)) \in l_0 \cap \f_{\om_{n+1}}(X).
\end{equation}
This follows because $\pi(\om)=\f_{\om|_n}(\pi(\sg^n(\om)))$, hence $\pi(\sg^n(\om)) \in \f_{\om|_n}^{-1} (l_{\om})=l_0$. Also by definition of the projection $\pi$,
$$\pi(\sg^n(\om))=\bigcap_{m=1}^\infty \f_{\sg^n(\om)|_{m}}(X) \subset \f_{\om_{n+1}}(X).$$
Hence $l_0=\pi (\sg^n(\om))+l_v$ and by \eqref{dirgap0} there exists some $z \in l_0$ such that
\begin{equation}
\label{etajs}
B(z,\eta) \subset \Int(X) \stm J_{\cS}.
\end{equation}
Without loss of generality we can assume that $\eta < \eta_{\cS}$. Hence
\begin{equation}
\label{bdpdirpor}
B(\f_{\om|_n}(z), K^{-1} \eta \|D \f_{\om|_n}\|_\infty) \overset{(\ref{4.1.8})}{\subset} \f_{\om|_n}(B(z,\eta)).
\end{equation}
We will now show that
\begin{equation}
\label{gap}
\f_{\om|_n}(B(z,\eta)) \cap J_{\cS}=\emptyset.
\end{equation}
By way of contradiction assume that there exists some $y \in B(z, \eta)$ such that $\f_{\om|_n}(y) \in J_{\cS}$. Hence there exists some $\tau \in E_A^\N$ such that 
\begin{equation}
\label{dirporeqstar}
\f_{\om|_n}(y)=\pi(\tau)=\f_{\tau|_n}(\pi (\sg^n(\tau))).
\end{equation}
If $\om|_n=\tau|_n$ then $y=\pi (\sg^n(\tau))$. Hence, $y \in J_{\cS}$ and this contradicts \eqref{etajs}. If $\om|_n \neq \tau|_n$ let $m \leq n$ be the smallest integer such that $\om_m \neq \tau_m$. By the open set condition and the fact that the maps $\f_e$ are homeomorphisms for every $e \in E$, we deduce that
$$\f_{\om_m}(\Int (X)) \cap \f_{\tau_m}(X) = \emptyset,$$
and consequently
$$\f_{\om|_m}(\Int (X)) \cap \f_{\tau|_m}(X) = \emptyset.$$
But this contradicts \eqref{dirporeqstar} because $\f_{\om|_n}(y) \in \f_{\om|_m}(\Int (X))$ and $\pi(\tau) \in \f_{\tau|_m}(X)$. Therefore we have established \eqref{gap}. 

We then have 
\begin{equation}
\label{jsdirporgap}B(\f_{\om|_n}(z), K^{-1} \eta \|D \f_{\om|_n}\|_\infty) \cap J_{\cS}\overset{(\ref{bdpdirpor} \wedge (\ref{gap})}{=}\emptyset.
\end{equation}
Since $z \in l_0 \cap X$ it follows that $\f_{\om|_n}(z) \in l_\om \cap \f_{\om|_n}(X) \subset B(\pi(\om), r)$. Without loss of generality we can assume that $\diam(X) \leq 1$. So, by the choice of $n$, we have that
\begin{equation}
\label{redercomp}
r \leq \diam (\f_{\om|_{n-1}}(X))  \overset{(\ref{boundbyder})}{\leq}  \frac{1}{m_0} \|D \f_{\om|_{n-1}}\|_\infty \|D \f_{\om_n}\|_\infty  \overset{(\ref{quasi-multiplicativity1})}{\leq} \frac{ K }{m_0}  \|D \f_{\om|_{n}}\|_\infty,
 \end{equation}
 where $m_0=\min \{ \| D \f_e\|_\infty: e \in E\}$. Therefore 
 \begin{equation*}
B(\f_{\om|_n}(z),   \frac{ m_0\, \eta}{K^2 }  \, r) \cap J_{\cS}\overset{(\ref{jsdirporgap}) \wedge (\ref{redercomp})}{=}\emptyset.
\end{equation*}
Thus, $J_{\cS}$ is $v$-directed $\frac{ m_0\, \eta}{K^2 }$-porous at   $x$. The proof is complete.
\end{proof} 

\subsection{Non-porosity}
\label{sec:nonpor}
So far, we have been investigating porosity properties of conformal GDMS. Nevertheless, limit sets of conformal GDMS very frequently are non-porous. In the following, we will prove that if the limit set of a finitely irreducible conformal GDMS is not porous at a single point, then it is not porous in a set of full measure.


\begin{thm} 
\label{notporoae} 
Let $\cS=\{\f_e\}_{e \in E}$ be a finitely irreducible conformal GDMS such that $J_{\cS}$ is not porous, or $J_{\cS}$ is not porous at some $\zeta \in \overline{J_{\cS}}$. If $\tilde{\mu}$ is a shift invariant ergodic probability measure on $E_A^\N$  with full topological support, then $J_{\cS}$ is not porous at $\tilde{\mu} \circ \pi^{-1}$ a.e. $x \in J_{\cS}$. 
\end{thm}

\begin{proof} We will first assume that  $J_{\cS}$ is not porous. Let $\ve \in (0,1)$. Since  $J_{\cS}$ is not $\ve$-porous then  there exists $r:=r_\ve \in (0, \eta_{\cS}/2)$ (recall  \eqref{itaS}) and $\xi:=\xi_\ve \in J_{\cS}$ such that $\por(J_{\cS}, \xi, r)< \ve$. Hence, for all $x \in B(\xi, r)$
\begin{equation}
\label{notporo2}
B(x, \ve r) \cap J_{\cS} \neq \emptyset.
\end{equation}

The measure $\tilde{\mu}$ has full topological support hence $
\tilde{\mu}(\pi^{-1} (B(\xi, s)))>0,$ for every $s>0$. Moreover, since $\tilde{\mu}$ is ergodic and Birkhoff's Ergodic Theorem implies that if 
$$Y_{\ve}:=\{\om \in E_A^\N: \sigma^n(\om) \in \pi^{-1}(B(\xi, (4 K )^{-1} r))\mbox{ for infinitely many }n \in \N\},$$
then $\tilde{\mu}(Y_{\ve})=1$.

Let $\om \in Y_{\ve}$ and $n \in \N$ such that $\sigma^n(\om) \in \pi^{-1}(B(\xi, (4  K )^{-1} r))$. Since $\xi \in J_{\cS}$ there exists some $v \in V$ such that $\xi \in X_v$. 
On the other hand $\pi(\sigma^n(\om)) \in X_{t(\om_n)}$ and 
\begin{equation}
\label{neweq}
d(\pi(\sigma^n(\om)), \xi) \leq \frac {r}{4  K} < \frac{\eta_{\cS}}{8  K}.
\end{equation}
By the definition of $\eta_{\cS}$ and the fact that the sets $X_v, v \in V,$ are pairwise disjoint we deduce that
$\xi \in X_{t(\om_n)}$. 
Hence, Lemma \ref{eyp} implies that
\begin{equation}
\label{xiinc}
B(\f_{\om|_n}(\xi), K^{-1} \|D \f_{\om|_n}\|_\infty r) \subset \f_{\om|_n}(B(\xi,r)) \subset   B(\f_{\om|_n}(\xi),  \|D \f_{\om|_n}\|_\infty r).
\end{equation}
Moreover \begin{equation}
\begin{split}
\label{dpomxi}
d(\pi(\om), \f_{\om|_n}(\xi))&=d(\f_{\om|_n}(\pi(\sigma^n(\om))),  \f_{\om|_n}(\xi)) \\
& \overset{\eqref{boundbyderlip}}{\leq} \|D \f_{\om|_n}\|_\infty d (\pi(\sigma^n(\om)), \xi) \overset{\eqref{neweq}}{\leq} \frac{ \|D \f_{\om|_n}\|_\infty r}{8 K}.
\end{split}
\end{equation}
Hence
\begin{equation}
\begin{split}
\label{bpoball}
B(\pi(\om), (2  \, K)^{-1}\|D \f_{\om|_n}\|_\infty r) &\overset{\eqref{dpomxi}}{\subset} B(\f_{\om|_n}(\xi), K^{-1} \|D \f_{\om|_n}\|_\infty r) \\&\overset{\eqref{xiinc}}{\subset} \f_{\om|_n}(B(\xi,r)).
\end{split}
\end{equation}

Let $\om \in Y_{\ve}$. We will show that $J_{\cS}$ is not $\ve'$-porous at $\pi(\om)$ where $\ve'=2  K  \ve$. Let $\delta'>0$ and choose $n \in \N$ big enough such that  
$(2  \, K)^{-1}\|D \f_{\om|_n}\|_\infty r<\delta'$; this is possible for example by \cite[Lemma 4.18]{CTU}. 
Now take $z \in B(\pi(\om), (2  \, K)^{-1}\|D \f_{\om|_n}\|_\infty r)$ and note that by \eqref{bpoball}, $z =\f_{\om|_n}(b)$  for some $b \in B(\xi, r)$. Therefore, \eqref{notporo2} implies that $B(b, \ve r) \cap J_{\cS} \neq \emptyset$. Let $y \in B(b, \ve r) \cap J_{\cS}$. Then, 
we have that $d(y, b) < \ve r< \eta_{\cS}/2 $
and
\begin{equation*}
\begin{split}
d(\f_{\om|_n}(y),  \f_{\om|_n}(b)) \overset{\eqref{boundbyderlip}}{\leq}  \|D \f_{\om|_n}\|_\infty d(y,b) <  \|D \f_{\om|_n}\|_\infty \ve r.
\end{split}
\end{equation*}
Therefore
\begin{equation*}
J_{\cS} \cap B(z,   \|D \f_{\om|_n}\|_\infty\ve r ) \neq \emptyset.
\end{equation*}
Hence $J_{\cS}$ is not $\ve'$-porous at $\pi(\om)$. 

Let $\mu=\tilde{\mu} \circ \pi^{-1}.$ We have shown that for every $\ve \in (0,1)$ there exists $A_\ve:=\pi(Y_{\ve}) \subset J_{\cS}$ such that $\mu (A_\ve)=1$ and $J_{\cS}$ is not $2 K \ve$-porous for all $z \in A_\ve$. Therefore there exist sets $A_n \subset J_{\cS}$ such that $\mu (A_n)=1$ and $J_{\cS}$ is not $2 K n^{-1}$-porous for all $z \in A_n$. Let $A:=\cap_{n \in \N} A_n$. Then $\mu (A)=1$ and we will show that $J_{\cS}$ is not porous at every $z \in A$. Let $z \in A$ and suppose by contradiction that $J_{\cS}$ is $d$-porous at $z$ for some $d>0$. Choose $n \in \N$ such that $2 K n^{-1}<d$. Since $z \in A_n$ for every $\delta>0$ there exists some $r_\delta \in (0, \delta)$ such that if $y \in B(z,r_\delta)$ then $J_{\cS} \cap B(y, 2K n^{-1} r_\delta) \neq \emptyset$. Hence, $J_{\cS} \cap B(y, d r_\delta) \neq \emptyset$ for all $y \in B(z,r_\delta)$ and we have reached a contradiction since we had assumed that $J_{\cS}$ is $d$-porous at $z$. So the proof is complete when $J_{\cS}$ is not porous.

If $J_{\cS}$ is not porous at some $\zeta \in \overline{J_{\cS}}$ the proof is almost identical. In that case for all $\ve>0$ there exists $r:=r_\ve \in (0, \frac{\eta_{\cS}}{2})$ such that $\por(J_{\cS}, \zeta, r)< \ve$. Hence, for all $x \in B(\zeta, r)$
\begin{equation*}
B(x, \ve r) \cap J_{\cS} \neq \emptyset.
\end{equation*}
The rest of the proof follows exactly in the same manner as in the previous case. The proof is complete.
\end{proof}

We are now ready to see how Theorem \ref{notporoaeconfint} follows from Theorem \ref{notporoae}. For the concepts appearing in the following corollary, recall  Theorem \ref{thm-conformal-invariant} and the related discussion in the end of Section \ref{sec:prelim}. 
\begin{coro}
\label{notporoaeconf} 
Let $\cS=\{\f_e\}_{e \in E}$ be a  finitely irreducible conformal GDMS such that $J_{\cS}$ is not porous, or $J_{\cS}$ is not porous at some $\zeta \in \overline{J_{\cS}}$. If $t \in \Fin(\cS)$ then $J_{\cS}$ is not porous at $m_{t}$-a.e. $x\in J_{\cS}$. In particular $J_{\cS}$ is not porous at $m_{h}$-a.e. $x\in J_{\cS}$ where $h=\dim_{\cH}(J_{\cS})$. 

\end{coro}

\begin{proof} Let $\tilde{\mu}_t$ be the unique shift invariant ergodic measure on $E_A^\N$ which is a Gibbs state for the potential $t \zeta$, see Theorem \ref{thm-conformal-invariant}. It follows by the definition of Gibbs states, see e.g. \cite[Section 6.2]{CTU}, that  $\tilde{\mu}_t([\om])>0$ for all $\om \in E_A^\ast$. Since the collection of cylinders forms a countable base for the topology of $E_A^\N$ (recall \eqref{d-alpha}) it follows that $\tilde{\mu}_t$ has full topological support. Hence the conclusion follows by Theorem \ref{notporoae} and Theorem \ref{thm-conformal-invariant}  because the unique eigenmeasure $\tilde m_t$ of the conjugate Perron-Frobenius operator $\mathcal{L}^\ast_t$ is globally equivalent to $\tilde{\mu}_t$.
\end{proof}

We record that since $\mathcal{\cH}^{h}\lfloor_{J_{\cS}} \ll m_h$ (see e.g.  \cite[Theorem 10.1]{CTU}), Corollary \ref{notporoaeconf} implies that if  $J_{\cS}$ is not porous, or $J_{\cS}$ is not porous at some $\zeta \in \overline{J_{\cS}}$, then $J_{\cS}$ is not porous at $\mathcal{\cH}^{h}$-a.e. $x\in J_{\cS}$. Nevertheless this statement is vacuous when $\mathcal{\cH}^{h} (J_{\cS})=0$.

It is known that if $\cS$ is a finitely irreducible GDMS which consists of finitely many conformal maps and $\Int (X) \stm J_{\cS} \neq \emptyset$  then $J_{\cS}$ is porous, see e.g. \cite[Theorem 4.6.4]{MUGDMS} and \cite[Theorem 2.5]{urbpor}. According to our next Theorem the situation is very different when the alphabet is infinite. For every $h$ less than the dimension of the ambient space we construct an infinite IFS consisting of similarities whose limit set is not porous almost everywhere. 
\begin{thm} 
\label{hnotpor}For every $h \in (0, n)$  there exists a CIFS  $\cS_h$ consisting of similarities  in $\Rn$ such that $\dim_{\cH} (J_{\cS_h})=h$ and $J_{\cS_h}$ is not porous at $m_h$-a.e. $x \in J_{\cS_h}$, where $m_{h}$ is the $h$-conformal measure of $\cS$.
\end{thm}

\begin{proof} We start by fixing some $\xi \in \Rn$ and $h \in (0,n)$. By Corollary \ref{notporoaeconf} it suffices to construct a CIFS $\cS_h$ consisting of similarities such that $\dim_{\cH} (J_{\cS_h})=h$, $ \xi \in \overline{J_{\cS_h}}$ and $J_{\cS_h}$ is not porous at $\xi$. 

Let $A=(a_k)_{k \in \N}$ such that,
\begin{enumerate}[label=(\roman*)]
\item $A \subset \overline{B}(0,1/2)$,
\item every $p \in A$ is an isolated point of $A$,
\item $\xi \in \bar{A} \stm A$,
\item there exists a sequence $(\ve_k)_{k\in \N}$, with $\ve_k \in (0,10^{-1})$ such that if $z_k \in \overline{B}(a_k,\ve_k)$ then $(z_k)_{k \in \N}$ is not porous at $\xi$.
\end{enumerate} 
It is not difficult to produce countable sets with the above properties. For example, start with 
$A= \cup_{j=4}^\infty \Sigma_j$ where $\Sigma_j$ is a maximal family in $\partial B(0,1/j)$ such that for all $p,q \in \Sigma_j$, $p \neq q$, $d(p,q) \geq j^{-2}$. 

Pick some $r_1 \in (0,1)$ such that $r_1^{h}<1$ and $r_1<\ve_1$. Proceeding inductively we obtain a sequence $(r_k)_{k \in \N}$ such that $0<r_k < \ve_k$ and 
\begin{enumerate}[label=(\roman*)]
\item $\sum_{k=1}^l r_k^h <1$ for all $l \in \N$,
\item $\overline{B} (a_k, r_k) \cap (\bar{A} \stm \{a_k\}) = \emptyset$,
\item the balls $\overline{B} (a_k, r_k)$ are pairwise disjoint.
\end{enumerate} 
Of course in that  case
\begin{equation}
\label{sumhle1}
\sum_{k=1}^\infty r_k^h \leq 1.
\end{equation}
For $k \in \N$ let $\f_k : \R^n \ra \R^n$ defined by 
$$\f_k= \tau_{a_k} \circ \delta_{r_k},$$
where $\tau_q(p)=q + p$ and $\delta_r(p)=rp$. Note that $\f_k (\overline{B}(0,1)) =\overline{B} (a_k, r_k)$, hence the sequence $(\f_k)_{k \in \N}$ defines a system of similitudes $\cS$ which satisfies the SSC, recall Remark \ref{ssc}. Note also that since for all $k \in \N$,
$$J_{\cS} \cap \overline{B} (a_k, r_k) \neq \emptyset,$$
we deduce that $\xi \in \overline{J_{\cS}}$. By the choices of $a_k$ and $r_k$ we also have that if $x_k \in J_{\cS} \cap B(a_k, r_k)$ for  $k \in \N$, then the set $\{x_k\}_{k \in \N}$ is not porous at $\xi$. Therefore $J_{\cS}$ is not porous at $\xi$.

If $\sum_{k \in \N} r_k^h=1$, then by \cite[Corollary 7.22]{CTU} we have that $\dim_{\cH} (J_{\cS})=h$ and we are done. So we only have to check the case when $\sum_{k \in \N} r_k^h<1$. Since 
$$J_{\cS} \subset \bigcup_{k=1}^\infty \overline{B}(a_k, r_k),$$
we deduce that $B(0,1) \stm \overline{J_{\cS}} \neq \emptyset$. Let 
$$\psi_m: B(0,1) \ra B_m, \, m=1,\dots, n_0,$$
be  a finite set of similarities such that the open balls $B_m$ are pairwise disjoint, $B_m \subset B(0,1) \stm \overline{J_{\cS}}$, and 
$$\sum_{n \in \N} r_n^h+\sum_{m=1}^{n_0}\|D \psi_m\|^h_\infty=1.$$
Note that this is possible because the radius of each $B_m$ is $\|D \psi_m\|_\infty$ and we can choose $a_k$ and $r_k$  such that $\cup_{k=1}^\infty \overline{B}(a_k, r_k) \subset B(0,1/2)$. Hence if $\cS_h=\cS \cup \{\psi_m\}_{m=1}^{n_0}$, $J_{\cS_h}$ is not porous at $\xi$ and $\dim_{\cH} (J_{\cS_h})=h$. Hence, Corollary \ref{notporoaeconf} implies that $J_{\cS_h}$ is not porous at $m_h$-a.e. $x \in J_{\cS_h}$.The proof is complete.
\end{proof}

\section{Porosity for complex continued fractions}
\label{sec:ccf}
In this section we explore porosity in the setting of {\it complex continued fractions}. We denote by 
$$
E:=\{m+ni:(m,n) \in \N \times \Z \}
$$ 
the Gaussian integers with positive real part. We also let
$$X:=\overline{B}(1/2,1/2) \  \  \mbox{ and } \   \  W=B(1/2,3/4).$$
For $e \in E$ the maps $\f_e: W \ra W$ are defined by 
\begin{equation*}
\f_e(z):=\frac{1}{e+z}.
\end{equation*}
The following proposition gathers some properties of the maps $\f_e$. It's proof can be found in \cite{MU1}.
\begin{propo}
\label{ccfprop} 
For every $e \in E$
\begin{enumerate}[label=(\roman*)]
\item $\f_e(W) \subset B(0, 4 |e|^{-1})$,
\item\label{derfe} $4^{-1} |e|^{-2} \leq |\f'_e(z)|\leq 4 |e|^{-2}$ for all $z \in W$,
\item \label{derestcf}$4^{-1} |e|^{-2} \leq \diam (\f_e (W)) \leq 4 |e|^{-2}$.
\end{enumerate}
\end{propo}
It is easy to check that $\f_e(\Int (X)) \cap \f_a(\Int (X))=\emptyset$ for $e,a \in E, e \neq a$. Hence, $\{\f_e \}_{e \in E}$ satisfies the OSC but formally $\{\f_e \}_{e \in E}$ is {\em not} a conformal IFS because $\f'_1(0)=1$. Nevertheless, the family $\{ \f_e \circ \f_j:(e,j) \in E \times E\}$ is indeed a conformal IFS. For $I \subset E$, slightly abusing notation, we will treat the families $\mathcal{C F}_I:=\{\f_e\}_{e \in I}$ as  conformal IFS and we will call them, {\em complex continued fractions} IFS. Moreover we will denote their corresponding limit sets by $J_I$. 
With regard to the bounded distortion properties of $\mathcal{C F}_I$ we record that the best distortion constant is  $K=4$, see \cite[Remark 6.7]{MU1}. 

We will also use repeatedly the following version of Koebe's distortion theorem.
\begin{thm}
\label{kdt}
There exists a nonnegative, continuous, and increasing function $t \ra K_t, t \in [0,1)$ with $K_0=1$, such that if $p \in \C, r>0$ and $f: B(p,r) \ra \C$ is a holomorphic function, then for every $t \in [0,1)$ and for all $w,z \in B(p,tr)$,
$$\frac{|f'(w)|}{|f'(z)|} \leq K_t.$$
Consequently, if $w,z \in B(p,tr)$ and $B(w,s) \subset B(p,tr)$ and for some $s > 0$ then
$$B(f(w), K^{-1}_t |f'(z)|s) \subset f(B(w,s)) \subset B(f(w), K_t |f'(z)|s).$$
In addition,
$$f(B(p,r)) \supset B(f(p), 4^{-1} |f'(p)|r).$$
\end{thm}

It follows by \cite[Theorem 2.5]{urbpor}, or \cite[Theorem 2.6]{jjm}, that if $F \subset E$ is finite then $J_{F}$ is porous. Nevertheless, using Theorem \ref{dirporlimset} we can show something much stronger; as it turns out $J_{F}$ is directed porous at all directions. We thus establish Theorem \ref{poroccfintro} \ref{poroccfintro1}, which we restate and prove below.

\begin{thm} Let $F\subset E$ be finite and let  $J_{F}$ be the limit set associated to the complex continued fractions system $\mathcal{CF}_{F}$. Then $J_{F}$ is $v$-directed porous for all $v \in S^1$.
\end{thm}
\begin{proof} We will first show that $\f_e(z) \in \partial X$ if and only if $\Rea (e)=1$ and $z=0$.  Note that $\f_e(x) \in \partial X$ means that
$$|\f_e(z)-1/2|=\left| \frac{1}{z+e}-\frac{1}{2}\right|=1/2,$$
or equivalently
\begin{equation}
\label{dirporcf1}
|2-(z+e)|=|z+e|.
\end{equation}
Developing \eqref{dirporcf1} we see that $\f_e(z) \in \partial X$ if and only if $\Rea (z+e)=1$, which holds exactly when $\Rea (e)=1$ and $z=0$, since $e \in E$ and $z\in\bar{B}(1/2,1/2)$. 

For $e \in E$ with $\Rea(e)=1$ let $\{b_e\}=\partial X \cap \f_e(X)$ and note that $\f_e^{-1}(b_e)=0$. We want to apply Theorem \ref{dirporlimset} hence we need to verify that
\begin{equation}
\label{jfpartial}
\dist(J_F, \partial X)>0.
\end{equation}
Therefore, it suffices to show that $J_{F} \cap \partial X =\emptyset$. Suppose by way of contradiction that there exists some  $x \in J_{F} \cap \partial X$. Then $x=\pi(\om)$ for some $\om \in E^{\N}$. Then $x=\f_{\om_1}(\pi(\sigma(\om)))$ and we let $e=\om_1$. Therefore  $x \in J_{F} \cap \f_e(X) \cap \partial X$. By the previous discussion we then deduce that $x= b_e \in J_{F}$. In that case we would have that 
$$0=\f_e^{-1}(b_e)=\f_e^{-1}(x)=\pi(\sigma(\om)) \in J_{F},$$
and this is contradiction, since $0 \notin J_F$. The proof is complete.
\end{proof}

Our first main theorem in this section, with many implications in the following, provides a characterization of porous limit sets of complex continued fractions with alphabet $I \subset E$. One interesting aspect of our characterization is that everything reduces to properties of the alphabet $I$. Hence, given $I \subset E$ one can check if $J_I$ is porous by solely examining a certain property of the alphabet $I$. 

\begin{thm}
\label{poroccf}
Let $I$ be an infinite subset of $E$. Then $\overline{J_I}$ is porous if and only if there exist $\theta \in (0,1), \kappa \in (0,1)$ and  $\rho>0$, such that for every $i \in I$ and every $R \in [\rho, \kappa |i| ]$, there exists some $y_{i, R} \in B(i, R)$ such that
$$E \cap B(y_{i,R}, \theta R) \subset E \setminus I.$$
\end{thm}
Before giving the proof of Theorem \ref{poroccf} we state a general characterization of porous limit sets of infinite CIFS which we will employ in our proof.
\begin{thm}[\cite{urbpor}]
\label{marcond2}
Let $\cS=\{\f_i\}_{i \in I}$ be an infinite CIFS. Then $\overline{J_I}$ is porous if there exists a cofinite set $F \subset I$ and parameters $\eta \geq 1, c>0, \xi >0, \beta \geq  1,$
such that for all $i \in I \stm F$ and every $r \in [\beta \diam (\f_i(X)), \xi)$, there exist $x_i \in B(\f_i(X), \eta r) \cap X$ such that
$$B(x_i,cr) \cap J_I=\emptyset.$$
\end{thm}

\begin{proof}[Proof of Theorem \ref{poroccf}] Throughout the proof we are going to denote $M(z)=z^{-1}$ for $z \in \C \setminus \{0\}$. We will first show that if $\overline{J_I}$ is porous then there exist $\theta \in (0,1), \kappa \in (0,1)$ and  $\rho>0$, such that for every $i \in I$ and every $R \in [\rho, \kappa |i| ]$, there exists some $y_{i, R} \in B(i, R)$ such that
$E \cap B(y_{i,R}, \theta R) \subset E \setminus I.$ 

Since $\overline{J_I}$ is porous there exists $\alpha \in (0,1)$ such that for all $z \in \overline{J_I}$ and $r \in (0,1)$ there exists some $w:=w_{z,r} \in \C$ such that
$$B(w, \alpha r) \subset B(z,r) \setminus \overline{J_I}.$$
Fix $\kappa$ so small so that $K_{\kappa} \, \kappa \leq 1$, where $K_{\kappa}$ is as in Theorem \ref{kdt}. Note that this is possible because by Theorem \ref{kdt} the function $t \ra K_t, t \in [0,1)$ is continuous and $\lim_{t \ra 0} K_t=1$. Theorem \ref{kdt} implies that for every $i \in I$ and every $R \in (0, \kappa |i|]$ 
\begin{equation}
\begin{split}
B(i^{-1}, K_{\kappa}^{-1}R |i|^{-2}) &\subset M(B(i,R)) \subset B(i^{-1}, K_{\kappa}R |i|^{-2})\\
&\subset B(i^{-1}, K_{\kappa} \kappa |i|^{-1})\subset B(i^{-1},|i|^{-1}).
\end{split}
\end{equation}

Note that if $i \in I$ then $i^{-1} \in \overline{J_I}$. Indeed, it is easy to see that, since $I$ is infinite, for every $\ve>0$ there exists some $\om_\ve \in I^\N$ such that $|\pi(\om_\ve)|<\ve$. Therefore 
$$|\pi(i \om)-i^{-1}|=|\f_i(\pi(\om_\ve))-\f_i(0)| \overset{\eqref{boundbyderlip}}{\leq} \|\f_i'\|_{\infty}\ve,$$
and our claim follows by choosing $\ve$ appropriately small.

Thus, since $\overline{J_I}$ is porous, there exists some $w \in \C$ such that 
\begin{equation}
\label{porofirstdir}
B:=B(w, \alpha r' )\subset B(i^{-1}, r') \stm \overline{J_I},
\end{equation}
where $r'=\frac{\kappa R}{K_{\kappa} |i|^2}$.
Note that $0 \notin B(i^{-1}, r')$ because $r'<|i|^{-1}$ and observe  that
\begin{equation}
\label{firstdir1}
I \cap M^{-1}(B)=\emptyset.
\end{equation}
Suppose that \eqref{firstdir1} is false, then there exists some $a \in I \cap M^{-1}(B)$. So $a^{-1} =M(a) \in B$, and
since $a^{-1} \in \overline{J_I}$, we deduce that $B \cap  \overline{J_I} \neq \emptyset$. This contradicts \eqref{porofirstdir}.

We will now show that 
\begin{equation}
\label{firstdir2}
M^{-1}(B) \supset B \left(w^{-1},  \frac{ \alpha \kappa R}{ 4 K_{\kappa}^{2}} \right).
\end{equation}
By Theorem \ref{kdt},
\begin{equation}
\label{firstdir3}
M^{-1}(B) \supset B(w^{-1}, 4^{-1} \alpha r' |w|^{-2})=B\left(w^{-1}, \frac{\alpha \, \kappa\, R}{4 K_{\kappa} |i|^2 |w|^2}\right).
\end{equation}
Since $w \in B\left(i^{-1}, \kappa \, \frac{R}{K_{\kappa} |i|^2}\right)$, Theorem \ref{kdt} implies that
\begin{equation}
\label{firstdir4}
|i|^{2}=|M'(i^{-1})| \leq K_{\kappa}|M'(w)|=K_{\kappa}|w|^{-2}.
\end{equation}
Therefore \eqref{firstdir2} follows by \eqref{firstdir3} and \eqref{firstdir4}. 

Hence,
\begin{equation}
\label{firstdir5}
I \cap B \left(w^{-1},  \frac{ \alpha \kappa R}{ 4 K_{\kappa}^{2}} \right)\overset{\eqref{firstdir1} \wedge \eqref{firstdir2}}{=}\emptyset.
\end{equation}
Since $w \in B\left(i^{-1}, \kappa \, \frac{R}{K_{\kappa} |i|^2}\right)$, Theorem \ref{kdt} implies that
\begin{equation}
\label{firstdir6}
|w^{-1}-i|=|M(w)-M(i^{-1})| \leq K_{\kappa}|i|^2  \, \frac{\kappa R}{K_{\kappa} |i|^2}= \kappa R<R.
\end{equation}
The proof of this implication follows by \eqref{firstdir5} and \eqref{firstdir6} after choosing 
$$y_{i,R}=w^{-1}, \, \theta=4^{-1}\alpha \kappa K_{\kappa}^{-2}, \mbox{ and }\rho=\kappa/2.$$
Actually, it follows by the proof that we could even choose $\rho=0$, see Theorem \ref{poroccf1v2}.

We will now prove the other direction.  Assume that there exist $\theta \in (0,1), \kappa \in (0,1)$ and  $\rho>0$, such that for every $i \in I$ and every $R \in [\rho, \kappa |i| ]$, there exists some $y_{i, R} \in B(i, R)$ such that
$$E \cap B(y_{i,R}, \theta R) \subset E \setminus I.$$ Note that we can assume that
\begin{align}
\label{inipar}\kappa &\leq (4 K_{1/2})^{-1},\\ 
\label{inipar2} \theta^{-1} &\leq \frac{\rho}{2}. 
\end{align}
Under this assumption we will show that $\overline{J_I}$ is porous. 

We will start by proving the following claim.
\begin{claim}
\label{claim1} Let
\begin{align*}
\eta_0&:=\frac{1}{4} K_{1/2} \kappa (1+\theta), \\
c_0 &:=(8 K_{1/2})^{-1} \theta \kappa (1+ \kappa)^{-2},\\
\beta_0&:=20 \rho \, \kappa^{-1}, \\
\end{align*}
Then for all $i \in I$ and $r \in [\beta_0 \diam (\f_i(X)), 4/|i|]$, there exists $x_i \in B(\f_i(X), \eta_0\, r) \cap X$ such that
$$B(x_i, c_0\,r) \cap \overline{J_I}=\emptyset.$$
\end{claim}
\begin{proof} First notice that for $i \in I$ and $r>0$ as in the claim, by Proposition \ref{ccfprop} \ref{derestcf}  we have that,
$$r \geq \beta_0 \diam \f_i (X)=\frac{20 \rho}{ \kappa}\diam \f_i (X)\geq \frac{20 \,\rho}{4  \kappa} \, |i|^{-2}.$$\
Hence, recalling the assumption for $i$, 
\begin{equation}
\label{ccf1+2}
\frac{\rho}{ \kappa} \frac{1}{|i|^2}<\frac{r}{4} \leq \frac{1}{|i|}.
\end{equation}
Then $0 \notin \overline{B} (\frac{1}{i}, \frac{r}{4})$ and Theorem \ref{kdt} implies that 
$$M\left(B\left(\frac{1}{i},\frac{r}{8}\right)\right) \subset B\left(i, K_{1/2}^{-1}\,|i|^2\,\frac{r}{8}\right).$$
Let $R:=\frac{\kappa}{4} r |i|^2$ and note that by \eqref{ccf1+2},
$$\frac{\kappa}{4} r |i|^2 \leq \kappa \frac{1}{|i|} |i|^2= \kappa |i| \quad\mbox{
and }\quad\frac{\kappa}{4} r |i|^2 >\rho.$$
Hence, $R \in [\rho, \kappa |i|]$ and by our assumption there exists some $y_{i,R} \in B(i,R)$ such that
\begin{equation}
\label{ccf4}
I \cap B(y_{i,R}, \theta R)=\emptyset.
\end{equation}
Note that
\begin{equation}
\label{ccf5}
\frac{1}{4} \theta \kappa r |i|^2+R \leq \frac{|i|}{2}.
\end{equation}
This is easy to check:
$$\frac{1}{4} \theta \kappa r |i|^2+R=\frac{1}{4} \theta \kappa r |i|^2+\frac{1}{4}  \kappa r |i|^2=(1+\theta) \kappa \frac{r}{4}|i|^2 \overset{\eqref{ccf1+2}}{\leq} (1+\theta) \kappa |i| \overset{\eqref{inipar}}{\leq} \frac{|i|}{2}.$$
Since $y_{i,R} \in B(i,R)$,
\begin{equation}
\label{ccf5.5}
B(y_{i,R},{4}^{-1} \theta \kappa r |i|^2) \subset B(i, {4}^{-1} \theta \kappa r |i|^2+R)\overset{\eqref{ccf5}}{\subset}B\left(i, \frac{|i|}{2}\right),
\end{equation}
hence Theorem \ref{kdt} implies that
\begin{equation}
\label{ccf6}
B\left(y_{i,R}^{-1}, (4 K_{1/2})^{-1} \theta \kappa r\right) \subset M(B(y_{i,R}, {4}^{-1} \theta \kappa r |i|^2)).
\end{equation}
Again by Theorem \ref{kdt} and \eqref{ccf5.5}
\begin{equation}
\label{ccf7old}
\begin{split}
M(B(y_{i,R},4^{-1} \theta \kappa r |i|^2)) &\subset M(B(i,4^{-1} \theta \kappa r |i|^2+R)) \\
& \subset B(i^{-1}, 4^{-1}K_{1/2}\theta \kappa r+K_{1/2} R |i|^{-2})\\
&=B(i^{-1}, 4^{-1} K_{1/2} \kappa (1+\theta)r).
\end{split}
\end{equation}
Therefore
\begin{equation}
\label{ccf7}
M(B(y_{i,R},4^{-1} \theta \kappa r |i|^2)) \subset B(\f_i(X), 4^{-1} K_{1/2} \kappa (1+\theta)r).
\end{equation}

Now assume that Claim \ref{claim1} is false. Then there exists some $i \in I$ and some 
$$\beta_0  \diam (\f_i(X)) \leq r \leq 4/|i|$$ such that for all $y \in B(\f_i(X), \eta_0\, r) \cap X$ it holds that
$$B(y, c_0\,r) \cap \overline{J_I} \neq\emptyset.$$
By \eqref{ccf7} we know that $y_{i,R}^{-1}\in B(\f_i(X), \eta_0\, r)$, where as before $R=\frac{\kappa}{4} r |i|^2$. Hence, there exists some $j \in I$ such that
$$B(y_{i,R}^{-1},c_0\,r) \cap \f_j(X)\neq \emptyset.$$
Note that
$$M(\f_j(X))=M \circ M \circ \tau_j(\overline{B}(1/2,1/2))= \tau_j(\overline{B}(1/2,1/2))=\overline{B}(j+1/2,1/2)\subset \overline{B}(j,1),$$
where $t_j(x)=j+x$. So, we conclude that
\begin{equation}
\label{ccf8}
M(B(y_{i,R}^{-1},c_0 r)) \cap \overline{B}(j,1) \neq \emptyset.
\end{equation}
Moreover,
\begin{equation*}
\begin{split}
B\left(y_{i,R}^{-1}, (4 K_{1/2})^{-1} \theta \kappa r\right) &\overset{\eqref{ccf6} \wedge \eqref{ccf7old}}{\subset} B(i^{-1}, 4^{-1} K_{1/2} \kappa (1+\theta)r)\\
&\overset{\eqref{ccf1+2}}{\subset}B (i^{-1}, K_{1/2} \kappa |i|^{-1}) \overset{\eqref{inipar}}{\subset} B(i^{-1}, (2 |i|)^{-1}).
\end{split}
\end{equation*}
Therefore since $c_0< (4 K_{1/2})^{-1} \theta \kappa$, $y_{i,R} \in B(i,R)$, and $R \leq \kappa |i|$, Theorem \ref{kdt} implies that
\begin{equation}
\label{ccf9}
M(B(y_{i,R}^{-1},c_0 \,r)) \subset B(y_{i,R}, K_{1/2} \,c_0 \, |y_{i,R}|^2 r) \subset B(y_{i,R}, K_{1/2} \,c_0 \, (1+ \kappa)^2 |i|^2 r).
\end{equation}
Consequently by  \eqref{ccf8} and \eqref{ccf9} we conclude that
\begin{equation*}
B(y_{i,R}, K_{1/2} \,c_0 \, (1+ \kappa)^2 |i|^2 r)\cap \overline{B}(j,1) \neq \emptyset,
\end{equation*}
and
\begin{equation}
\label{ccf10}
|y_{i,R}-j| \leq 1+ K_{1/2} \,c_0 \, (1+ \kappa)^2 |i|^2 r.
\end{equation}
Note also that
\begin{equation}
\label{ccf11}
2 \theta^{-1} \overset{\eqref{inipar2}}{\leq} \rho \overset{\eqref{ccf1+2}}{<} \frac{\kappa}{4} r |i|^2.
\end{equation}
So, recalling the definition of $c_0$,
\begin{equation}
\label{ccf11'}
|y_{i,R}-j| \overset{\eqref{ccf10} \wedge \eqref{ccf11}}{<} \frac{r}{8} \theta \kappa |i|^2 r + K_{1/2} \,c_0 \, (1+ \kappa)^2 |i|^2 r \leq \frac{r}{4} \theta \kappa |i|^2=\theta R.
\end{equation}
Now notice that \eqref{ccf11'} contradicts \eqref{ccf4}. The proof of Claim \ref{claim1} is complete.
\end{proof}
We now continue with the proof of the theorem. Note that since $0 \notin \f_e(X)$ for all $e \in E$, we have that
\begin{equation}
\label{r0def}
r_0:=\dist \left(\bigcup_{i \in I: |i| \leq 8} \f_i(X),0\right)>0.
\end{equation}
We will prove the following claim.
\begin{claim}
\label{claim2} Let $F= \{i \in I: |i| > 4/r_0\}$, and
\begin{align}
\label{eta}\eta&:=\max \left\{ 5,\frac{1}{4} K_{1/2} \kappa (1+\theta)+6+\frac{\kappa}{25 \rho} \right\},\\
\label{c}c &:= c_0= (8 K_{1/2})^{-1} \theta \kappa (1+ \kappa)^{-2} ,\\
\label{beta}\beta&:=\beta_0=\frac{20 \rho}{ \kappa}, \\
\label{xi}\xi&:=\min\left\{\frac{1}{8}, \frac{r_0}{24}, \frac{\kappa}{10^5 \rho}\right\}.
\end{align}
Then for all $i \in I \stm F$ and $r \in [\beta \diam (\f_i(X)), \xi]$, there exists $x_i \in B(\f_i(X), \eta\, r) \cap X$ such that
$$B(x_i, c\,r) \cap \overline{J_I}=\emptyset.$$
\end{claim}
\begin{proof} Let $i \in I \stm F$ and $r \in [\beta \diam (\f_i(X)),4 /|i|]$. In that case, since $\eta \geq \eta_0, c = c_0$ and $\beta=\beta_0$, Claim \ref{claim2} follows by Claim \ref{claim1}. Therefore we can assume that $i \in I \stm F$ and 
\begin{equation}
\label{ccf12}
\max\{\beta \diam \f_i(X), 4/|i|\} \leq r \leq \xi.
\end{equation}
We are going to distinguish two cases. First assume that
\begin{equation}
\label{ccf13}
(B(i^{-1},6r) \stm B(i^{-1}, 2r)) \cap \bigcup_{a \in I} \f_a(X)=\emptyset.
\end{equation}
It is easy to see that for every $e \in E$, 
$$e^{-1} \in X=\overline{B}(1/2,1/2).$$
Moreover by the choice of $\xi$ we have that $4r \leq 1/2$, hence $X \cap \partial B(i^{-1},4r) \neq \emptyset$. Let 
$$x_i \in X \cap \partial B(i^{-1},4r).$$
Since $x_i \in \partial B(i^{-1},4r)$ we have that $B(x_i,  r) \subset B(i^{-1},6r) \stm B(i^{-1}, 2r)$, therefore \eqref{ccf13} implies that $B(x_i, r) \cap \overline{J_I}=\emptyset$. Since $x_i \in B(\f_i(X), 5r)$  the claim has been proven in that case.

We are left with the case when
$(B(i^{-1},6r) \stm B(i^{-1}, 2r)) \cap \bigcup_{a \in I} \f_a(X) \neq \emptyset.$
Hence there exists some $a \in I$ such that
\begin{equation}
\label{ccf14} 
(B(i^{-1},6r) \stm B(i^{-1}, 2r)) \cap \f_a(X) \neq \emptyset.
\end{equation} Therefore, 
\begin{equation*}
\begin{split}
|a|^{-1} &\geq |a^{-1}-i^{-1}|-|i|^{-1} \overset{\eqref{ccf12} \wedge \eqref{ccf14}}{\geq}2r-\diam( \f_a(X))-\frac{r}{4} \\
&\geq  \frac{7r}{4}-4|a|^{-2} \geq \frac{7r}{4}-4|a|^{-1},
\end{split}
\end{equation*}
where in the third inequality we used  Proposition \ref{ccfprop} \ref{derestcf}. Hence,
\begin{equation}
\label{ccf15}
|a|^{-1} \geq \frac{7 r}{20} >\frac{r}{4}.
\end{equation}

Observe that 
\begin{equation}
\label{a8}
|a| > 8.
\end{equation} To see this, assume by contradiction that $|a| \leq 8$. Then \eqref{r0def} implies that 
\begin{equation}
\label{distfa}
\dist( \f_a(X),0)\geq r_0.
\end{equation}
Moreover, by \eqref{ccf14} there exists $z \in (B(i^{-1},6r) \stm B(i^{-1}, 2r)) \cap \f_a(X)$. Since $i \in F$ and $r \leq \xi$ we deduce that
$$|z| \leq |i|^{-1}+6r < \frac{r_0}{4} +\frac{r_0}{4}=\frac{r_0}{2}.$$
And this contradicts \eqref{distfa}, establishing \eqref{a8}. 

We record again that \eqref{ccf14} implies that
\begin{equation}
\label{ccf15+}
|i^{-1}-a^{-1}| \leq 6r+\diam (\f_a(X)).
\end{equation}
We now have,
\begin{equation}
\begin{split}
\label{ccf16}
|a|^{-1} &\leq |i|^{-1}+|i^{-1}-a^{-1}|  \overset{\eqref{ccf12}\wedge \eqref{ccf15+}}{\leq} 7r+\diam (\f_a(X))\\
&\leq 7r+4 |a|^{-2} \overset{\eqref{a8}}{\leq} 7r+\frac{1}{2} |a|^{-1},
\end{split}
\end{equation}
where in the third inequality we used  Proposition \ref{ccfprop} \ref{derestcf}.
Hence
\begin{equation}
\label{ccf16.5}
\begin{split}
|a|^{-2} &\overset{\eqref{ccf16}}{\leq} 14^2 r \leq  14^2 r^2 \xi \leq  \overset{\eqref{xi}}{\leq} r \frac{\kappa}{10^2 \rho }.
\end{split}
\end{equation}
Consequently, Proposition \ref{ccfprop} \ref{derestcf} implies that
\begin{equation}
\label{ccflast}
\frac{25 \rho }{\kappa}\diam (\f_a(X))\leq \frac{100 \, \rho }{\kappa} |a|^{-2} \overset{\eqref{ccf16.5}}{\leq} r.
\end{equation}
Since $\frac{25 \rho }{\kappa} > \beta_0$, \eqref{ccf15} and \eqref{ccflast} allows us to use Claim \ref{claim1} and obtain some $x_a \in B(a^{-1}, \eta_0 r)$ such that 
\begin{equation}
\label{xaxi}
B(x_a, c_0 r) \cap \overline{J_I}=\emptyset.
\end{equation} Notice then, that 
\begin{equation}
\label{etafin}
\begin{split}
|x_a-i^{-1}| &\leq |x_a-a^{-1}|+|a^{-1}-i^{-1}|\overset{\eqref{ccf15+}}{\leq} \eta_0 r+6r+\diam (\f_a(X))  \\
& \overset{\eqref{ccflast}}{\leq} \left( \eta_0 +6+\frac{\kappa}{25 \rho} \right) r  \overset{\eqref{eta}}{\leq} \eta r.
\end{split}
\end{equation}
Hence Claim \ref{claim2} follows by \eqref{xaxi} and \eqref{etafin} after choosing $x_i:=x_a$.
\end{proof}
Notice that Claim \ref{claim2} and Theorem \ref{marcond2} imply that $\overline{J_I}$ is porous. The proof is complete.
\end{proof}
Observe that the proof of the first implication of Theorem \ref{poroccf} implies that if $\overline{J_I}$ is porous we can choose the parameter $\rho=0$. While, what we prove in the second implication of Theorem \ref{poroccf} is stronger than assuming that $\rho=0$. Therefore we have also shown the following.
\begin{thm}
\label{poroccf1v2}
Let $I$ be an infinite subset of $E$. Then $\overline{J_I}$ is porous if and only if there exist $\theta \in (0,1)$ and  $\kappa \in (0,1)$ such that for every $i \in I$ and every $R \in [0, \kappa |i| ]$, there exists some $y_{i, R} \in B(i, R)$ such that
$$E \cap B(y_{i,R}, \theta R) \subset E \setminus I.$$
\end{thm}

Since the (full) complex continued fractions IFS $\mathcal{C F}_E$ satisfies condition \eqref{xx1}, as an immediate consequence of this theorem and Theorem~\ref{thmporfix} (i), we get the following.  

\begin{thm}
If $\mathcal{C F}_E$ is a (full) complex continued fractions IFS, then its limit set $J_E$ is not porous but it is porous at a dense set of its points.
\end{thm}

\begin{remark}
\label{norms} The balls appearing in Theorems \ref{poroccf1v2} and \ref{poroccf} are taken with respect to the Euclidean norm. Nevertheless, Theorems \ref{poroccf1v2} and \ref{poroccf} hold true (with different constants $\theta$, $\kappa$ and $\rho$) if the balls are taken with respect to any norm in $\C$, since all norms in finite dimensional spaces are bi-Lipschitz equivalent.
\end{remark}
\begin{remark} 
\label{remk:marporchar} It is interesting to compare Theorem \ref{poroccf} with \cite[Theorem 3.3]{urbpor}. The latter concerns real continued fractions and, strictly speaking, it
 is entirely independent of Theorem \ref{poroccf}. Indeed, porosity is not an intrinsic property of a subset of a metric space but does depend on the (ambient) metric space as well.  For example, the interval $[0,1]$ is trivially porous in $\C$, hence every limit set of real continued fractions, as considered in \cite{urbpor}, is porous in $\C$. However, as it was proven in \cite{urbpor}, there exist many such limit sets which are not porous as subsets of $\R$.
 
Nevertheless, if we carry on the proof of Theorem \ref{poroccf} with $E \subset \N$ and $\C$ replaced by $\R$, we will obtain a characterization of limit sets $J_{E}$ that are porous in $\R$, analogous to Theorem \ref{poroccf} (in that case $y_{i,R} \in \R$ and the balls $B(i,R)$ and $B(y_{i,R}, \theta R)$  are taken with respect to the usual topology of $\R$). It is not hard to see that such a characterization is equivalent to that of \cite[Theorem 3.3]{urbpor}.
\end{remark}

\begin{defn}
\label{dens}For any $I \subset E=\{m+ni: m\in \N, n \in \Z\}$ we define the {\it upper and lower densities of $I$ in $E$} as
$$\overline{\rho}_{E} (I):=\limsup_{R \ra +\infty} \frac{\sharp (I \cap \overline{B}_\infty(1,R))}{\sharp (E \cap \overline{B}_\infty(1,R))} \quad\mbox{ and }\quad\underline{\rho}_{E} (I):=\liminf_{R \ra +\infty} \frac{\sharp (I \cap \overline{B}_\infty(1,R))}{\sharp (E \cap \overline{B}_\infty(1,R))},$$
where $B_\infty (z,r)$ denotes the ball centered at $z$ with radius r, with respect to the norm $\|w\|_\infty=\max \{|\Rea(w)|, |\Img(w)|\}$.
\end{defn}
Slightly abusing notation, we will call a set $I \subset E$ {\it porous} in $E$ if $\overline{J_I}$ is porous in $\C$. The following proposition relates porosity to the notion of upper density which we just defined. The proof is based on Theorem \ref{poroccf1v2}.

\begin{propo}
\label{densle1} If $I \subset E$ is porous then $\overline{\rho}_{E} (I)<1$.
\end{propo}
\begin{proof} According to Theorem \ref{poroccf1v2} and Remark \ref{norms} there exist constants $\theta \in (0,1)$ and  $\kappa \in (0,1)$ such that for every $i \in I$ and every $R \in [0, \kappa \|i\|_\infty ]$, there exists some $y_{i, R} \in B_\infty(i, R)$ such that
$$E \cap B_\infty(y_{i,R}, \theta R) \subset E \setminus I.$$
Let $R> 100\,\theta^{-1}\, \kappa^{-1}.$ We will estimate the quantity
$$\frac{\sharp (I \cap \overline{B}_\infty(1,R))}{\sharp (E \cap \overline{B}_\infty(1,R))}.$$
Notice that since $\sharp (S \cap \overline{B}_\infty(1,R))=\sharp (S \cap B_{\infty}(1,\floor*R))$ for every $S \subset E$  and $R>0$ we can assume that $R \in \N$.

We distinguish two cases. First assume that 
\begin{equation}
\label{emptyanulus}
I \cap A_\infty (1, R /4, R/2)=\emptyset,
\end{equation}
where $A_\infty (z,r,s)=\{w \in \C: r \leq \|w-z\|_\infty \leq s\}$. Notice that \eqref{emptyanulus} implies that 
\begin{equation}
\label{cardan}
\sharp (I \cap \overline{B}_\infty(1,R))< \sharp (E \cap \overline{B}_\infty(1,R))-\floor*{R/4}^2<\sharp (E \cap \overline{B}_\infty(1,R))-(R/8)^2.
\end{equation}
Now assume that \eqref{emptyanulus} fails. Then, there exists some 
$b \in I \cap A_\infty (1, R /4, R/2).$ Let $R'=\kappa \,R/4$. Then $R' \leq \kappa \,\|b\|_\infty$. Hence there exists some $y_{b,R'} \in B_\infty (b, R')$ such that
$$I \cap B_{\infty}(y_{b,R'}, \theta R')=\emptyset.$$
Note that
\begin{equation}
\label{cardan1}
\begin{split}
\sharp (I \cap \overline{B}_\infty(1,R)) &\leq \sharp (E \cap \overline{B}_\infty(1,R))-\sharp(E \cap B_{\infty}(y_{b,R'}, \theta R'))\\
& \leq \sharp (E \cap \overline{B}_\infty(1,R))-( \theta R')^2 \\
&=\sharp (E \cap \overline{B}_\infty(1,R))-(4^{-1} \theta \, \kappa R )^2.
\end{split}
\end{equation}
Let
$$c=\min \left\{\frac{1}{64}, \frac{\kappa^2 \theta^2}{16 }\right\}.$$
Then, by \eqref{cardan} and \eqref{cardan1} we deduce that for every $R>100\,\theta^{-1}\, \kappa^{-1}$, we have that
$$\sharp (I \cap \overline{B}_\infty(1,R)) \leq \sharp (E \cap \overline{B}_\infty(1,R))-c R^2.$$
Observe that  $\sharp (E \cap \overline{B}_\infty(1,R))=(R+1) \cdot (2R+1)$, hence
$$\frac{\sharp (I \cap \overline{B}_\infty(1,R))}{\sharp (E \cap \overline{B}_\infty(1,R))} \leq 1-\frac{c R^2}{2R^2+3R+1},$$
and consequently
$$\limsup_{R \ra \infty} \frac{\sharp (I \cap \overline{B}_\infty(1,R))}{\sharp (E \cap \overline{B}_\infty(1,R))} \leq 1-c/2<1.$$
The proof is complete.
\end{proof}
We are now ready to prove Theorem \ref{poroccfintro} \ref{poroccfintro2}, which improves and extends significantly \cite[Theorem 4.2]{urbpor}.  As the reader can check, the proof will follow easily from Proposition \ref{densle1}, Corollary \ref{notporoaeconf} and Corollary \ref{meanporoconfhaus}.
\begin{thm}
\label{cofsubofccf}  
Let $I$ be any co-finite subset of $E$, let  $J_{I}$ be the limit set associated to the complex continued fractions system $\mathcal{CF}_I$ and let $h_I=\dim_{\cH}(J_I)$. Let also $m_{h_I}$ be the $h_I$-conformal measure of $\mathcal{CF}_I$. Then:
\begin{enumerate}[label=(\roman*)]
\item \label{cofsubofccf1}  The limit set $J_I$ is not porous at $m_{h_I}$-a.e. $x \in J_I$.
\item  \label{cofsubofccf2} There exists a constant $c_I$ such that $J_I$ is $c_I$-mean porous at $m_{h_I}$-a.e. $x \in J_I$.
\end{enumerate}
\end{thm}
\begin{proof} 

As an immediate corollary of Proposition \ref{densle1} we deduce that if $I \subset E$ is co-finite then the limit set $J_{I}$ is not porous. Hence \ref{cofsubofccf1} follows by Corollary \ref{notporoaeconf}. 

Moreover, it follows by \cite[Proposition 6.1]{MU1} that  $\mathcal{CF}_E$ is co-finitely regular. Since $I$ is co-finite \cite[Lemma 3.10]{CLU} implies that  $\mathcal{CF}_I$ is co-finitely regular. Hence, $\mathcal{CF}_I$ is strongly regular by \eqref{cofimpliesreg}. 
Therefore \ref{cofsubofccf2} follows by Corollary \ref{meanporoconfhaus}.
\end{proof}


\begin{defn}
\label{dens}For any $I \subset \N$ we define the {\it upper and lower densities of $I$ in $\N$} as
$$\overline{\rho}_{\N} (I):=\limsup_{n \ra +\infty} \frac{\sharp (I \cap [1,n])}{n}\quad\mbox{ and }\quad\underline{\rho}_{\N} (I):=\liminf_{n \ra +\infty} \frac{\sharp (I \cap [1,n])}{n}.$$
In a similar manner if $I \subset \Z$ we define the {\it upper and lower densities of $I$ in $\Z$} by

$$\overline{\rho}_{\Z} (I):=\limsup_{n \ra +\infty} \frac{\sharp (I \cap [-n,n])}{2n+1}\quad\mbox{ and }\quad\underline{\rho}_{\Z} (I):=\liminf_{n \ra +\infty} \frac{\sharp (I \cap [-n,n])}{2n+1}.$$
\end{defn}

\begin{defn}
\label{nzpor}We say that $I \subset \N$ is {\it $\N$-porous} if the second condition of Theorem \ref{poroccf} holds with
$$y_{i,R} \in (\N \times \{0\}) \cap B(i,R), \, i \in I \times \{0\}.$$
In the same manner, we say that $I \subset \Z$ is {\it $\Z$-porous} if the second condition of Theorem \ref{poroccf} holds with
$$y_{i,R} \in (1 \times \Z) \cap B(i,R), \, i \in \{1\} \times I.$$
\end{defn}
The proof of the following proposition is straightforward and we leave it to the reader.
\begin{propo}
\label{propoprod} Let $I_1 \subset \N$ and $I_2 \subset \Z$.
\begin{enumerate}[label=(\roman*)]
\item If $\overline{\rho}_{\N} (I_1)<1$ and $\overline{\rho}_{\N} (I_2)<1$ then $\overline{\rho}_{E} (I_1 \times I_2)<1.$
\item \label{poroprod} If $I_1$ is $\N$-porous and $I_2$ is $\Z$-porous then $I_1 \times I_2$ is porous.
\end{enumerate}
\end{propo}
It follows by \cite[Theorem 3.15]{urbpor} that if $a \geq 2$ then the set $I_{a}:=\{\a^n\}_{n \in \N}$ is $\N$-porous. Hence, the following corollary follows by Proposition \ref{propoprod} \ref{poroprod}.
\begin{coro} Let $a,b,c \geq 2$. If $I_1 \subset I_a$ and $I_2 \subset I_b \cup (-I_c)$ then the set $I_1 \times I_2$ is porous.
\end{coro}

For $z \in \C$ and $r>0$ we are going to denote by $Q(z,r)$ the closed filled square centered at $z$ with sides parallel to the axis and sidelength  $\ell(Q)=r$, i.e. $Q(z,r)=\overline{B}_{\infty}(z,r/2)$. Moreover for $r>0$ we will denote
$$\Delta(r)=\{Q(z,r): z \in \C\}.$$
We will also use the notation 
$$
\mathcal{D}=\bigcup_{r>0} \Delta(r)
$$ 
for the collection of all closed squares with sides parallel to the axis. Being motivated by the concept of upper density dimension for subsets of positive integers, introduced in Section 3 of \cite{urbpor}, we propose the following analogous but improved definition for subsets of $\Z^2$.

\begin{defn}
\label{bddefn} If $I \subset \Z^2$ the \textit{upper box dimension} of $I$ is defined as 
$$\bd (I)=\overline{\lim}_{R \ra +\infty} \, \sup \left \{ \frac{\log\sharp (I \cap Q)}{\log R}: Q \in \Delta(R)\right\}.$$
\end{defn}
\label{boxdimdef}
Note that if $I \subset \Z^2$ then $\sharp (I \cap Q) \leq (R+1)^2$ for every $R>0$ and every $Q \in \Delta(R)$. Therefore
\begin{equation}
\label{bdleq2}
\bd (I) \leq 2 \  \  \mbox{ for all } \  I \subset \Z^2.
\end{equation}
In our next theorem we show that if $I \subset \N \times \Z=E$ is porous then the inequality in \eqref{bdleq2} is strict.
\begin{thm}
\label{bdporous} If $I \subset E$ is porous then $\bd(I)<2$.
\end{thm}
\begin{proof}Fix some $R\geq 2$ and consider a square $Q_0 \in \Delta(R)$ such that  
$$\Rea(Q):=\{\Rea(z):z\in Q\}\subset [0,+\infty).$$ We are going to construct inductively a finite number of families of squares from $\mathcal{D}$ with mutually disjoint interiors, whose union contains $Q_0 \cap I$. 

The first family contains only $Q_0$ and we denote it by $\Sigma_1=\{Q_0\}$. Now suppose that $\Sigma_n$ has been defined. Passing to the inductive step we start with any 
$$
Q:=Q(w, \ell(Q)) \in \Sigma_n.
$$
We then decompose $Q$ into $81$ squares from $\Delta(\frac{1}{9} \ell(Q))$ with mutually disjoint interiors. We call $Q_{c}$ the square from  $\Delta(\frac{1}{9} \ell(Q))$ which shares the same center with $Q$. 

If $Q_{c} \cap I =\emptyset$, we remove $Q_{c}$ and we denote by $\Sigma^1_{n+1}(Q)$  the family of the remaining $80$ squares from $\Delta(\frac{1}{9} \ell(Q))$ whose union is $Q \stm \Int(Q_{c})$. Note also that
\begin{equation}
\label{areacenter}
\mbox{Area}\left(\bigcup_{Q'\in \Sigma^1_{n+1}(Q)}Q\right)=\ell(Q)^2-\frac{1}{81} \ell(Q)^2=\frac{80}{81}\ell(Q)^2= \frac{80}{81} \mbox{Area}(Q).
\end{equation}

If $Q_{c} \cap I \neq \emptyset$ we pick some $\xi \in Q_{c} \cap I$. According to Theorem \ref{poroccf1v2} and Remark \ref{norms}, since $I$ is porous there exist constants $\theta \in (0,1)$ and  $\kappa \in (0,1)$ such that for every $i \in I$ and every $R \in [0, \kappa \|i\|_\infty ]$, there exists some $y_{i, R} \in B_\infty(i, R)$ such that $E \cap B_\infty(y_{i,R}, \theta R) \subset E \setminus I.$ Notice that $\ell(Q) \leq 3 \|\xi\|_{\infty}$, since $\Rea(Q)\subset [0,+\infty)$. In particular, if we choose
\begin{equation}
\label{defl}
L=\frac{\kappa}{9} \ell(Q),
\end{equation}
we have that $L < \kappa \|\xi\|_{\infty}$. Therefore there exists a point $z \in Q(\xi, L)$ such that
$$I \cap Q(z, \theta L)=\emptyset.$$
Since we can assume that $\theta<1/9$, we have that
\begin{equation}
\begin{split}
Q(z, \theta L)& \subset Q\left(w, \theta L+L+\frac{1}{9} \ell(Q)\right)\overset{\eqref{defl}}{\subset} Q\left(w,(1+\theta)\frac{1}{9} \ell(Q)+\frac{1}{9} \ell(Q)\right) \\
& \subset Q(w, \ell(Q))=Q.
\end{split}
\end{equation}

Let $k$ be the smallest natural number such that $2^{-k} \ell(Q) \leq \frac{\theta}{9}L$, or equivalently the smallest natural number such that $2^{-k} \leq \frac{\theta \kappa}{81}$, and decompose $Q$ into elements of $\Delta(2^{-k} \ell(Q))$.  We record that by the definition of $k$,
\begin{equation}
\label{mink}
2^{-k} \ell(Q) > \frac{\theta}{18} L.
\end{equation}
Let $P \in \Delta(2^{-k} \ell(Q))$ such that 
$$P \cap Q\left(z, \frac{\theta}{9} L\right) \neq \emptyset.$$
Then
$$P \subset  Q\left(z, \frac{\theta}{9} L+2 \,2^{-k} \ell(Q)\right) = Q\left(z, \frac{\theta}{3} L\right).$$
We remove $P$ and we denote by $\Sigma^2_{n+1}(Q)$ the family of the remaining $2^{2k}-1$ squares from $\Delta(2^{-k} \ell(Q))$ whose union is $Q \stm \Int (P)$. We also have that
 \begin{equation}
 \begin{split}
\label{areadeep}
\mbox{Area}\left(\bigcup_{Q'\in \Sigma^2_{n+1}(Q)}Q\right)&=\ell(Q)^2-2^{-2k} \ell(Q)^2 \overset{\eqref{mink}}{<} \left(1-\left( \frac{\theta\, L}{18 \ell(Q)}\right)^2 \right) \ell(Q)^2 \\
&\overset{\eqref{defl}}{=}  \left(1-\left( \frac{\theta \kappa}{162} \right)^2 \right) \ell(Q)^2:=\eta \, \ell(Q)^2=\eta \, \mbox{Area}(Q),
\end{split}
\end{equation}
and we record that $\eta \in (0,1)$.

So we can now complete the inductive step. For any $Q \in \Sigma_n$ we let
$$\Sigma_{n+1}(Q)=\begin{cases}\Sigma^1_{n+1}(Q):\mbox{ if } Q_{c} \cap I =\emptyset\\
\Sigma^2_{n+1}(Q):\mbox{ if } Q_{c} \cap I \neq \emptyset,
\end{cases}$$
and we define 
$$\Sigma_{n+1}:= \{ \Sigma_{n+1}(Q): Q \in \Sigma_n \}.$$
The process terminates at level $N \in \N$ if at least one element of $\Sigma_N$ contains a square with sidelength less than $1$. 

Let $C_n=\bigcup_{Q \in \Sigma_n} Q$ and notice that for all $n=1,\dots,N-1,$
\begin{equation*}
\begin{split}
\mbox{Area} (C_{n+1})&=\sum_{Q \in \Sigma_{n}} \mbox{Area}\left(\bigcup_{Q'\in \Sigma_{n+1}(Q)}Q'\right)  \overset{\eqref{areacenter} \wedge \eqref{areadeep}}{\leq} \eta \sum_{Q \in \Sigma_{n}} \mbox{Area}(Q)= \eta \, \mbox{Area}(C_n),
\end{split}
\end{equation*}
and
\begin{equation}
\label{limitarea}
\mbox{Area} (C_{n}) \leq \eta^{n-1} \mbox{Area} (C_{1})= \eta^{n-1}\mbox{Area}(Q_0)=\eta^{n-1}R^2.
\end{equation}
Moreover, notice that for all $Q \in \Sigma_{N},$
\begin{equation}
\label{limitlength}
\ell(Q) \overset{\eqref{mink}}{\geq} \left(\frac{\theta \kappa}{162}\right)^{N-1} \ell(Q_0):= \alpha^{N-1} R,
\end{equation}
and $\alpha \in (0,1)$. Observe also that 
$$I \cap Q_0=I \cap C_n$$
for all $n=1,\dots,N$.

Since the process terminates at level $N$ there exists at least one $Q \in \Sigma_{N}$ such that $\ell(Q)<1$. Therefore \eqref{limitlength} implies that $(N-1) \log \alpha+\log R < 0$, or equivallently
\begin{equation}
\label{logeqn}
(N-1)\frac{\log \alpha}{ \log R}<-1.
\end{equation}
Let $Q=Q(w, \ell(Q)) \in \Sigma_{N-1}$ such that $I \cap Q \neq \emptyset$. Note that since $I \subset \N + \Z i$, if $i,j \in  I \cap Q, i\neq j,$ then
$$\Int(Q(i,1)) \cap \Int(Q(j,1))=\emptyset.$$
Therefore, since $\ell(Q) \geq 1$,
\begin{equation}
\label{sharpiq}
\sharp (I \cap Q)= \mbox{Area} \left( \bigcup_{i \in I \cap Q} Q(i,1) \right) \subset \mbox{Area}(Q(w, 2 \ell(Q)))=4\,\mbox{Area}(Q),
\end{equation}
and
\begin{equation}
\label{sharpiq0}
\sharp (I \cap Q_0)\leq\sum_{Q \in \Sigma_{N-1}} \sharp (I \cap Q) \overset{\eqref{sharpiq}}{\leq}4\, \sum_{Q \in \Sigma_{N-1}}\mbox{Area}(Q)=4\, \mbox{Area}(C_{N-1}).
\end{equation}

From now on we will assume that $R \geq \alpha^{-2}$. Notice that this implies that $N \geq 3$. Hence, using that $\eta, \alpha \in (0,1)$, we get the following estimate
\begin{equation}
\label{finalbd}
\begin{split}
\frac{\log \sharp (I \cap Q_0)}{\log R} &\overset{\eqref{sharpiq0}}{\leq} \frac{\log\big(\mbox{Area}(C_{N-1})\big) }{\log R}+\frac{\log 4}{\log R} \\
&\overset{\eqref{limitarea}}{\leq} \frac{(N-2) \log \eta+2 \log R}{ \log R}+ \frac{\log 4}{\log R}\\
&=2+\frac{N-2}{N-1} \, (N-1) \frac{\log \a}{\log R} \cdot \frac{\log \eta}{\log \alpha}+ \frac{\log 4}{\log R} \\
&\overset{\eqref{logeqn}}{<} 2-\frac{N-2}{N-1}\cdot \frac{\log \eta}{\log \alpha}+ \frac{\log 4}{\log R}\\
&\overset{N \geq 3}{\leq} 2-\frac{1}{2} \frac{\log \eta}{\log \alpha}+ \frac{\log 4}{\log R}.
\end{split}
\end{equation}
Finally notice that since $I \subset \N \times \Z$,
$$
\bd (I)=\overline{\lim}_{R \ra +\infty} \, \sup \left \{ \frac{\log \sharp (I \cap Q)}{\log R}: Q \in \Delta(R) \mbox{ and } \Rea(Q) \subset [0,+\infty)\right\}
$$
Therefore, \eqref{finalbd} implies that $\bd (I) \leq 2-\frac{1}{2} \frac{\log \eta}{\log \alpha}<2$. The proof is complete.
\end{proof}

\begin{rem}
Theorem \ref{bdporous} corresponds to Theorem~3.5 in \cite{urbpor}. Although, logically speaking, both theorems are independent as one of them concerns subsets of $\N$ while the other concerns Gaussian integers, their assertions are analogous and the proofs are related. However, the proof provided in the current paper is clearer, simpler, and better describes the key ideas. One could easily adopt it to give a better proof of Theorem~3.5 in \cite{urbpor}.
\end{rem}

We will apply Theorem \ref{bdporous} to the set of \textit{Gaussian primes} which has been studied extensively in number theory. Recall that $\Z[i]$, the set of Gaussian integers, has exactly four \textit{units} \,:  $1,-1,i$ and $-i$. These are the only elements of $\Z[i]$ whose Euclidean norm is equal to $1$. Multiplying any $z \in \Z[i]$ by the units of $\Z[i]$ we obtain its associates, i.e. the \textit{associates} of $z$ are $z, -z,iz$ and $-iz$. A Gaussian integer $z \in \Z[i]$ is called \textit{prime} if $|z|>1$ and it is divisible only by the units and its associates. There are many good sources of information for the basic divisibility properties of Gaussian integers and in particular about Gaussian primes, see e.g. \cite{conrad,stillwell}. See also \cite{hecke} for a treatment of several analytic topics related to Gaussian primes. 

If $-\pi \leq a<b \leq \pi$ we denote
$$GP_{a,b}=\{w \in E: w \mbox{ is a Gaussian prime and }\arg w \in [a,b)\},$$
and recall that $E$ is the set of all Gaussian integers with positive real part.
\begin{lm}
\label{gpreg}
If $-\pi/2\leq a<b \leq \pi/2$ then the complex continued fractions system $\mathcal{CF}_{GP_{a,b}}$ is co-finitely regular.
\end{lm}
\begin{proof} By Hecke's Prime Number Theorem, see e.g. \cite[Theorem 4, pages 134-135]{hecke}, if 
$$\pi_{a,b} (x)=\sharp \{ w: w \mbox{ is a Gaussian prime and } a \leq \arg w< b, |w|^2 \leq x \} $$ then
\begin{equation}
\label{heckepntang}
\pi_{a,b}(x) \sim \frac{2}{\pi}(b-a) \frac{x}{ \log x}.
\end{equation}
Note that if $a \in (-\pi/2, \pi/2)$ then for $R>0$
$$\sharp (GP_{a,b} \cap \bar{B}(0,R))=\pi_{a,b}(R^2).$$
While, if $a=-\pi/2$ then
$$\sharp (GP_{a,b} \cap \bar{B}(0,R))+\sharp (N \cap \bar{B}(0,R))=\pi_{a,b}(R^2),$$
where 
$$N:=\{w \mbox{ is a Gaussian prime such that} \Rea w=0 \mbox{ and }\Img w \leq 0\}.$$
Note that $\sharp (N \cap \bar{B}(0,R)) \leq R$ therefore, if $a=-\pi/2$ then
$$\sharp (GP_{a,b} \cap \bar{B}(0,R)) \geq \pi_{a,b}(R^2) -R.$$

Hence, using \eqref{heckepntang} it is not difficult to show that if $-\pi/2\leq a<b \leq \pi/2$ then there exists some $R_0>0$ such that for all $R \geq R_0$,
\begin{equation}
\label{gprimeann}
\frac{1}{16\,\pi} (b-a) \frac{R^2}{\log R}\leq \sharp (GP_{a,b} \cap \bar{B}(0,2R) \stm \bar{B}(0,R)) \leq \frac{8}{\pi} (b-a) \frac{R^2}{\log R}.
\end{equation}
By \eqref{zn} and Proposition \ref{ccfprop} \ref{derfe}, for $t\geq 0$
\begin{equation}
\label{z1ccf}
Z_1(\mathcal{CF}_{GP_{a,b}},t):=Z_1(t) \approx \sum_{e \in GP_{a,b}}|e|^{-2t}.
\end{equation}
Now let $A_n= \bar{B}(0,2^{n+1}R_0)\stm  \bar{B}(0,2^{n}R_0) \cap GP_{a,b}, n \in \N,$ and note that
\begin{equation}
\label{z1est}
\begin{split}
\sum_{e \in GP_{a,b},\, |e| \geq 2 R_0} |e|^{-2t}&=\sum_{n=1}^\infty \sum_{e \in A_n} |e|^{-2t} \approx \sum_{n=1}^\infty \sharp A_n \cdot 2^{-2nt} \\
&\overset{\eqref{gprimeann}}{\approx} \sum_{n=1}^\infty \frac{(2^n R_0)^2}{\log(2^n \cdot R_0)} 2^{-2nt} \approx  \sum_{n=1}^\infty \frac{2^{2n(1-t)}}{n}.
\end{split}
\end{equation}
Since 
$$Z_1(t) \overset{\eqref{z1est}}{\approx}\sum_{e \in GP_{a,b},\, |e| < 2 R_0} |e|^{-2t}+\sum_{n=1}^\infty \frac{2^{2n(1-t)}}{n},$$
\eqref{thetaz} implies that $\theta(\mathcal{CF}_{GP_{a,b}})=1$ and $Z_1(1)=+\infty$. Therefore, recalling Definition \ref{regulardef} and \eqref{presz1}, we deduce that $\mathcal{CF}_{GP_{a,b}}$ is co-finitely regular. The proof is complete.
\end{proof}
We will conclude this section with the proof of Theorem \ref{poroccfintro} \ref{poroccfintro3}.
\begin{thm} 
\label{gaussianprimes}
Let $-\pi/2\leq a<b \leq \pi/2$ and let $I$ be any co-finite subset of $GP_{a,b}$. Let  $J_{I}$ be the limit set associated to the complex continued fractions system $\mathcal{CF}_{I}$ and let $h_{I}=\dim_{\cH}(J_{I})$. Then:
\begin{enumerate}[label=(\roman*)]
\item \label{gpsubofccf1}  The limit set $J_I$ is not porous at $m_{h_I}$-a.e. $x \in J_I$.
\item  \label{gpsubofccf2} There exists a constant $c_I$ such that $J_I$ is $c_I$-mean porous at $m_{h_I}$-a.e. $x \in J_I$.
\end{enumerate}

\end{thm}
\begin{proof} 
Using \eqref{gprimeann}, we deduce that there exits some $c \in (0,1)$ such that if $R$ is large enough
\begin{equation}
\label{heckepntapp}
\sharp (I \cap Q(0,R)) \geq \frac{c\,R^2}{\log R}.
\end{equation}
Hence, for $R$ large enough
\begin{equation}
\label{logpnt}
\frac{\log \sharp (I \cap Q(0,R))}{\log R} \overset{\eqref{heckepntapp}}{\geq}2+\frac{\log c}{\log R}-\frac{\log \log R}{\log R},
\end{equation}
and consequently  
\begin{equation}
\label{logpnt2}\sup \left \{ \frac{\log\sharp (I \cap Q)}{\log R}: Q \in \Delta(R)\right\} \overset{\eqref{logpnt}}{\geq} 2+\frac{\log c }{\log R}-\frac{\log \log R}{\log R}.
\end{equation}
Recalling Definition \ref{bddefn} we see that \eqref{logpnt2} implies that $\bd(I) \geq 2$. Hence by \eqref{bdleq2} we deduce that $\bd(I) = 2$, and Theorem \ref{bdporous} implies that $I$ is not porous, i.e. the set $J_{I}$ is not porous. Therefore Corollary \ref{notporoaeconf} implies that $J_{I}$ is not porous at $m_{h_{I}}$-a.e. $x \in J_{I}$.


By Lemma \ref{gpreg} the system $\mathcal{CF}_{GP_{a,b}}$ is co-finitely regular. Now the proof of \ref{gpsubofccf2} follows exactly as in the proof of Theorem \ref{cofsubofccf}  \ref{cofsubofccf1}. The proof is complete.
\end{proof}

Taking $a=-\pi/2$ and $b=\pi/2$ in  Theorem \ref{gaussianprimes} we obtain the following corollary involving  the complex continued fractions system whose alphabet is the set of Gaussian primes with positive real part, see also Figure \ref{gausprimes}.
\begin{coro}
Let $GP^{+}$ be the set of Gaussian primes with positive real part. Let  $J_{GP^{+}}$ be the limit set associated to the complex continued fractions system $\mathcal{CF}_{GP^{+}}$ and let $h_{GP^{+}}=\dim_{\cH}(J_{GP^{+}})$. Then:
\begin{enumerate}[label=(\roman*)]
\item  The limit set $J_{GP^{+}}$ is not porous at $m_{h_{GP^{+}}}$-a.e. $x \in J_{GP^{+}}$.
\item   There exists a constant $c_{GP^{+}}$ such that $J_{GP^{+}}$ is $c_{GP^{+}}$-mean porous at $m_{h_{GP^{+}}}$-a.e. $x \in J_{GP^{+}}$.
\end{enumerate}

\end{coro}

\section{Porosity for meromorphic functions}\label{pmf}

In this short section we deal with some quite general classes of meromorphic (either rational functions or transcendental) functions from $\C$ to $\hat{\C}$. A very powerful tool of meromorphic dynamics is the concept of a \textit{nice set}. Roughly speaking a nice set of a meromorphic function is a set such that the holomorphic inverse branches of the first return map to this set form a conformal IFS. By means of nice sets we will apply our results on mean porosity of conformal IFSs to the realm of some large classes of meromorphic functions from $\C$ to $\hat{\C}$. 

More precisely, let $f:\C\to\oc$ be a meromorphic function. Let $\Sing(f^{-1})$ be the set of all \textit{singular points} of $f^{-1}$, i. e. the set of all points $w\in\oc$ such that if $W$ is any open connected neighborhood of $w$, then there exists a connected component $U$ of $f^{-1}(W)$ such that the map $f:U\to W$ is not bijective. Of course, if $f$ is a rational function, then $\Sing(f^{-1})=f(\Crit(f))$, where 
$$
\Crit(f):=\big\{w \in \C: f'(w)=0\big\}.
$$
We also define
$$
\PS(f):=\bu_{n=0}^\infty f^n(\Sing(f^{-1})).
$$

We are now going to recall the definitions of Fatou and Julia sets of meromorphic functions.
\begin{defn}
Let $f: \C \ra \hat{\C}$ be a meromorphic function. The {\it Fatou set} $F(f)$ of the function $f$ is the set of all points $z \in \C$ for which there exists an open neighborhood $U_z$ of $z$ such that all iterates $f^n|_{U_z}, n \in \N,$ are well defined and form a normal family in the sense of Montel. 

We also define the {\it Julia set} of $f$, as
$$J(f):=\hat{\C} \stm F(f).$$
\end{defn} 
Following \cite{PU_tame} and \cite{SkU} a meromorphic function $f:\C\to\oc$ is called \textit{tame} if and only if
$$
J(f)\sms \ov{\PS(f)}\ne\es.
$$

\begin{rem}
\label{tamerem}
Tameness is a very mild hypothesis which is satisfied by many natural classes of maps.  These include:
\begin{enumerate}
\item Quadratic maps $\oc\ni z \mapsto z^2 + c\in\oc$ for which $c\in\R$ and the Julia set is not contained in the real line;
\item Rational maps for which the restriction to the Julia set is expansive which includes the case of expanding rational functions; and 
\item Misiurewicz maps, where the critical point is not recurrent. 

\item Dynamically regular meromorphic functions introduced and considered in \cite{MayerU1} and \cite{MayerU2}.
\end{enumerate}
\end{rem}

In this paper the main advantage of dealing with tame functions is that these admit \textit{nice sets}. Before giving their formal definition (included in the next theorem) we record that Rivera-Letelier \cite{Riv07} introduced the concept of nice sets in the realm of  dynamics of
rational maps of the Riemann sphere. In \cite{Dob11} Dobbs proved
their existence for tame meromorphic functions from $\C$ to $\oc$; see also \cite{KU1} for an extended treatment of nice sets. Before quoting Dobbs' theorem, given a tame meromorphic function $f:\C\to\oc$ a set $F\sbt\oc$ and an integer $n\ge 0$, we denote by
$
\cC_{F,f}(n):=\cC_{F}(n)
$ 
the collection of all connected components of $f^{-n}(F)$.

\begin{thm}\label{prop:1}
Let $f:\C\to\oc$ be a tame meromorphic function. Fix a non-periodic point $z\in J(f)\sms
\ov{\PS(f)}$, and two parameters $\kappa>1$, and $K>1$. Then for all $L>1$ and 
for all $r>0$ sufficiently small there exists
an open connected simply connected set $V=V(z,r)\sbt\C\sms\ov{\PS(f)}$, called a nice set, such that
  \begin{enumerate}[label=(\roman*)] 
  \item \label{prop:1a} If $U\in \cC_V(n)$ and $U\cap V\neq \emptyset$, then 
    $U\subseteq V$.
  \item If $U\in \cC_V(n)$ \  \  and  \  \  $U\cap V\neq \emptyset$,
    then, for all $w,w'\in U,$ 
    \begin{displaymath}
      |(f^n)'(w)|\ge L 
\  \  \   \textrm{ and } \  \  \ 
       \frac{|(f^n)'(w)|}{|(f^n)'(w')|}\le K. 
    \end{displaymath}
  \item $\overline{B(z,r)}\subset V\subset B(z,\kappa r)\sbt B(z,2\kappa r)\sbt \C\sms\ov{\PS(f)}$. 
\end{enumerate}
\end{thm}

\fr Each nice set of a tame meromorphic function canonically gives rise to a countable alphabet conformal iterated function system in the sense considered in the previous sections of the present paper. Namely, let $V$ be a nice set of a tame meromorphic function $f:\C\to\oc$, and put 
$$
\cC_V^*=\bu_{n=1}^\infty\cC_V(n).
$$
=== It is easy to see (comp. \cite{KU1} or \cite{SkU} for more details) that for every $U\in \cC_V^*$ let $\tau_V(U)\ge 1$ the unique integer $n\ge 1$
such that $U\in \cC_V(n)$. Since $V \subset \C$ is open, connected, simply connected and disjoint from $\ov{\PS(f)}$, using the Inverse Function Theorem and the Monodromy Theorem in the standard way (see \cite{Conway}), we see that there exists 
$$
f_U^{-\tau_V(U)}:B(z,2\kappa r)\to \C
$$
a unique holomorphic branch of $f^{-\tau_V(U)}$ such that 
$$
f_U^{-\tau_V(U)}(V)=U.
$$
Denote
$$
\phi_U:=f_U^{-\tau_V(U)}
$$
and keep in mind that
$$
\phi_U(V)=U.
$$
Denote by $E_V$ the collection of all elements $U$ of $\cC_V^*$ such that 

\begin{itemize}

\item[(a)] $\phi_U(V)\sbt V$, 
\item[(b)] $f^k(U)\cap V=\es$ \, for all \, $k=1,2,\ldots,\tau_V(U)-1$.
\end{itemize}
Of course (a) yields
\begin{itemize}
\item[(a')]$\phi_U(\ov V)\sbt \ov V.$
\end{itemize}
We note that the collection $E_V$ is not empty and the details can be found in the proof of Lemma \ref{l220190618}.

We now form a conformal iterated function system as follows. Let
$$
X:=\ov V 
\  \  \  {\rm and} \  \  \
W:=B(z,2\kappa r).
$$
Assume now that the parameter $L$ from Theorem \ref{prop:1} is so large that $2L^{-1}<1$. Take any $w \in W=B(z, 2 \kappa r)$ and let $w'$ be the point in $[z,w] \cap \overline{B(z,r)}$ closest to $w$, where $[z,w]$ denotes the line segment with endpoints $z$ and $w$. Then for every $U \in E_V$ we have that $\f_U(w') \in \overline{V}$, and
$$|\f_{U}(w)-\f_{U}(w')|\leq L^{-1}|w-w'| \leq L^{-1} 2 \kappa r.$$
Hence, 
$$|\f_U(w)-z| \leq |\f_{U}(w)-\f_{U}(w')|+|\f_{U}(w')-z| \leq L^{-1} 2 \kappa r + \kappa r<2 \kappa r.$$
Therefore, $\f_{U}(w) \in B(z,2 \kappa r)$ and consequently $\f_{U}(W) \subset W$.

The collection of maps 
\begin{equation}\label{520190614}
\cS_V:=\big\{\phi_U:W\to W\big\}_{U\in E_V}
\end{equation}
is therefore well defined and, by (a'),
$$
\phi_U(X)\sbt X
$$
for all $U\in E_V$. We claim that $\cS_V$ is a conformal IFS satisfying the 
Open Set Condition. We only mention that all the sets $\{U=\phi_U(V):U\in E_V\}$ are mutually disjoint by their very definition which is given by means of continuous inverse branches of positive iterates of a single meromorphic map. Therefore, the Open Set Condition follows. Furthermore, uniform contraction of the elements of the system $\cS_V$ follows immediately from item (b) of Theorem~\ref{prop:1}.

In other words the elements of $\cS_V$ are determined by all holomorphic inverse branches of the first return map $F_V:V\to V$. In particular, $\tau_V(U)$ is the first return time of all points in $U=\phi_U(V)$ to $V$.

It is easily seen from this construction, and we provide a proof, that the following holds.

\begin{lm}\label{l220190618}
Let $f:\C\to\oc$ be a tame meromorphic function. If $V$ is a nice set produced in Theorem~\ref{prop:1} and $\cS_V$ is the corresponding conformal IFS produced in formula \eqref{520190614}, then the limit set $J_V$ of the system $\cS_V$ contains all transitive points of the map $f:J(f)\to J(f)$ lying in $V$, i.e. the set of points in $V\cap J(f)$ whose orbit under iterates of $f$ is dense in $J(f)$. In particular:
$$
\ov{J_V}=\ov V\cap J(f).
$$
In addition, if $J(f)\ne\C$ (this in fact means that $J(f)$ is nowhere dense in $\C$), then condition \eqref{xx1} holds for the IFS $\cS_V$.
\end{lm} 
\begin{proof} 
The inclusion $\ov{J_V}\sbt \ov V$ is obvious. The inclusion $\ov{J_V}\sbt J(f)$ follows since $\ov{J_V}$ is the closure of all fixed points of all elements $\phi_\om$, $\om\in E_V^*$, and these are repelling periodic points of $f$ which are all in the closed set $J(f)$. Thus,
$$
\ov{J_V}\sbt \ov V\cap J(f).
$$
In order to see the opposite inclusion, recall first that the set $V\cap J(f)$ contains transitive points, see e.g. \cite{Bergweiler,KU1}. Take any transitive point $w\in V\cap J(f)$ and let et $(n_j)_{j=1}^\infty$ be the unbounded increasing sequence of all consecutive visits of $w$ in $V$ under the action of $f$. In other words
$$
f^{n_j}(w)\in V
$$
for all $j\ge 1$ and 
$$
f^k(w)\notin V
$$
if $k\ne n_j$ for all $j\ge 1$. Let $m_j=n_j-n_{j-1}, j \in \N$ and let $n_0=0$. Then, using Theorem~\ref{prop:1}, we see that for every $j\ge 1$ there exists a unique holomorphic branch of $f^{-m_j}$ defined on $W$ and mapping $V$ into $V$, and furthermore, that this branch belongs to $\cS_V$. Hence, this branch is equal to $\phi_{U_j}$ for some $U_j\in E_V$. Note that for all $k \in \N$,
\begin{equation}
\label{fcompinv}
\phi_{U_1} \circ \phi_{U_2} \circ \dots \circ \phi_{U_k} \circ f^{m_1 }\circ f^{m_2} \circ \dots f^{m_k} (w)=w.
\end{equation}
Since $f^{m_1 }\circ f^{m_2} \circ \dots f^{m_k} (w)=f^{n_k}(w) \in V$, \eqref{fcompinv} implies that $w \in \phi_{U_1} \circ \phi_{U_2} \circ \dots \circ \phi_{U_k} (V)$.
Therefore,
$$
w=\pi (U_1U_2U_3\ldots)\in J_V
$$
and we are done. 
\end{proof}

As an almost immediate consequence of this lemma and Theorem~\ref{thmporfix} (i), we get the following.  

\begin{thm}
If $f:\C\to\oc$ be a tame meromorphic function such that $J(f)\ne\C$, then the Julia set $J(f)$ is porous at a dense set of its points.
\end{thm}

\begin{proof}
Since $J(f)\ne\C$, it is nowhere dense in $\C$, and therefore the conformal IFS $\cS_V$ produced in formula \eqref{520190614}, satisfies condition \eqref{xx1}. Hence, the proof concludes by a direct application of Theorem~\ref{thmporfix} (i) and Lemma~\ref{l220190618}.
\end{proof}

As a fairly direct consequence of Lemma~\ref{l220190618}, Theorem~\ref{prop:1}, and Theorem~\ref{meanporthm}, we shall prove the following.

\begin{thm}\label{t120190618}
Let $f:\C\to\oc$ be a tame meromorphic function such that $J(f)\ne\C$. If $\mu$ is a Borel probability $f$-invariant ergodic measure on $J(f)$ with full topological support and with finite Lyapunov exponent  
$\chi_{\mu}(f):=\int_{J(f)}\log|f'|\,d\mu$, then there exist some $\a_{f}, p_f \in (0,1]$ such that set $J(f)$ is $(\a_f,p_f)$--mean porous 
at $\mu$-a.e. $x\in J(f)$.
\end{thm} 

\begin{proof} Let $V$ be a nice set as in Lemma \ref{l220190618}. Notice then, that 
\begin{equation}
\label{vbdemptyint}
V \cap \bigcup_{n=1}^\infty f^{n}(\partial V)=\emptyset.
\end{equation} 
To see this, suppose by contradiction that there exists $x \in \partial V$ such that $f^{n}(x) \in V$ for some $n \in \N$. Then there exists a holomorphic branch $\phi:V \ra \C$ of $f^{-n}$ such that $\phi (f^n(x))=x$. Hence, $\phi (V) \in \cC_V(n)$ and since non-constant meromorphic functions are open, $\phi (V)$ is open. Therefore, $\phi (V) \cap V^{c} \neq \emptyset$, because $\phi(V) \cap \partial V \neq \emptyset$. But this contradicts Theorem \ref{prop:1} \ref{prop:1a}, and \eqref{vbdemptyint} has been proven.

Since $\spt (\mu)=J(f)$, we have that $\mu(V)>0$. Hence, $\mu(\cup_{n=1}^\infty f^{n}(\partial V))<1$. Since the set 
$$
\bigcup_{n=1}^\infty f^{n}(\partial V)
$$ 
is $f$-forward invariant, that is 
$$
f\left(\bigcup_{n=1}^\infty f^{n}(\partial V)\right) \subset \bigcup_{n=1}^\infty f^{n}(\partial V),
$$
its complement is $f$-backwards invariant, that is 
$$
f^{-1}\left(\oc \stm \bigcup_{n=1}^\infty f^{n}(\partial V)\right) \subset \oc \stm \bigcup_{n=1}^\infty f^{n}(\partial V).
$$
Hence, the ergodicity of the measure $\mu$ implies that 
\begin{equation}
\label{mfthmeq1}
\mu \left(\bigcup_{n=1}^\infty f^{n}(\partial V)\right)=0.
\end{equation}
Since $\partial V \subset f^{-1}(f(\partial V))$ and the measure $\mu$ is $f$-invariant, this implies that $\mu(\partial V)=0$. So, again by the $f$-invariance of $\mu$, we get that
\begin{equation}
\label{mfthmeq2}
\mu \left(\bigcup_{n=0}^\infty f^{-n}(\partial V)\right)=0.
\end{equation}
In particular, if $\mu_{J_V}$ is the conditional measure on $J_V$, i.e. 
$$\mu_{J_{V}}(F)=\frac{\mu(F)}{\mu(J_V)}, \quad F \subset J_{V} \mbox{ Borel},$$
then
\begin{equation}
\label{mfthmeq3}
\mu_{J_V} \left(J_V\cap \bigcup_{n=0}^\infty f^{-n}(\partial V)\right)=0.
\end{equation}
We will now show that if $\pi:E_V^{\N} \ra J_{V}$ is the projection map generated by the IFS $\cS_{V}$, then 
\begin{equation}
\label{aeuniquecode}
\mu\big(\{z \in J_{V}: \pi^{-1}(z) \mbox{ is a singleton}\}\big)=1.
\end{equation}
Note first that \eqref{aeuniquecode} will follow from \eqref{mfthmeq3} if we prove that every point
$$
z \in J_{V} \stm\bigcup_{n=0}^\infty f^{-n}(\partial V)
$$ 
has a unique code. Suppose that this is not the case, then there exists 
$$
z \in J_{V} \stm\bigcup_{n=0}^\infty f^{-n}(\partial V)
$$ 
and two distinct words $\om, \tau \in E_V^\N$ such that $\pi(\om)=\pi(\tau)=z$. Let $k \in \N$ be the first instance such that $\om_k \neq \tau_k$. Then 
$$z \in \phi_{\om|_k}(\overline{V}) \cap \phi_{\tau|_k}(\overline{V}).$$

Hence, the formula
\begin{equation}
\label{mfthmeq3}
\tilde{\mu}(A):=\mu_{J_{V}}(\pi_{V}(A))\quad A \subset E_V^\N \mbox{ Borel},
\end{equation}
defines a Borel probability measure on $E^\N$. Also then, $\mu_{J_{V}}=\tilde{\mu} \circ \pi_{V}^{-1}$ and the following diagram commutes
\begin{equation*}
\begin{tikzcd}
  E^{\N} \arrow[r, "\sigma"] \arrow[d, "\pi_{V}"]
    & E^{\N} \arrow[d, "\pi_{V}"] \\
  J_{V} \arrow[r,  "f_{V}" ]
&J_{V}\end{tikzcd},
\end{equation*}
where $f_{V}:J_{V} \ra J_{V}$ is the first return map from $J_{V}$ to $J_{V}$ defined (by the Poincar\'e Recurrence Theorem) $\mu_{J_{V}}$-a.e. Furthermore, the projection map $\pi_{V}: (E^{\N}, \tilde{\mu}) \ra (J_{V}, \mu_{V})$ is a metric isomorphism. Since, the measure $\mu$ is $f$-invariant and ergodic, the measure $\mu_{J_{V}}$ is $f_V$-invariant and ergodic. Consequently, the measure $\tilde{\mu}$ is $\sigma-$ invariant and ergodic. Moreover, by Kac's Lemma, 
$$\chi_{\tilde{\mu}}(\sigma)=\int_{J_{V}} \log |f_V'|\, d \mu_{J_{V}}=\frac{\chi_\mu(f)}{\mu(J_{V})}<+\infty.$$
Therefore, an application of Theorem \ref{meanporthm} gives that the set $J_{V}$ is 
$$
\left(\alpha_{\cS_{V}}, \frac{\mu(J_{V})\, \log 2}{2\chi_{\tilde{\mu}}(f)}\right)-{\rm mean \ porous}
$$
at $\mu_{J_V}$-a.e. $x \in J_{V}$. Since, by Lemma \ref{l220190618}, $\overline{J_V} \supset J (f) \cap V$, the set $J(f)$ is  $$
\left(\alpha_{\cS_{V}}, \frac{\mu(J_{V})\, \log 2}{2\chi_{\tilde{\mu}}(f)}\right)-{\rm mean \ porous}
$$ at $\mu$-a.e. $x \in J(f) \cap V$.
Note that apart from critical points and poles of $f$ (which is a countable set and so of $\mu$ measure $0$), the property of being $(\alpha,p)$-mean porous is $f$-invariant because of Lemma~\ref{porconfinv}. Indeed, if $p$ is not a critical point or a pole of $f$ and $W$ is a neighborhood of $p$ such that $f$ is injective on $W$ then $J(f) \cap f(W)=f(J(f)\cap W)$. Thus, we can apply Lemma~\ref{porconfinv} with $W$ as above and $X=Y=J(f)$. Hence, the ergodicity of $\mu$ implies that the set $J(f)$ is 
$$
\left(\alpha_{\cS_{V}}, \frac{\mu(J_{V})\, \log 2}{2\chi_{\tilde{\mu}}(f)}\right)-{\rm mean \ porous}
$$
at $\mu$-a.e. point of $J(f)$. The proof is complete.
\end{proof}

\begin{rem}\label{r220190621}
We would like to note that using Pesin's theory, as developed in \cite{PUbook}, we could remove the tameness hypothesis from Theorem~\ref{t120190618} if the Lyapunov exponent $\chi_{\mu}(f)$ is positive. 
\end{rem}

In essentially the same way (only simpler) as Theorem~\ref{notporoae}, we can prove the following. Alternatively, we could deduce it from Theorem~\ref{notporoae}, Theorem~\ref{prop:1}, and Lemma~\ref{l220190618}, in a similar was as we deduced Theorem~\ref{t120190618}.

\begin{thm} 
\label{t320190618}Let $f:\C\to\oc$ be a tame meromorphic function such that $J(f)\ne\C$ and $J(f)$ is not porous in $\C$. If $\mu$ is a Borel probability $f$-invariant ergodic measure on $J(f)$ with full topological support, then $J(f)$ is not porous at $\mu$-a.e. $x \in J(f)$.
\end{thm}

\section{Porosity for elliptic functions}
\label{sec:elliptic}
In the last section of our paper we prove that the Julia sets of non-constant elliptic functions are \textit{not} porous. We first recall the definition of elliptic functions.
\begin{defn}
\label{defell}
A meromorphic function $f: \C \ra \hat{\C}$ is called {\it elliptic} if it is doubly periodic; i.e. if there exist two complex numbers $w_1,w_2 \in \C \stm \{0\}$ such that $\Img (w_2/w_1) \neq 0$ and 
$$f(z+w_1)=f(z+w_2)=f(z)$$
for all $z \in \C$. We also define the {\it basic fundamental parallelogram } of $f$ by
$$\mathcal{R}_f:=\{t_1 w_1+t_2 w_2: t_1,t_2 \in [0,1]\}.$$
and we denote by $\Lambda_f$ the lattice generated by $w_1$ and $w_2$, i.e.
$$
\Lambda_f:=\{n_1w_1+n_2w_2:n_1, n_2\in\Z\}
$$
\end{defn}

In the following we collect some basic well-known facts about elliptic functions which are going to be used in the following, see \cite{KU1} for more information. 
If $G$ is an open subset of $\C$ and $b$ is a pole of a meromorphic function $g:G \ra \hat{\C}$, we denote  by $q_b \geq 1$ the order of $g$ at $b$. The following proposition presents some elementary facts about poles in the context we will make use of them. 

\begin{propo}
If $f: \C \ra \hat{\C}$ is a non-constant meromorphic function, then 
for all $n \in \N$ and for every $b \in f^{-n}(\infty)$ there exists some $R_1:=R_1(n,b)>0$ and some $A:=A(n,b) \geq 1$ such that for all $z \in B(b, R_1)$
\begin{equation}
\label{fact1eq1}
A^{-1}|z-b|^{-q_b} \leq |f^n(z)| \leq A|z-b|^{-q_b},
\end{equation}
and
\begin{equation}
\label{fact1eq2}
A^{-1}|z-b|^{-(q_b+1)} \leq |(f^n)'(z)| \leq A|z-b|^{-(q_b+1)}.
\end{equation}
where, we recall, $q_b\ge 1$ is the order of $b$ as the pole of $f^n$.
\end{propo}

\begin{propo}\label{p1120180817}
If $f:\C \ra \hat{\C}$ is a non-constant elliptic function, then 
there exists some $R_2>0$ such that for all $w \in \C$,
\begin{equation}
\label{fact1eq3}
f(B(w, R_2))=\hat{\C}.
\end{equation}
\end{propo}
 
\begin{proof}
Just take $R_2>0$ so large that each ball $B(w, R_2)$ contains a congruent, mod $\Lambda_f$, copy of the fundamental parallelogram $\mathcal{R}_f$. 
\end{proof}

\begin{remark}\label{r1120180817}
It is immediate from Montel's Theorem that if $f$ is a non-constant elliptic function and $U$ is an open set such that all iterates $f^n|_{U}, n \in \N,$ are well defined (in fact $f^{-1}(\infty)\cap\bigcup_{n\ge 0}f^n(U)=\emptyset$ and remember that $f^{-1}(\infty)$ is an infinite set, so in particular it contains three different points), then they form a normal family in the sense of Montel. 
\end{remark}
The following proposition gathers some properties of the Julia set of elliptic functions.
\begin{propo}
\label{factfatou}
Let $f: \C \ra \hat{\C}$ be a non-constant elliptic function. Then
\begin{equation}
\label{fact2eq1}
J(f)=\overline{\cup_{n=1}^\infty f^{-n}(\infty)} \neq \emptyset.
\end{equation}
and
\begin{equation}
\label{fact2eq2}f(J(f) \stm \{\infty\})=J(f)=f^{-1}(J(f)) \cup \{\infty\}.
\end{equation}
\end{propo}
\begin{proof}
The formula \eqref{fact2eq1} is an immediate consequence of Remark~\ref{r1120180817} while \eqref{fact2eq2} is one of the most basic facts in the theory of iteration of meromorphic functions and  for example it is formulated in  \cite[Lemma~2]{Bergweiler}. 
\end{proof}

Notice that if $f$ is a non-constant elliptic function then for all $w \in \C$
\begin{equation}
\label{4ef3}
J(f) \cap B(w,R_2) \neq \emptyset.
\end{equation}
This follows because
$$J(f) \cap  f(B(w,R_2)) \overset{\eqref{fact1eq3}\wedge\eqref{fact2eq1}}{ \neq}\emptyset.$$
Hence there exists some 
$$z \in f^{-1}(J(f)) \cap B(w, R_2)\overset{\eqref{fact2eq2}}{ \subset} J(f) \cap B(w, R_2),$$
and \eqref{4ef3} follows.

We are now ready to prove Theorem \ref{nonporoellipintro} which we restate for the convenience of the reader.
\begin{thm}\label{t120190621}
Let $f:\C \ra \hat{\C}$ be a non-constant elliptic function. Then:
\begin{enumerate}[label=(\roman*)]
\item \label{elnonpor1} The Julia set $J(f)$ is not porous at a dense set of its points, in particular it is not porous at any point of the set 
$$
P_f:=\bigcup_{n=1}^{\infty} f^{-n}(\infty).
$$

\item \label{elnonpor2} For every $b \in P_f$ and for all $\kappa \in (0,1)$ there exists $R(b,\kappa)>0$ such that
$$
\por(J(f),b,r)\le \kappa
$$
for all $r\in(0,R(b,\kappa))$.
\item\label{elpor} If in addition $J(f)\ne\C$, then $J(f)$ is porous at a dense set of its points; the repelling periodic points of $f$.
\end{enumerate}
\end{thm}
\begin{proof} Without loss of generality we can assume that $f:\hat{\C} \ra \hat{\C}$ and $f(\infty)=\infty$. of course \ref{elnonpor1} follows from \ref{elnonpor2}. So for the proof of \ref{elnonpor2} let 
$$
b \in \bigcup_{n=1}^{\infty} f^{-n}(\infty),
$$
and fix some $n \in \N$ such that 
$$
b \in f^{-n}(\infty).
$$
We will show that for all $\kappa \in (0,1)$ and all $r \in (0,1)$ small enough, we have that
\begin{equation}
\label{1ef3}
J(f) \cap B(z,\kappa r) \neq \emptyset,
\end{equation}
for all $z \in B(b,r) \stm \{b\}$.
 
Note that since $f'$ is a non-constant elliptic function with the same basic fundamental parallelogram as $f$, the set $\Crit(f) \cap \mathcal{R}_f$ is finite. Hence $f(\Crit (f))=f(\Crit(f) \cap \mathcal{R}_f)$ is finite as well. As a consequence the set $\bigcup_{k=1}^n f^k (\Crit (f))$ is also finite. Therefore there exists some $R_3>0$ (depending on $b$) such that
\begin{equation}
\label{2ef3a}
\left(\bigcup_{k=1}^n f^k (\Crit (f)) \right) \cap B(0,R_3)^c =\emptyset.
\end{equation}
Observe that \eqref{2ef3a} implies that \begin{equation}
\label{2ef3}
f^n (\Crit (f^n)) \cap B(0,R_3)^c =\emptyset.
\end{equation}
To verify \eqref{2ef3} let $\xi \in \Crit (f^n)$. Then there exists some $l=0,\dots,n-1$ such that $f^{l}(\xi) \in \Crit (f)$, hence  
$$f^{n}(\xi) \in f^{n-l} (\Crit (f)).$$
Thus $f^n (\Crit (f^n)) \subset \bigcup_{k=1}^n f^k (\Crit (f))$, and \eqref{2ef3} follows from \eqref{2ef3a}. Arguing in the same way it actually follows that $f^n (\Crit (f^n))=\bigcup_{k=1}^n f^k (\Crit (f)$.

Now let $\kappa \in (0,1)$ and let
\begin{equation}
\label{rrange}
0 <r< \min\{R_1, R_2, \kappa (K_{1/2} A R_2)^{-1/q_b}, (A(2 R_2+R_3))^{-1/q_b}\}:=R(b,\kappa),
\end{equation}
where $q_b \geq 1$ is the order of $b$ as pole of $f^n$. In the following, $z \in B(b,r)\stm \{b\}$. Note then that
\begin{equation}
\label{fnzbound}
|f^n(z)| \overset{\eqref{fact1eq1}}{\geq}A^{-1} |z-b|^{-q_b} \geq A^{-1} r^{-q_b}  \overset{\eqref{rrange}}{\geq} 2 R_2+R_3.
\end{equation}
Therefore \eqref{2ef3}, \eqref{fnzbound}, and the Monodromy Theorem, imply that there exists a unique holomorphic inverse branch of $f^n$,
$$f_z^{-n}(B(f_n(z),2 R_2)) \ra \C,$$
such that $f_z^{-n}(f_n(z))=z$. By Theorem \ref{kdt} we have that,
\begin{equation}
\begin{split}
\label{kdel}
f_z^{-n}(B(f^n(z),R_2)) &\subset B(z, K_{1/2} |(f^n)'(z)|^{-1} R_2)
\overset{\eqref{fact1eq2}}{\subset} B(z, K_{1/2} R_2 A |z-b|^{q_b+1}) \\
& \subset B(z, K_{1/2} R_2 A r^{q_b+1})\overset{\eqref{rrange}}{\subset} B(z, \kappa r).
\end{split}
\end{equation}
Recalling \eqref{4ef3} we have that $J(f) \cap B(f^n(z), R_2) \neq \emptyset$. Hence, there exists some $x \in f_z^{-n}(J(f)) \cap f_z^{-n}(B(f^n(z),R_2))$. Since \eqref{fact2eq2} implies that $f_z^{-n} (J(f)) \subset J(f)$ we deduce that 
\begin{equation}
\label{jfinters}
J(f) \cap f_z^{-n}(B(f^{n}(z),R_2)) \neq \emptyset.
\end{equation}
Combing \eqref{kdel} and \eqref{jfinters} we deduce that 
$$J(f) \cap B(z, \kappa r) \neq \emptyset.$$

The proof of \ref{elpor} is essentially the same as the proof of Theorem~\ref{thmporfix} (i).
The proof of Theorem~\ref{t120190621} is complete.
\end{proof}

\begin{rem}\label{r320190621}
Theorem  \ref{t120190618} implies that if $f$ is a tame elliptic function such that $J(f)\ne\C$ and $\mu$ is a Borel probability $f$-invariant ergodic measure on $J(f)$ with full topological support and with finite Lyapunov exponent, then there exist some $\a_{f}, p_f \in (0,1]$ such that $J(f)$ is $(\a_{f},p_f)$-mean porous at $\mu$-a.e. $x\in J(f)$. As in Remark~\ref{r220190621}, we note that tameness of the elliptic function $f$ is in fact not needed if the Lyapunov exponent $\chi_{\mu}(f)$ is positive. 
\end{rem}


\begin{bibdiv}
\begin{biblist}
 \bib{smibel}{article}{
      AUTHOR = {Beliaev, D.},
      Author={Smirnov, S.},
     TITLE = {On dimension of porous measures},
   JOURNAL = {Math. Ann.},
  FJOURNAL = {Mathematische Annalen},
    VOLUME = {323},
      YEAR = {2002},
    NUMBER = {1},
     PAGES = {123--141},
}

 \bib{smibeletal}{article}{
    AUTHOR = {Beliaev, D.},
    Author= {J\"{a}rvenp\"{a}\"{a}, E.},
    Author={J\"{a}rvenp\"{a}\"{a}, M.} 
    Author={K\"{a}enm\"{a}ki, A.} 
    Author={Rajala, T.}, 
    Author={Smirnov, S.},
    Author={Suomala, V.},
     TITLE = {Packing dimension of mean porous measures},
   JOURNAL = {J. Lond. Math. Soc. (2)},
  FJOURNAL = {Journal of the London Mathematical Society. Second Series},
    VOLUME = {80},
      YEAR = {2009},
    NUMBER = {2},
     PAGES = {514--530},
      ISSN = {0024-6107},
   MRCLASS = {28A75 (28A80)},
  MRNUMBER = {2545265},
MRREVIEWER = {Zhi-Ying Wen},
       URL = {https://doi.org/10.1112/jlms/jdp040},
}


 \bib{Bergweiler}{article}{
	Author = {Bergweiler, W.},
	Journal = {Bull. Amer. Math. Soc.},
	Pages = {151--188},
	Title = {Iteration of meromorphic functions},
	Volume = {29},
	Year = {1993}}
	

 \bib{bruna}{article}{
      author={{Bruna}, J.},
     TITLE = {Boundary interpolation sets for holomorphic functions smooth to the boundary and BMO},
   JOURNAL = {Trans. Amer. Math. Soc.},
    VOLUME = {264},
      YEAR = {1981},
     PAGES = {393--409},
  }


 \bib{dirpor}{article}{
      author={{Chousionis}, V.},
     TITLE = {Directed porosity on conformal iterated function systems and
              weak convergence of singular integrals},
   JOURNAL = {Ann. Acad. Sci. Fenn. Math.},
  FJOURNAL = {Annales Academi\ae  Scientiarum Fennic\ae . Mathematica},
    VOLUME = {34},
      YEAR = {2009},
    NUMBER = {1},
     PAGES = {215--232},
      ISSN = {1239-629X},
   MRCLASS = {28A80 (42B20)},
  MRNUMBER = {2489023},
MRREVIEWER = {Dmitri B. Beliaev},
}
	
	   \bib{CLU}{article}{
      author={{Chousionis}, V.},
      author={{Leykekhman}, D.},
      author={{Urba{\'n}ski}, M.},
       title={{The dimension spectrum of graph directed Markov systems}},
     journal={Selecta Math, to appear},
}
	
       \bib{CTU}{article}{
 AUTHOR = {Chousionis, V.},
   AUTHOR = {Tyson, J. T.},
      AUTHOR = {Urba\'nski, M.}
       TITLE = {Conformal graph directed Markov systems on Carnot groups},
   JOURNAL = {Mem. Amer. Math. Soc. to appear}
     }
     
   
 \bib{Conway}{book}{
    AUTHOR = {Conway, J. B..},
     TITLE = {Functions of One Complex Variable I},
 PUBLISHER = {Springer},
      YEAR = {1978},
     PAGES = {xiv+316},
}  
      
\bib{De}{article}{
    AUTHOR = {Denjoy, A.},
     TITLE = {Sur une propri\' et\' e des s\' eries trigonom\' etriques},
   JOURNAL = {Verlag v.d.G.V. der Wis-en Natuur},
      YEAR = {Afd. 1920}
      }
     
\bib{Dob11}{article}{
   AUTHOR = {Dobbs, N.},
    TITLE = {Nice sets and invariant densities in complex dynamics},
 JOURNAL = {Ergod. Th. \& Dynam. Sys.},  
  Volume = {211},
	Pages = {73--86},
	 Year = {1992}}
      
      \bib{do}{article}{
    AUTHOR = {Dol\v{z}enko, E.},
     TITLE = {Boundary properties of arbitrary functions (in Russian)},
   JOURNAL = {Izv. Akad. Nauk SSSR Ser. Mat.},
   Pages = {3--14},
       NUMBER = {31},
      YEAR = {1967}
      }
      
           \bib{conrad}{unpublished}{
	Author = {Conrad, K.},
	Note = {Available at \url{http://www.math.uconn.edu/~kconrad/blurbs/ugradnumthy/Zinotes.pdf}},
	Title = {The Gaussian Integers },
	Year = {(2016)}}

\bib{ejj}{article}{
   Author= {Eckmann, J. P.},
    Author= {J\"{a}rvenp\"{a}\"{a}, E.},
      Author= {J\"{a}rvenp\"{a}\"{a}, E.},
     TITLE = {Porosities and dimensions of measures},
   JOURNAL = {Nonlinearity},
    VOLUME = {13},
      YEAR = {2000},
     PAGES = {1--18},
      ISSN = {1078-0947},
}

     \bib{edm}{article}{
		Author = {Edgar, G. A.},
		Author={Mauldin, D.}	
	Title = {Multifractal decompositions of digraph recursive fractals},
	Journal = {Proc. Lond. Math. Soc.},
	Number = {65},
	Pages = {604--628},
	Volume = {3},
	Year = {1992}
	}
     
  
          \bib{nuss}{article}{
	Author = {Falk, R. S.},
  AUTHOR = {Nussbaum, R. D.},	
     TITLE = {A {N}ew {A}pproach to {N}umerical {C}omputation of {H}ausdorff
              {D}imension of {I}terated {F}unction {S}ystems: {A}pplications
              to {C}omplex {C}ontinued {F}ractions},
   JOURNAL = {Integral Equations Operator Theory},
  FJOURNAL = {Integral Equations and Operator Theory},
    VOLUME = {90},
      YEAR = {2018},
    NUMBER = {5},
     PAGES = {90--61},
      ISSN = {0378-620X},
   MRCLASS = {11K55 (37C30 65D05)},
  MRNUMBER = {3851775},
}	

	\bib{moat}{article}{
 Author= {Gethner, E.},
      Author= {Wagon, S.},
        Author= {Wick, B.},
     TITLE = {A stroll through the Gaussian primes},
   JOURNAL = { American Mathematical Monthly},
    VOLUME = {105},
      YEAR = {1998},

}

	\bib{marpor}{article}{
 
    Author= {Granlund, S.},
      Author= {Linqvist, P.},
        Author= {Martio, O.},
     TITLE = {$F$-harmonic measure in space},
   JOURNAL = {Ann. Acad. Sci. Fenn. Math.},
    VOLUME = {7},
      YEAR = {1982},
     PAGES = {233--247},
}
     

 \bib{henbook}{book}{
    AUTHOR = {Hensley, D.},
     TITLE = {Continued fractions},
 PUBLISHER = {World Scientific Publishing Co. Pte. Ltd., Hackensack, NJ},
      YEAR = {2006},
     PAGES = {xiv+245},
      ISBN = {981-256-477-2},
   MRCLASS = {11A55 (11J70 28D05 30B70 37A45 37F10 40A15)},
  MRNUMBER = {2351741},
MRREVIEWER = {Oto Strauch},
       URL = {https://doi.org/10.1142/9789812774682},
}
     
     \bib{hecke}{book}{
    AUTHOR = {Hlawka, E.}, 
    AUTHOR = {Schoissengeier, J.},
    AUTHOR = {Taschner, R.},
     TITLE = {Geometric and analytic number theory},
    SERIES = {Universitext},
      NOTE = {Translated from the 1986 German edition by Charles Thomas},
 PUBLISHER = {Springer-Verlag, Berlin},
      YEAR = {1991},
     PAGES = {x+238},
      ISBN = {3-540-52016-3},
   MRCLASS = {11-01 (11Hxx 11Kxx 11Nxx)},
  MRNUMBER = {1123023},
       DOI = {10.1007/978-3-642-75306-0},
       URL = {https://doi.org/10.1007/978-3-642-75306-0},
       }
       

\bib{HK}{article}{
	AUTHOR = {Hawkins, J.},
  AUTHOR={Koss, L.},
	Journal = { Monatsh. Math.},
	Number = {137},
	Pages = {273--301},
	Title = {Parametrized dynamics of the Weiestrass elliptic function},
	Year = {2002}}       

      \bib{ahur}{article}{
		Author = {Hurwitz, A.},
     TITLE = {\"Uber die Entwicklung complexer Gr\"ossen in Kettenbr\"uche},
   JOURNAL = {Acta Math.},
    VOLUME = {XI},
      YEAR = {1888},
     PAGES = {187--200},

}

        \bib{jhur}{article}{
		Author = {Hurwitz, J.},
     TITLE = {\"Uber eine besondere Art der Kettenbruch- Entwicklung complexer Gr\"ossen,},
   JOURNAL = {Dissertation, University of Halle},
      YEAR = {1895},
}

\bib{jjm}{article}{
    Author= {J\"{a}rvenp\"{a}\"{a}, E.},
      Author= {J\"{a}rvenp\"{a}\"{a}, E.},
    Author={Mauldin, D.} 
     TITLE = {Deterministic and random aspects of porosities},
   JOURNAL = {Discrete Contin. Dyn. Syst.},
  FJOURNAL = {Discrete and Continuous Dynamical Systems. Series A},
    VOLUME = {8},
      YEAR = {2002},
    NUMBER = {1},
     PAGES = {121--136},
      ISSN = {1078-0947},
   MRCLASS = {28A80 (28A12)},
  MRNUMBER = {1877831},
MRREVIEWER = {Robert W. Vallin}
}

\bib{kasu1}{article}{
    Author= {K\"{a}enm\" aki, A.},
      Author= {Suomala, V.},
     TITLE = {Conical upper density theorems and porosity of measures},
   JOURNAL = {Adv. Math. },
    VOLUME = {217},
      YEAR = {2008},
     PAGES = {952--966},
}

\bib{kasu2}{article}{
    Author= {K\"{a}enm\" aki, A.},
      Author= {Suomala, V.},
     TITLE = {Nonsymmetric conical upper density and $k$-porosity},
   JOURNAL = {Adv. Math. },
    VOLUME = {363},
      NUMBER = {3},
      YEAR = {2011},
     PAGES = {1183--1195},
}

	 \bib{kes1}{article}{
	 AUTHOR = {Kesseb\"ohmer, M.}
	 Author = {Kombrink, S.},
     TITLE = {Minkowski content and fractal {E}uler characteristic for
              conformal graph directed systems},
   JOURNAL = {J. Fractal Geom.},
  FJOURNAL = {Journal of Fractal Geometry. Mathematics of Fractals and
              Related Topics},
    VOLUME = {2},
      YEAR = {2015},
    NUMBER = {2},
     PAGES = {171--227},
      ISSN = {2308-1309},
   MRCLASS = {28A80 (28A75 60K05)},
  MRNUMBER = {3353091},
MRREVIEWER = {Steffen Winter},
       URL = {https://doi.org/10.4171/JFG/19},
}

 \bib{kes2}{article}{
	 AUTHOR = {Kesseb\"ohmer, M.}
	 Author = {Kombrink, S.},
    TITLE = {Minkowski measurability of infinite conformal graph directed systems and application to Apollonian packings}
	Journal = {ArXiv  1702.02854},
	}

      

\bib{kore}{article}{
    AUTHOR = {Koskela, P.},
    Author={Rohde, S.},
     TITLE = {Hausdorff dimension and mean porosity},
   JOURNAL = {Math. Ann.},
  FJOURNAL = {Mathematische Annalen},
    VOLUME = {309},
      YEAR = {1997},
    NUMBER = {4},
     PAGES = {593--609},
      ISSN = {0025-5831},
   MRCLASS = {28A75 (28A78 30C65)},
  MRNUMBER = {1483825},
MRREVIEWER = {Wayne Stewart Smith},}


\bib{KU2}{article}{
	AUTHOR = {Kotus, J.},
  AUTHOR={Urba\'ski, U.},
	Journal = {Bull. London Math. Soc. },
	Number = {35},
	Pages = {269--275},
	Title = {Hausdorff dimension and Hausdorff measures of
Julia sets of elliptic functions},
	Year = {2003}}

\bib{KU3}{article}{
	AUTHOR = {Kotus, J.},
  AUTHOR={Urba\'ski, U.},
	Journal = {J. d'Analyse Math. },
	Number = {93},
	Pages = {35--102},
	Title = {Geometry and ergodic theory of non-recurrent
elliptic functions},
	Year = {2004}}

\bib{KU1}{book}{
    AUTHOR = {Kotus, J.},
     AUTHOR={Urba\'ski, U.},
     TITLE = {Ergodic Theory, Geometric Measure Theory, Conformal Measures and the Dynamics of Elliptic Functions},
 PUBLISHER = {In Preparation},
  
 PAGES = {~630},}
\bib{lpt}{book}{
    AUTHOR = {Lindenstrass, J.},
    AUTHOR={Preiss, D.},
    AUTHOR={Ti\v{s}er, J.}
     TITLE = {Fr\'	echet differentiability of Lipschitz maps and porous sets in Banach spaces},
 PUBLISHER = {Princeton University Press},
      YEAR = {2012},
 PAGES = {x+425},
}

\bib{mavu}{article}{
	Author = {Martio, O.},
  AUTHOR = {Vuorinen, M.},
	Journal = {Exposition. Math.},
	Number = {5},
	Pages = {17--40},
	Title = {Whitney cubes, p-capacity, and Minkowski content},
	Year = {1987}}

\bib{mat}{article}{
	Author = {Mattila, P.},
	Journal = {J. London Math. Soc.},
	Number = {38},
	Pages = {125--132},
	Title = {Distribution of sets and measures along planes},
	Year = {1988}}

     \bib{MU1}{article}{
	Author = {Mauldin, D.},
  AUTHOR = {Urba\'nski, M.},
	Journal = {Proc.\ London Math.\ Soc.\ (3)},
	Number = {1},
	Pages = {105--154},
	Title = {Dimensions and measures in infinite iterated function systems},
	Volume = {73},
	Year = {1996}}
	
	\bib{MUGDMS}{book}{
	Address = {Cambridge},
	Author = {Mauldin, D.},
  AUTHOR = {Urba\'nski, M.},
	Publisher = {Cambridge University Press},
	Series = {Cambridge Tracts in Mathematics},
	Title = {Graph directed {M}arkov systems, Geometry and dynamics of limit sets},
	Volume = {148},
	Year = {2003}}
	
	\bib{MW}{article}{
	Author = {Mauldin, D.},
  AUTHOR = {Williams, S. C.},
	Journal = {Trans. Amer. Math. Soc.},
	Pages = {811--829},
	Title = {Hausdorff dimension in graph directed constructions},
	Volume = {309},
	Year = {1988}}
	
\bib{MayerU1}{article}{
	Author = {Mayer, V.},
  AUTHOR = {Urba\'nski, M.},
	Journal = {Ergod. Th. \& Dynam. Sys.},
	Pages = {915--946},
	Title = {Geometric Thermodynamical Formalism and 
Real Analyticity for Meromorphic Functions of Finite Order},
	Volume = {28},
	Year = {2008}}	
	
\bib{MayerU2}{article}{
	Author = {Mayer, V.},
  AUTHOR = {Urba\'nski, M.},
	Journal = {Memoirs of AMS},
	Title = {Thermodynamical Formalism and Multifractal Analysis for Meromorphic Functions of Finite Order},
	Volume = {203},
	Year = {2010}}
	

\bib{MayerU3}{article}{
	Author = {Mayer, V.},
  AUTHOR = {Urba\'nski, M.},
	Journal = {Math. Zeit.},
	Pages = {657--683},
	Title = {Gibbs and equilibrium measures for elliptic functions},
	Volume = {250},
	Year = {2005}}		
	
	
\bib{niem}{article}{
    AUTHOR = {Nieminen, T.},
     TITLE = {Generalized mean porosity and dimension},
   JOURNAL = {Ann. Acad. Sci. Fenn. Math.},
  FJOURNAL = {Annales Academi\ae  Scientiarum Fennic\ae . Mathematica},
    VOLUME = {31},
      YEAR = {2006},
    NUMBER = {1},
     PAGES = {143--172},
      ISSN = {1239-629X},
   MRCLASS = {28A75 (28A80 30C65)},
  MRNUMBER = {2210114},
MRREVIEWER = {Piotr Haj\l asz},
}

 \bib{polur}{article}{
	Author = {Pollicott, M.},
  AUTHOR = {Urba\'nski, M.},
	 JOURNAL = {Mem. Amer. Math. Soc. to appear},
	Title = {Asymptotic Counting in Conformal Dynamical Systems},
	}
	
	\bib{preissza}{article}{
    AUTHOR = {Preiss, D.},
    AUTHOR = {Zaj\'{\i}\v{c}ek, L.},
     TITLE = {Sigma-porous sets in products of metric spaces and
              sigma-directionally porous sets in {B}anach spaces},
   JOURNAL = {Real Anal. Exchange},
  FJOURNAL = {Real Analysis Exchange},
    VOLUME = {24},
      YEAR = {1998/99},
    NUMBER = {1},
     PAGES = {295--313},
      ISSN = {0147-1937},
   MRCLASS = {28A05 (26B05 46G99)},
  MRNUMBER = {1691753},
MRREVIEWER = {\'{E}tienne Matheron},
}

	\bib{preissza2}{article}{
    AUTHOR = {Preiss, D.},
    AUTHOR = {Zaj\'{\i}\v{c}ek, L.},
     TITLE = {Directional derivatives of {L}ipschitz functions},
   JOURNAL = {Israel J. Math.},
  FJOURNAL = {Israel Journal of Mathematics},
    VOLUME = {125},
      YEAR = {2001},
     PAGES = {1--27},
      ISSN = {0021-2172},
   MRCLASS = {46G05 (46T20 49J52)},
  MRNUMBER = {1853802},
MRREVIEWER = {Warren B. Moors},
       URL = {https://doi.org/10.1007/BF02773371},
}

 \bib{pri}{article}{
		Author = {Priyadarshi, A.},
	Title = {Lower bound on the Hausdorff dimension of a set of complex continued fractions},
	Journal = {J. Math. Anal. Appl.},
	Volume = {449},
	Number = {1},
	Pages = {91--95},
	Year = {2017}}
	
\bib{PR1}{article}{
   AUTHOR = {Przytycki, F.},
   AUTHOR = {Rivera-Letelier, J.},
     TITLE = {Statistical properties of topological Collet-Eckmann maps},
  JOURNAL = {Ann. Sci. \'Ecole Norm. Sup.},
    VOLUME = {40},
      YEAR = {2007},
     PAGES = {135--178},
}


\bib{prro}{article}{
   AUTHOR = {Przytycki, F.},
   AUTHOR = {Rohde, S.},
     TITLE = {Porosity of {C}ollet-{E}ckmann {J}ulia sets},
  JOURNAL = {Fundamenta Mathematicae},
    VOLUME = {155},
      YEAR = {1998},
    NUMBER = {2},
     PAGES = {189--199},
      ISSN = {0016-2736},
   MRCLASS = {37F50 (28A78 28A80 30D05 37F15)},
  MRNUMBER = {1606527},
}
	
\bib{PU_tame}{article}{
   AUTHOR = {Przytycki, F.},
   AUTHOR = {Urba\'nski, M.},
    TITLE = {Rigidity of tame rational functions},
 JOURNAL = {Bull. Pol. Acad. Sci., Math.},  
  Volume = {47},
	Number = {2},
	Pages = {163--182},
	 Year = {1999}}
	 

 \bib{PUbook}{book}{
    AUTHOR = {Przytycki, F.},
    AUTHOR = {Urba\'nski, M.},
     TITLE = {Conformal Fractals — Ergodic Theory Methods},
 PUBLISHER = {Cambridge University Press},
      YEAR = {2017},
	 }

\bib{Riv07}{article}{
   AUTHOR = {Rivera-Letelier, J.},
    TITLE = {A connecting lemma for rational maps satisfying a no--growth
             condition},
 JOURNAL = {Ergod. Th. \& Dynam. Sys.},  
  Volume = {27},
	Pages = {595--636},
	 Year = {2007}}
	 
	
\bib{roy}{article}{
	  AUTHOR = {Roy, M.},
     TITLE = {A new variation of {B}owen's formula for graph directed
              {M}arkov systems},
   JOURNAL = {Discrete Contin. Dyn. Syst.},
  FJOURNAL = {Discrete and Continuous Dynamical Systems. Series A},
    VOLUME = {32},
      YEAR = {2012},
    NUMBER = {7},
     PAGES = {2533--2551},
      ISSN = {1078-0947},
   MRCLASS = {37D35},
  MRNUMBER = {2900559},
MRREVIEWER = {Heber Enrich},
       URL = {https://doi.org/10.3934/dcds.2012.32.2533},
}

	 \bib{salli}{article}{
	  AUTHOR = {Salli, A.},
     TITLE = {On the Minkowski dimension of strongly porous fractal sets in $\Rn$},
   JOURNAL = {Proc. London Math. Soc.},
    VOLUME = {62},
      YEAR = {1991},
    NUMBER = {7},
     PAGES = {353--372},
}

	
	\bib{shmemp}{article}{
    AUTHOR = {Shmerkin, P.},
     TITLE = {The dimension of weakly mean porous measures: a probabilistic
              approach},
   JOURNAL = {Int. Math. Res. Not. IMRN},
  FJOURNAL = {International Mathematics Research Notices. IMRN},
      YEAR = {2012},
    NUMBER = {9},
     PAGES = {2010--2033},
      ISSN = {1073-7928},
   MRCLASS = {60G46 (28A12 60G42)},
  MRNUMBER = {2920822},
MRREVIEWER = {Ren\'{e} L. Schilling},
       DOI = {10.1093/imrn/rnr085},
       URL = {https://doi.org/10.1093/imrn/rnr085},}
       
       
 \bib{SkU}{article}{
	  AUTHOR = {Skorulski, B.},
	  AUTHOR = {Urba\'nski, M.},
     TITLE = {Finer Fractal Geometry for Analytic Families of Conformal 
              Dynamical Systems},
   JOURNAL = {Dynamical Systems},
    VOLUME = {29},
      YEAR = {2014},
     PAGES = {369--398},}

\bib{speight}{book}{
    AUTHOR = {Speight, G.},
     TITLE = {Porosity and differentiability},
      NOTE = {Thesis (Ph.D.)--University of Warwick (United Kingdom)},
 PUBLISHER = {ProQuest LLC, Ann Arbor, MI},
      YEAR = {2013},
     PAGES = {1},
   MRCLASS = {Thesis},
  MRNUMBER = {3389367},
       URL =
              {http://gateway.proquest.com/openurl?url_ver=Z39.88-2004&rft_val_fmt=info:ofi/fmt:kev:mtx:dissertation&res_dat=xri:pqm&rft_dat=xri:pqdiss:U636455},
}
	
	\bib{stillwell}{book}{
    AUTHOR = {Stillwell, J.},
     TITLE = {Elements of number theory},
    SERIES = {Undergraduate Texts in Mathematics},
 PUBLISHER = {Springer-Verlag, New York},
      YEAR = {2003},
     PAGES = {xii+254},
      ISBN = {0-387-95587-9},
   MRCLASS = {11-01 (11D09)},
  MRNUMBER = {1944957},
MRREVIEWER = {Gwynneth G. H. Coogan},
       DOI = {10.1007/978-0-387-21735-2},
       URL = {https://doi.org/10.1007/978-0-387-21735-2},
}
	
	\bib{vai}{article}{
  AUTHOR = {V\" ais\" al\" a, J.},
	Journal = {Trans. Amer. Math. Soc.},
	Number = {299},
	Pages = {525--533},
	Title = {Porous Sets and Quasisymmetric Maps},
	Year = {1987}}
	
	
	\bib{urbpor}{article}{
  AUTHOR = {Urba\'nski, M.},
	Journal = {J. Number Th.},
	Number = {2},
	Pages = {283--312},
	Title = {Porosity in Conformal Infinite Iterated Function Systems},
	Volume = {88},
	Year = {2001}}
	
	\bib{za}{article}{
  AUTHOR = {Zaj\' i\v{c}ek , L.},
	Journal = {Real Anal. Exchange},
	Number = {13},
	Pages = {314--350},
	Title = {Porosity and $\sigma$-porosity},
	Year = {1987/88}}
     
     \end{biblist}
\end{bibdiv}
\end{document}